\documentclass[10pt,reqno]{amsart}

\usepackage{mathrsfs}
\usepackage{amsmath,amsthm, amscd}
\usepackage[psamsfonts]{amssymb}
\usepackage[all]{xy}
\usepackage{wrapfig}

\usepackage{tikz}
\usepackage{tikz-cd}
\usepackage{booktabs}
\usepackage{pinlabel}
\usepackage{ytableau}
\usepackage{braket}

\usetikzlibrary{arrows.meta, angles, quotes, calc, decorations.pathreplacing,calligraphy,intersections}
\usetikzlibrary{positioning, decorations.markings,patterns}
\usepackage{comment}
\usepackage{tikz-3dplot}
\usepackage[normalem]{ulem} % ulemは通常\emphを波線に変えてしまうため、[normalem]をつけて元の強調(\emph)を維持しつつ下線機能だけ使うのが安全です
\usepackage[absolute,overlay]{textpos}
\usepackage[legacycolonsymbols]{mathtools}
\usepackage{subcaption}
\usepackage{thmtools} % 定理環境をA,B,Cにするため
\usepackage{extpfeil} %xtwoheadrightarrowを使うため

\usepackage[backend=biber, style=alphabetic]{biblatex}
\usepackage{config}
\title[Rational homotopy of long embeddings]{Infinite-dimensionality of the rational homotopy groups of the space of long embeddings of codimension 2}
\author{Daiki Irikura}
\address{Department of Mathematics, Kyoto University, Kyoto 606-8502, Japan}
\email{irikura.daiki.45m(at)st.kyoto-u.ac.jp}
\date{\today}
\newsavebox{\equationchordsII}
\newsavebox{\MyRibbonTypeI}
\newsavebox{\MyRibbonTypeII}
\newsavebox{\MyRibbonTypeIPerturbed}

\definecolor{ribbonfill}{gray}{0.85} % 薄いグレー
\begin{document}

%inputfiles==========================

% =============================================
% 1. 箱の名前を定義
% =============================================
% \newsavebox{\equationchordsII}

% =============================================
% 2. 図の中身をボックスに保存（等倍で作成）
% =============================================
\savebox{\equationchordsII}{%
% =============================================
% 整数 k の設定 (k >= 4 推奨)
\def\numK{4} 
% =============================================

\begin{tikzpicture}[
    % --- スタイル定義 (維持) ---
    vertex/.style={circle, fill=black, inner sep=2pt},
    Wvertex/.style={circle, draw=black, thick, fill=white, inner sep=2pt},
    mid arrow/.style={postaction={decorate,decoration={
        markings,
        mark=at position 0.55 with {\arrow{Latex[length=2.5mm, width=1.8mm]}}
    }}},
    solid edge/.style={thick, draw=black},
    dashed edge/.style={thick, dashed, draw=black}
]

    % 図全体の縮尺
    \begin{scope}[scale=0.8]

    % =========================================================
    % 1. 上段左 (LHS): 全て破線, p0のみ黒
    % =========================================================
    \begin{scope}[shift={(0, 0)}]
        % 座標定義
        \coordinate (p0) at (-1, 1.5);
        \coordinate (p1) at (-1, 0);
        \draw[dashed edge, mid arrow] (p0) -- (p1);

        % p0 の上の縦線
        % \draw[dashed edge] (p1) -- ++(0, 1.5);

        % ループ: p1 -> ... -> pk
        \foreach \i in {2,...,\numK} {
            \pgfmathsetmacro{\x}{-\i}
            \coordinate (p\i) at (\x, 0);
            
            % 前の点からのエッジ (p_{i-1} <- p_i)
            \pgfmathsetmacro{\prev}{int(\i-1)}
            \draw[dashed edge, mid arrow] (p\i) -- (p\prev);
            
            % 縦の破線
            \draw[dashed edge] (p\i) -- ++(0, 1.5);
        }
        % pk から左への破線
        \draw[dashed edge] (p\numK) -- ++(-1, 0);

         % p0 から右への破線 (外へ向かう線)
        \draw[dashed edge] (p1) -- ++(1,0);

        % ノード描画
        \node[vertex, label=above:\Large $p_0$] at (p0) {};
        \foreach \i in {1,...,\numK} {
            \node[Wvertex, label=below:\Large $p_\i$] at (p\i) {};
        }
    \end{scope}

    % =========================================================
    % 等号 (=)
    % =========================================================
    \node at (1.5, 0.75) {\Huge $=$};

    % =========================================================
    % 2. 上段右 (RHS1): p1->p0 実線, p0,p1黒
    % =========================================================
    \begin{scope}[shift={(\numK + 3.5, 0)}]
        \coordinate (p0) at (0, 0);
        
        % 右への線
        \draw[dashed edge] (p0) -- ++(1, 0);
        % 縦線
        \draw[dashed edge] (p0) -- ++(0, 1.5);

        % p1 (位置は -1,0)
        \coordinate (p1) at (-1, 0);
        % p1->p0 は実線 (solid edge)
        \draw[solid edge, mid arrow] (p1) -- (p0);
        \draw[dashed edge] (p1) -- ++(0, 1.5);

        % p2 -> ... -> pk (破線)
        \foreach \i in {2,...,\numK} {
            \pgfmathsetmacro{\x}{-\i}
            \coordinate (p\i) at (\x, 0);
            \pgfmathsetmacro{\prev}{int(\i-1)}
            \draw[dashed edge, mid arrow] (p\i) -- (p\prev);
            \draw[dashed edge] (p\i) -- ++(0, 1.5);
        }
        \draw[dashed edge] (p\numK) -- ++(-1, 0);

        % ノード: p0, p1 は黒(vertex), それ以外は白(Wvertex)
        \node[vertex, label=below:\Large $p_0$] at (p0) {};
        \node[vertex, label=below:\Large $p_1$] at (p1) {};
        \foreach \i in {2,...,\numK} {
            \node[Wvertex, label=below:\Large $p_\i$] at (p\i) {};
        }
    \end{scope}

    % =========================================================
    % 2行目への移動 (等号)
    % =========================================================
    \node at (1.5, -3.0) {\Huge $=$};

    % =========================================================
    % 3. 下段左 (RHS2): p2->p1->p0 実線, p0,p1,p2黒
    % =========================================================
    \begin{scope}[shift={(\numK + 3.5, -3.75)}]
        \coordinate (p0) at (0, 0);
        \draw[dashed edge] (p0) -- ++(1, 0);
        \draw[dashed edge] (p0) -- ++(0, 1.5);

        \coordinate (p1) at (-1, 0);
        \draw[solid edge, mid arrow] (p1) -- (p0);
        \draw[dashed edge] (p1) -- ++(0, 1.5);

        \coordinate (p2) at (-2, 0);
        \draw[solid edge, mid arrow] (p2) -- (p1); % p2->p1 実線
        % \draw[dashed edge] (p2) -- ++(0, 1.5);

        % p3 -> ... -> pk
        \foreach \i in {3,...,\numK} {
            \pgfmathsetmacro{\x}{-\i}
            \coordinate (p\i) at (\x, 0);
            \pgfmathsetmacro{\prev}{int(\i-1)}
            \draw[dashed edge, mid arrow] (p\i) -- (p\prev);
            \draw[dashed edge] (p\i) -- ++(0, 1.5);
        }
        \draw[dashed edge] (p\numK) -- ++(-1, 0);

        % ノード: p0, p1, p2 黒
        \node[vertex, label=below:\Large $p_0$] at (p0) {};
        \node[vertex, label=below:\Large $p_1$] at (p1) {};
        \node[vertex, label=below:\Large $p_2$] at (p2) {};
        
        \foreach \i in {3,...,\numK} {
            \node[Wvertex, label=below:\Large $p_\i$] at (p\i) {};
        }
    \end{scope}

    % =========================================================
    % プラス (+)
    % =========================================================
    % 左図の幅を考慮して配置 (k=4の場合、幅は約5)
    \node at (\numK + 5.0, -3.0) {\Huge $+$};

    % =========================================================
    % 4. 下段右 (RHS3): p1->p0実線, p2->p0実線(曲線), p0,p1,p2黒
    % =========================================================
    \begin{scope}[shift={(2*\numK + 7.5, -3.75)}]
        \coordinate (p0) at (0, 0);
        \draw[dashed edge] (p0) -- ++(1, 0);
        \draw[dashed edge] (p0) -- ++(0, 1.5);

        \coordinate (p1) at (-1, 0);
        \draw[dashed edge] (p1) -- ++(0, 1.5);

        \coordinate (p2) at (-2, 0);
        \draw[solid edge, mid arrow] (p2) -- (p1);

        % p2 -> p0 への曲線 (p1をスキップ)
        \draw[solid edge, mid arrow] (p2) to [out=-60, in=-120, looseness=0.8] (p0);
        % \draw[dashed edge] (p2) -- ++(0, 1.5);
        
        % p2 -> p1 の間の破線 (つながっていないことを明示する場合)
        % 元のコードには描画がなかったので省略、または破線でつなぐなら以下:
        % \draw[dashed edge, mid arrow] (p2) -- (p1);

        % p3 -> ... -> pk
        \foreach \i in {3,...,\numK} {
            \pgfmathsetmacro{\x}{-\i}
            \coordinate (p\i) at (\x, 0);
            \pgfmathsetmacro{\prev}{int(\i-1)}
            \draw[dashed edge, mid arrow] (p\i) -- (p\prev);
            \draw[dashed edge] (p\i) -- ++(0, 1.5);
        }
        \draw[dashed edge] (p\numK) -- ++(-1, 0);

        % ノード: p0, p1, p2 黒
        \node[vertex, label=below:\Large $p_0$] at (p0) {};
        \node[vertex, label=below:\Large $p_1$] at (p1) {};
        \node[vertex, label=below:\Large $p_2$] at (p2) {};
        
        \foreach \i in {3,...,\numK} {
            \node[Wvertex, label=below:\Large $p_\i$] at (p\i) {};
        }
    \end{scope}

    % =========================================================
    % 2行目への移動 (等号)
    % =========================================================
    \node at (1.5, -6.75) {\Huge $=$};

    % =========================================================
    % 3. 下段左 (RHS2): p2->p1->p0 実線, p0,p1,p2黒
    % =========================================================
    \begin{scope}[shift={(\numK + 3.5, -7.5)}]
        \coordinate (p0) at (0, 0);
        \draw[dashed edge] (p0) -- ++(1, 0);
        % \draw[dashed edge] (p0) -- ++(0, 1.5);

        \coordinate (p1) at (-1, 0);
        \draw[solid edge, mid arrow] (p1) -- (p0);
        \draw[dashed edge] (p1) -- ++(0, 1.5);

        \coordinate (p2) at (-2, 0);
        \draw[solid edge, mid arrow] (p2) -- (p1); % p2->p1 実線
        \draw[dashed edge] (p2) -- ++(0, 1.5);

        \coordinate (p3) at (-3, 0);
        \draw[solid edge, mid arrow] (p3) -- (p2); % p2->p1 実線
        % \draw[dashed edge] (p3) -- ++(0, 1.5);

        % \coordinate (p4) at (-4, 0);
        % \draw[solid edge, mid arrow] (p4) -- (p3); % p2->p1 実線
        % \draw[dashed edge] (p3) -- ++(0, 1.5);

        % p3 -> ... -> pk
        \foreach \i in {4,...,\numK} {
            \pgfmathsetmacro{\x}{-\i}
            \coordinate (p\i) at (\x, 0);
            \pgfmathsetmacro{\prev}{int(\i-1)}
            \draw[dashed edge, mid arrow] (p\i) -- (p\prev);
            \draw[dashed edge] (p\i) -- ++(0, 1.5);
        }
        \draw[dashed edge] (p\numK) -- ++(-1, 0);

        % ノード: p0, p1, p2 黒
        \node[vertex, label=below:\Large $p_0$] at (p0) {};
        \node[vertex, label=below:\Large $p_1$] at (p1) {};
        \node[vertex, label=below:\Large $p_2$] at (p2) {};
        \node[vertex, label=below:\Large $p_3$] at (p3) {};

        \foreach \i in {4,...,\numK} {
            \node[Wvertex, label=below:\Large $p_\i$] at (p\i) {};
        }
    \end{scope}

    % =========================================================
    % プラス (+)
    % =========================================================
    % 左図の幅を考慮して配置 (k=4の場合、幅は約5)
    \node at (\numK + 5.0, -6.75) {\Huge $+$};

    % =========================================================
    % 4. 下段右 (RHS3): p1->p0実線, p2->p0実線(曲線), p0,p1,p2黒
    % =========================================================
    \begin{scope}[shift={(2*\numK + 7.5, -7.5)}]
        \coordinate (p0) at (0, 0);
        \draw[dashed edge] (p0) -- ++(1, 0);
        \draw[dashed edge] (p0) -- ++(0, 1.5);

        \coordinate (p1) at (-1, 0);
        \draw[solid edge, mid arrow] (p1) -- (p0);
        \draw[dashed edge] (p1) -- ++(0, 1.5);
        
        \coordinate (p2) at (-2, 0);
        % \draw[solid edge, mid arrow] (p2) -- (p1); % p2->p1 実線
        \draw[dashed edge] (p2) -- ++(0, 1.5);

        \coordinate (p3) at (-3, 0);
        \draw[solid edge, mid arrow] (p3) -- (p2); % p2->p1 実線
        % \draw[dashed edge] (p3) -- ++(0, 1.5);

        %p2 -> p0 への曲線 (p1をスキップ)
        \draw[solid edge, mid arrow] (p3) to [out=-60, in=-120, looseness=0.8] (p1);
        % \draw[dashed edge] (p2) -- ++(0, 1.5);
        
        % p2 -> p1 の間の破線 (つながっていないことを明示する場合)
        % 元のコードには描画がなかったので省略、または破線でつなぐなら以下:
        % \draw[dashed edge, mid arrow] (p2) -- (p1);

        % p3 -> ... -> pk
        \foreach \i in {4,...,\numK} {
            \pgfmathsetmacro{\x}{-\i}
            \coordinate (p\i) at (\x, 0);
            \pgfmathsetmacro{\prev}{int(\i-1)}
            \draw[dashed edge, mid arrow] (p\i) -- (p\prev);
            \draw[dashed edge] (p\i) -- ++(0, 1.5);
        }
        \draw[dashed edge] (p\numK) -- ++(-1, 0);

        % ノード: p0, p1, p2 黒
        \node[vertex, label=below:\Large $p_0$] at (p0) {};
        \node[vertex, label=below:\Large $p_1$] at (p1) {};
        \node[vertex, label=below:\Large $p_2$] at (p2) {};
        \node[vertex, label=below:\Large $p_3$] at (p3) {};

        \foreach \i in {4,...,\numK} {
            \node[Wvertex, label=below:\Large $p_\i$] at (p\i) {};
        }
    \end{scope}
    % =========================================================
    % プラス (+)
    % =========================================================
    % 左図の幅を考慮して配置 (k=4の場合、幅は約5)
    \node at (2*\numK  + 10, -6.75) {\Huge $+$};

    % =========================================================
    % 4. 下段右 (RHS3): p1->p0実線, p2->p0実線(曲線), p0,p1,p2黒
    % =========================================================
    \begin{scope}[shift={(3*\numK + 12.5, -7.5)}]
       \coordinate (p0) at (0, 0);
        \draw[dashed edge] (p0) -- ++(1, 0);
        \draw[dashed edge] (p0) -- ++(0, 1.5);

        \coordinate (p1) at (-1, 0);
        \draw[dashed edge] (p1) -- ++(0, 1.5);
        
        \coordinate (p2) at (-2, 0);
        \draw[solid edge, mid arrow] (p2) -- (p1);
        \draw[dashed edge] (p2) -- ++(0, 1.5);

        \coordinate (p3) at (-3, 0);
        \draw[solid edge, mid arrow] (p3) -- (p2); % p2->p1 実線
        % \draw[dashed edge] (p3) -- ++(0, 1.5);

        %p2 -> p0 への曲線 (p1をスキップ)
        \draw[solid edge, mid arrow] (p3) to [out=-60, in=-120, looseness=0.8] (p0);
        % \draw[dashed edge] (p2) -- ++(0, 1.5);
        
        % p2 -> p1 の間の破線 (つながっていないことを明示する場合)
        % 元のコードには描画がなかったので省略、または破線でつなぐなら以下:
        % \draw[dashed edge, mid arrow] (p2) -- (p1);

        % p3 -> ... -> pk
        \foreach \i in {4,...,\numK} {
            \pgfmathsetmacro{\x}{-\i}
            \coordinate (p\i) at (\x, 0);
            \pgfmathsetmacro{\prev}{int(\i-1)}
            \draw[dashed edge, mid arrow] (p\i) -- (p\prev);
            \draw[dashed edge] (p\i) -- ++(0, 1.5);
        }
        \draw[dashed edge] (p\numK) -- ++(-1, 0);

        % ノード: p0, p1, p2 黒
        \node[vertex, label=below:\Large $p_0$] at (p0) {};
        \node[vertex, label=below:\Large $p_1$] at (p1) {};
        \node[vertex, label=below:\Large $p_2$] at (p2) {};
        \node[vertex, label=below:\Large $p_3$] at (p3) {};

        \foreach \i in {4,...,\numK} {
            \node[Wvertex, label=below:\Large $p_\i$] at (p\i) {};
        }
    \end{scope}
    
     % =========================================================
    % プラス (+)
    % =========================================================
    % 左図の幅を考慮して配置 (k=4の場合、幅は約5)
    \node at (3*\numK  + 15, -6.75) {\Huge $+$};

    % =========================================================
    % 4. 下段右 (RHS3): p1->p0実線, p2->p0実線(曲線), p0,p1,p2黒
    % =========================================================
    \begin{scope}[shift={(4*\numK + 17.5, -7.5)}]
       \coordinate (p0) at (0, 0);
        \draw[dashed edge] (p0) -- ++(1, 0);
        \draw[dashed edge] (p0) -- ++(0, 1.5);

        \coordinate (p1) at (-1, 0);
        % \draw[solid edge, mid arrow] (p1) -- (p0);
        \draw[dashed edge] (p1) -- ++(0, 1.5);
        
        \coordinate (p2) at (-2, 0);
        % \draw[solid edge, mid arrow] (p2) -- (p1); % p2->p1 実線
        \draw[dashed edge] (p2) -- ++(0, 1.5);

        \coordinate (p3) at (-3, 0);
        \draw[solid edge, mid arrow] (p3) -- (p2); % p2->p1 実線
        % \draw[dashed edge] (p3) -- ++(0, 1.5);

        %p2 -> p0 への曲線 (p1をスキップ)
        \draw[solid edge, mid arrow] (p2) to [out=-60, in=-120, looseness=0.8] (p0);
        \draw[solid edge, mid arrow] (p3) to [out=-60, in=-120, looseness=0.8] (p1);
        % \draw[dashed edge] (p2) -- ++(0, 1.5);
        
        % p2 -> p1 の間の破線 (つながっていないことを明示する場合)
        % 元のコードには描画がなかったので省略、または破線でつなぐなら以下:
        % \draw[dashed edge, mid arrow] (p2) -- (p1);

        % p3 -> ... -> pk
        \foreach \i in {4,...,\numK} {
            \pgfmathsetmacro{\x}{-\i}
            \coordinate (p\i) at (\x, 0);
            \pgfmathsetmacro{\prev}{int(\i-1)}
            \draw[dashed edge, mid arrow] (p\i) -- (p\prev);
            \draw[dashed edge] (p\i) -- ++(0, 1.5);
        }
        \draw[dashed edge] (p\numK) -- ++(-1, 0);

        % ノード: p0, p1, p2 黒
        \node[vertex, label=below:\Large $p_0$] at (p0) {};
        \node[vertex, label=below:\Large $p_1$] at (p1) {};
        \node[vertex, label=below:\Large $p_2$] at (p2) {};
        \node[vertex, label=below:\Large $p_3$] at (p3) {};

        \foreach \i in {4,...,\numK} {
            \node[Wvertex, label=below:\Large $p_\i$] at (p\i) {};
        }
    \end{scope}

    \end{scope} % End scale scope
\end{tikzpicture}
}

% =============================================
% 2. 図の中身をボックスに保存（等倍で作成）
% =============================================
\savebox{\MyRibbonTypeI}{%

%========本文=========================================================================

\begin{tikzpicture}[scale=1.0, z={(-0.3cm,-0.2cm)}, line join=round, line cap=round]

    % ==========================================
    % === 設定パラメータ ===
    % ==========================================
    \def\LineH{4.0}      % L_j の長さ (片側)
    \def\DiskR{0.3}      % D_{i:k} (円盤) の半径
    \def\SphereR{0.55}   % Q_r (球) の半径
    
    % --- 既存のリボン設定 ---
    \def\AngSt{-75}      % リボン1: 右側の接続角度
    \def\AngEn{-105}     % リボン1: 左側の接続角度

    % --- 新しいディスクとリボンの設定 ---
    \def\DiskTwoZ{1.4}      
    \def\RibTwoRootX{1.2}   % リボン2: 長方形側の右端のX座標
    \def\RibTwoWidth{0.15}  % リボン2の幅
    \def\RibTwoAngSt{-45}   % リボン2: ディスク側の開始角度 (右)
    \def\RibTwoAngEn{-85}   % リボン2: ディスク側の終了角度 (左)

    % --- 座標・ボックス設定 ---
    \def\DiskY{1.0}       % 円盤と直線の中心Y座標
    \def\RootY{-3.1}      % 下の長方形(D_{i:0})の上辺のY座標
    \def\BoxTopY{3.0}     % ボックスの上端Y座標
    
    % D_s のY座標
    \def\DsY{-2.5}
    
    % --- 内部変数の計算 ---
    \def\BoxW{1.6}            
    \def\BoxBot{\RootY - 0.3}
    \def\BoxZ{\LineH}         
    \def\DosZ{0.5*\BoxZ} % D_{0,s} のZ座標 (約2.0)

    % ==========================================
    % === 描画開始 ===
    % ==========================================

    % --- 1. Background (Box Back/Left) ---
    % 奥側の枠線（破線）
    \draw[dashed, gray!60] (\BoxW, \BoxTopY, -\BoxZ) -- (-\BoxW, \BoxTopY, -\BoxZ) -- (-\BoxW, \BoxBot, -\BoxZ) -- (\BoxW, \BoxBot, -\BoxZ);
    \draw[dashed, gray!60] (-\BoxW, \BoxBot, -\BoxZ) -- (-\BoxW, \BoxBot, \BoxZ);

    % --- 2. Line L_j (Bottom part) ---
    \draw[thick, blue] (0, \DiskY, -\LineH) -- (0, \DiskY, 0);

    % --- 3. Root Rectangle D_{0,p} & D_{0,s} ---
    % D_{0, p} (手前/左のボックス)
    \filldraw[fill=gray!20, draw=black] (-1.5, {\RootY - 0.3}, 0) rectangle (1.5, \RootY, 0);
    \node at (-1.8, {\RootY - 0.15}, 0) {\scriptsize $D_{0}$};

    % D_{0, s} (奥/右のボックス: Z方向にシフト)
    \filldraw[fill=gray!30, draw=black] (-1.5, {\RootY - 0.3}, \DosZ ) rectangle (1.5, \RootY, \DosZ);
    \node at (-1.8, {\RootY - 0.15}, \DosZ) {\scriptsize $D_{0}'$};

    % =======================================================
    % ★ D_s 描画エリア (奥半分 -> リボン -> 手前半分)
    % =======================================================
     % --- 4a. Sphere Q Equator (Back Part: z < 0) ---
    % % 赤いバンドより先に描くことで、バンドの後ろに隠れるようにする
    % \begin{scope}[canvas is xz plane at y=\DsY]
    %     % z < 0 は sin(theta) < 0 なので 180度〜360度
    %     \draw[dashed, orange!80!black] (180:\SphereR) arc (180:360:\SphereR);
    % \end{scope}

    % --- 4a. Disk D_s (奥側半分: z < 0) ---
    \begin{scope}[canvas is zx plane at y=\DsY]
        \fill[green!10, opacity=0.8] (0,0) -- (90:\DiskR) arc (90:270:\DiskR) -- cycle;
        \draw[green!60!black] (90:\DiskR) arc (90:270:\DiskR);
    \end{scope}

    % --- 4b. Band 1 (Center at z=0, from top Disk to D_s) ---
    \begin{scope}[canvas is xy plane at z=0]
        \filldraw[fill=red!3, draw=black, opacity=0.9] 
            ({ \DiskR * cos(\AngSt) }, \RootY)
            -- ({ \DiskR * cos(\AngSt) }, { \DiskY + \DiskR * sin(\AngSt) })
            arc (\AngSt:\AngEn:\DiskR)
            -- ({ \DiskR * cos(\AngEn) }, \RootY)
            -- cycle;
            
        % ★追加: ラベル B_+
        \node[red] at (-0.4, -1.7) {\scriptsize $B_+$};
    \end{scope}
    
    % --- 4c. Intersection Line (BandとD_sの交差線) ---
    \draw[green!60!black, ultra thick] 
        ({ \DiskR * cos(\AngEn) }, \DsY, 0) -- ({ \DiskR * cos(\AngSt) }, \DsY, 0);

    % =======================================================
    % ★追加: D_s から D_{0,s} への新しいリボン (修正版)
    % =======================================================
    % 計算: リボンの開始点(D_s上)
    \pgfmathsetmacro{\RibXsRight}{\DiskR * cos(\AngSt)}
    \pgfmathsetmacro{\RibXsLeft}{\DiskR * cos(\AngEn)}
    % 円の方程式 z = sqrt(r^2 - x^2) でz座標を計算 (D_sの円周上の点)
    \pgfmathsetmacro{\RibZsRight}{sqrt(\DiskR*\DiskR - \RibXsRight*\RibXsRight)}
    \pgfmathsetmacro{\RibZsLeft}{sqrt(\DiskR*\DiskR - \RibXsLeft*\RibXsLeft)}

    % リボンの描画: ベジェ曲線 (controls) を使用して滑らかに曲げる
    % Start(D_s) -> End(D_{0,s})
    % Yは -2.5 -> -3.0 へ下がる
    % Zは RibZs(約0) -> DosZ(2.0) へ奥へ進む
    
    \filldraw[fill=green!5, draw=black, opacity=0.8] 
        % 1. 右上の点 (D_s上)
        (\RibXsRight, \DsY, \RibZsRight) 
        % 2. 右側の曲線 (下へ、奥へ)
        .. controls (\RibXsRight, {\DsY - 0.3}, {\RibZsRight + 0.5}) and (\RibXsRight, {\RootY + 0.3}, {\DosZ - 0.5}) ..
        % 3. 右下の点 (D_{0,s}上)
        (\RibXsRight, \RootY, \DosZ)
        % 4. 下辺 (D_{0,s}上を左へ)
        -- (\RibXsLeft, \RootY, \DosZ)
        % 5. 左側の曲線 (上へ、手前へ)
        .. controls (\RibXsLeft, {\RootY + 0.3}, {\DosZ - 0.5}) and (\RibXsLeft, {\DsY - 0.3}, {\RibZsLeft + 0.5}) ..
        % 6. 左上の点 (D_s上)
        (\RibXsLeft, \DsY, \RibZsLeft)
        % 7. 閉じる
        -- cycle;

    % --- 4d. Disk D_s (手前半分: z > 0) ---
    \begin{scope}[canvas is zx plane at y=\DsY]
        \fill[green!10, opacity=0.6] (0,0) -- (-90:\DiskR) arc (-90:90:\DiskR) -- cycle;
        \draw[green!60!black] (-90:\DiskR) arc (-90:90:\DiskR);
        \node[green!60!black] at (0, {-\DiskR -0.2}) {\scriptsize $D'$}; % ラベル位置調整
    \end{scope}
    
    % =======================================================

    % --- 7. Disk D_+ (Top Red) ---
    \begin{scope}[canvas is xy plane at z=0]
        \filldraw[fill=red!10, draw=red, opacity=0.8] (0, \DiskY) circle (\DiskR);
        \node[red] at ({\DiskR + 0.2}, {\DiskY + 0.2}) {\scriptsize $D_{+}$};
        \fill[red] (0, \DiskY) circle (1.5pt);
    \end{scope}

    % --- 8. Middle Line (Inside Sphere) ---
    \draw[thick, blue] (0, \DiskY, 0) -- (0, \DiskY, \DiskTwoZ);

    % --- 9. Second Disk & Ribbon ---
    % Disk D_{new} (D_-)
    \begin{scope}[canvas is xy plane at z=\DiskTwoZ]
        \filldraw[fill=blue!10, draw=blue, opacity=0.9] (0, \DiskY) circle (\DiskR);
        \node[blue] at (0, {\DiskY}) [xshift=0.5cm, yshift=0.4cm] {\scriptsize $D_{-}$};
        \fill[blue] (0, \DiskY) circle (1.5pt);
    \end{scope}

  % Ribbon B_{new} (Twisted Ribbon with Arc Top)
    \filldraw[fill=blue!3, draw=black, opacity=0.9] 
        % 1. Start at Bottom Right (D_{0,p} side)
        (\RibTwoRootX, \RootY, 0)
        
        % 2. Curve up to Top Left (D_{-} side)
        .. controls (\RibTwoRootX, {\RootY+1.5}, -3.0) and ({ \DiskR * cos(\RibTwoAngEn) }, {\DiskY - 1.5}, {\DiskTwoZ - 1.0}) ..
        ({ \DiskR * cos(\RibTwoAngEn) }, { \DiskY + \DiskR * sin(\RibTwoAngEn) }, \DiskTwoZ)
        
        % 3. Arc at Top
        -- plot[domain=\RibTwoAngEn:\RibTwoAngSt, variable=\t, samples=25] 
        ({ \DiskR * cos(\t) }, { \DiskY + \DiskR * sin(\t) }, \DiskTwoZ)
        
        % 4. Curve down to Bottom Left (D_{0,p} side)
        .. controls ({ \DiskR * cos(\RibTwoAngSt) }, {\DiskY - 1.5}, 1.0) and ({\RibTwoRootX - \RibTwoWidth}, {\RootY+1.5}, 2.0) ..
        ({\RibTwoRootX - \RibTwoWidth}, \RootY, 0) 
        
        % 5. Close loop
        -- cycle;

    % --- 10. Q_r (Sphere) ---
    % \shade[ball color=orange, opacity=0.15] (0, \DsY, 0) circle (\SphereR);
    %  % --- 4d. Sphere Q Equator (Front Part: z > 0) ---
    % \begin{scope}[canvas is xz plane at y=\DsY]
    %     % z > 0 は sin(theta) > 0 なので 0度〜180度
    %     \draw[dashed, orange!80!black] (0:\SphereR) arc (0:180:\SphereR);
    % \end{scope}
    % \node[orange!80!black] at ({-\SphereR-0.1}, {\DsY+0.5}, 0) {\scriptsize $Q$};

    % --- 11. Line L_j (Top) ---
    \draw[thick, blue] (0, \DiskY, \DiskTwoZ) -- (0, \DiskY, \LineH) node[above] {$L$};

    % --- 12. Box 1-Skeleton (Front) ---
    \draw[black] (-\BoxW, \BoxBot, \BoxZ) rectangle (\BoxW, \BoxTopY, \BoxZ); % Front Face
    \draw[black] (-\BoxW, \BoxTopY, -\BoxZ) -- (-\BoxW, \BoxTopY, \BoxZ); % Top Left connection
    \draw[black] (\BoxW, \BoxBot, \BoxZ) -- (\BoxW, \BoxBot, -\BoxZ) -- (\BoxW, \BoxTopY, -\BoxZ) -- (\BoxW, \BoxTopY, \BoxZ); % Right side connection
    \draw[black] (-\BoxW, \BoxTopY, -\BoxZ) -- (\BoxW, \BoxTopY, -\BoxZ); % Top Back

    % --- 13. Coordinate Axes ---
    \coordinate (AxisOrg) at (-\BoxW - 0.8, \BoxBot-0.2, \BoxZ); 
    \draw[->, >=stealth, thick] (AxisOrg) -- ++(1.0, 0, 0) node[right] {\scriptsize $x$};
    \draw[->, >=stealth, thick] (AxisOrg) -- ++(0, 0, -1.6) node[right] {\scriptsize $y$}; 
    \draw[->, >=stealth, thick] (AxisOrg) -- ++(0, 1.0, 0) node[above] {\scriptsize $z$};

\end{tikzpicture}

}

\savebox{\MyRibbonTypeII}{%

\begin{tikzpicture}[scale=1.0, z={(-0.3cm,-0.2cm)}, line join=round, line cap=round]

    % ==========================================
    % === 1. 設定パラメータ ===
    % ==========================================
    \def\LineH{4.0}      % L_j の長さ (片側)
    \def\DiskR{0.5}      % D_{i:k} (円盤) の半径
    
    % --- リボン設定 ---
    \def\AngSt{-75}      % リボン1: 右側の接続角度
    \def\AngEn{-105}     % リボン1: 左側の接続角度

    % --- ディスクとリボン2の設定 ---
    \def\DiskTwoZ{1.4}      
    \def\RibTwoRootX{1.2}    
    \def\RibTwoWidth{0.15}   
    \def\RibTwoAngSt{-45}    
    \def\RibTwoAngEn{-85}    

    % --- 座標・ボックス設定 ---
    \def\DiskY{1.0}        
    \def\RootY{-3.1}       
    \def\BoxTopY{3.0}      
    
    % --- 内部変数の計算 ---
    \def\BoxW{1.6}             
    \def\BoxBot{\RootY - 0.3}
    \def\BoxZ{\LineH}          

    % ShiftX (正の値)
    \pgfmathsetmacro{\ShiftX}{0.5 * \BoxW}

    % ==========================================
    % === 2. 背景 (Background Box) ===
    % ==========================================
    \draw[dashed, gray!60] (\BoxW, \BoxTopY, -\BoxZ) -- (-\BoxW, \BoxTopY, -\BoxZ) -- (-\BoxW, \BoxBot, -\BoxZ) -- (\BoxW, \BoxBot, -\BoxZ);
    \draw[dashed, gray!60] (-\BoxW, \BoxBot, -\BoxZ) -- (-\BoxW, \BoxBot, \BoxZ);

    % ==========================================
    % ★追加: 矢印 (奥側: z < 0)
    % ==========================================
    % D_+ (Left, x=-ShiftX) を中心に、D_+' (Right, x=+ShiftX) から回る
    % 半径 = 距離(Left, Right) = 2 * ShiftX
    % xz平面での角度: 0度=+x, -90度=-z (奥), 180度=-x
    \begin{scope}[canvas is xz plane at y=\RootY]
        \pgfmathsetmacro{\ArrowRadius}{0.75* \ShiftX}
        % 中心を Left Bandの根本 (-ShiftX) に設定
        % スタート地点(角度0)は Right Bandの根本 (+ShiftX) に相当
        % ★修正: dashed を追加
        \draw[very thick, purple, dashed] (-0.25*\ShiftX, 0) ++(0: \ArrowRadius) arc [start angle=0, end angle=-180, radius=\ArrowRadius];
    \end{scope}

    % ==========================================
    % === 3. 下部の直線 ===
    % ==========================================
    \draw[thick, blue] (-\ShiftX, \DiskY, -\LineH) -- (-\ShiftX, \DiskY, 0);
    \draw[thick, blue] (\ShiftX, \DiskY, -\LineH) -- (\ShiftX, \DiskY, 0);

    % ==========================================
    % === 4. z=0 平面のオブジェクト (Root, Band1, Disk+) ===
    % ==========================================
    
    % --- Root Rectangle ---
    \filldraw[fill=gray!20, draw=black] (-1.5, {\RootY - 0.3}, 0) rectangle (1.5, \RootY, 0);
    \node at (-1.8, {\RootY - 0.15}, 0) {\scriptsize $D_{0}$};

    % --- Left Band 1 (Red) / D_+ ---
    \begin{scope}[canvas is xy plane at z=0]
        \pgfmathsetmacro{\BOneBotRightX}{\DiskR * cos(\AngSt)}
        \pgfmathsetmacro{\BOneBotLeftX}{\DiskR * cos(\AngEn)}
        \pgfmathsetmacro{\BOneTopRightX}{-\ShiftX + \DiskR * cos(\AngSt)}
        \pgfmathsetmacro{\BOneTopLeftX}{-\ShiftX + \DiskR * cos(\AngEn)}
        \pgfmathsetmacro{\BOneTopRightY}{\DiskY + \DiskR * sin(\AngSt)}
        \pgfmathsetmacro{\BOneTopLeftY}{\DiskY + \DiskR * sin(\AngEn)}

        \filldraw[fill=red!3, draw=black, opacity=0.9] 
            (\BOneBotRightX, \RootY) -- (\BOneTopRightX, \BOneTopRightY)
            arc[start angle=\AngSt, end angle=\AngEn, radius=\DiskR]
            -- (\BOneBotLeftX, \RootY) -- cycle;
    \end{scope}

    % --- Right Band 1 (Red) ---
    \begin{scope}[canvas is xy plane at z=0]
        \pgfmathsetmacro{\BOneBotRightX}{\DiskR * cos(\AngSt)}
        \pgfmathsetmacro{\BOneBotLeftX}{\DiskR * cos(\AngEn)}
        % 右側用に座標計算 (ShiftXの符号反転など)
        \pgfmathsetmacro{\BOneTopRightX}{\ShiftX + \DiskR * cos(\AngSt)}
        \pgfmathsetmacro{\BOneTopLeftX}{\ShiftX + \DiskR * cos(\AngEn)}
        \pgfmathsetmacro{\BOneTopRightY}{\DiskY + \DiskR * sin(\AngSt)}
        \pgfmathsetmacro{\BOneTopLeftY}{\DiskY + \DiskR * sin(\AngEn)}

        \filldraw[fill=red!3, draw=black, opacity=0.9] 
            (\BOneBotRightX+0.5*\ShiftX, \RootY)
            -- (\BOneTopRightX, \BOneTopRightY)
            arc[start angle=\AngSt, end angle=\AngEn, radius=\DiskR]
            -- (\BOneBotLeftX+0.5*\ShiftX, \RootY)
            -- cycle;
    \end{scope}
            
    % --- Disks ---
    \begin{scope}[canvas is xy plane at z=0]
        % Left D+
        \filldraw[fill=red!10, draw=red, opacity=0.8] (-\ShiftX, \DiskY) circle (\DiskR);
        \node[red] at ({-\ShiftX + \DiskR + 0.2}, {\DiskY + 0.2}) {\scriptsize $D_{+}$};
        \fill[red] (-\ShiftX, \DiskY) circle (1.5pt);
        
        % Right D+'
        \filldraw[fill=red!10, draw=red, opacity=0.8] (\ShiftX, \DiskY) circle (\DiskR);
        \node[red] at ({\ShiftX + \DiskR + 0.2}, {\DiskY + 0.2}) {\scriptsize $D_{+}'$};
        \fill[red] (\ShiftX, \DiskY) circle (1.5pt);
    \end{scope}

    % ==========================================
    % ★追加: 矢印 (手前側: z > 0)
    % ==========================================
    % 奥側と同じ軌道で、手前側 (角度 -180 -> -350 相当、または 180 -> 10) を描く
    % xz平面: 180度=-x, 90度=+z(手前), 0度=+x
    \begin{scope}[canvas is xz plane at y=\RootY]
        \pgfmathsetmacro{\ArrowRadius}{0.75* \ShiftX}
        % ★修正: dashed を追加
        \draw[->, very thick, purple, >=stealth, dashed] 
            (-0.25*\ShiftX, 0) ++(180: \ArrowRadius) 
            arc [start angle=180, end angle=10, radius=\ArrowRadius];
    \end{scope}

    % ==========================================
    % === 5. 中間の構造 (Middle Lines & Ribbon2) ===
    % ==========================================
    \draw[thick, blue] (-\ShiftX, \DiskY, 0) -- (-\ShiftX, \DiskY, \DiskTwoZ);
    \draw[thick, blue] (\ShiftX, \DiskY, 0) -- (\ShiftX, \DiskY, \DiskTwoZ);

    % Left Ribbon 2
    \filldraw[fill=blue!3, draw=black, opacity=0.9] 
        (-\RibTwoRootX, \RootY, 0)
        .. controls (\RibTwoRootX, {\RootY+1.5}, -1.0) and ({ -\ShiftX + \DiskR * cos(\RibTwoAngEn) }, {\DiskY - 1.5}, {\DiskTwoZ - 1.0}) ..
        ({ -\ShiftX + \DiskR * cos(\RibTwoAngEn) }, { \DiskY + \DiskR * sin(\RibTwoAngEn) }, \DiskTwoZ)
        -- plot[domain=\RibTwoAngEn:\RibTwoAngSt, variable=\t, samples=25] 
        ({ -\ShiftX + \DiskR * cos(\t) }, { \DiskY + \DiskR * sin(\t) }, \DiskTwoZ)
        .. controls ({ -\ShiftX + \DiskR * cos(\RibTwoAngSt) }, {\DiskY - 1.5}, 1.0) and ({\RibTwoRootX - \RibTwoWidth}, {\RootY+1.5}, -1.0) ..
        ({-\RibTwoRootX - \RibTwoWidth}, \RootY, 0) 
        -- cycle;

    % Right Ribbon 2
    \filldraw[fill=blue!3, draw=black, opacity=0.9] 
        (\RibTwoRootX, \RootY, 0)
        .. controls (\RibTwoRootX, {\RootY+1.5}, -1.0) and ({ \ShiftX + \DiskR * cos(\RibTwoAngEn) }, {\DiskY - 1.5}, {\DiskTwoZ - 1.0}) ..
        ({ \ShiftX + \DiskR * cos(\RibTwoAngEn) }, { \DiskY + \DiskR * sin(\RibTwoAngEn) }, \DiskTwoZ)
        -- plot[domain=\RibTwoAngEn:\RibTwoAngSt, variable=\t, samples=25] 
        ({ \ShiftX + \DiskR * cos(\t) }, { \DiskY + \DiskR * sin(\t) }, \DiskTwoZ)
        .. controls ({ \ShiftX + \DiskR * cos(\RibTwoAngSt) }, {\DiskY - 1.5}, 1.0) and ({\RibTwoRootX - \RibTwoWidth}, {\RootY+1.5}, -1.0) ..
        ({\RibTwoRootX - \RibTwoWidth}, \RootY, 0) 
        -- cycle;

    % ==========================================
    % === 6. z=1.4 平面のオブジェクト (Disk-) ===
    % ==========================================
    \begin{scope}[canvas is xy plane at z=\DiskTwoZ]
        \filldraw[fill=blue!10, draw=blue, opacity=0.9] (-\ShiftX, \DiskY) circle (\DiskR);
        \node[blue] at (-\ShiftX, {\DiskY}) [xshift=0.5cm, yshift=0.4cm] {\scriptsize $D_{-}$};
        \fill[blue] (-\ShiftX, \DiskY) circle (1.5pt);
    \end{scope}

    \begin{scope}[canvas is xy plane at z=\DiskTwoZ]
        \filldraw[fill=blue!10, draw=blue, opacity=0.9] (\ShiftX, \DiskY) circle (\DiskR);
        \node[blue] at (\ShiftX, {\DiskY}) [xshift=0.5cm, yshift=0.4cm] {\scriptsize $D_{-}'$};
        \fill[blue] (\ShiftX, \DiskY) circle (1.5pt);
    \end{scope}

    % ==========================================
    % === 7. 上部の直線 ===
    % ==========================================
    \draw[thick, blue] (-\ShiftX, \DiskY, \DiskTwoZ) -- (-\ShiftX, \DiskY, \LineH) node[above] {$L$};
    \draw[thick, blue] (\ShiftX, \DiskY, \DiskTwoZ) -- (\ShiftX, \DiskY, \LineH) node[above] {$L'$};

    % ==========================================
    % === 8. 前景 ===
    % ==========================================
    \draw[black] (-\BoxW, \BoxBot, \BoxZ) rectangle (\BoxW, \BoxTopY, \BoxZ); 
    \draw[black] (-\BoxW, \BoxTopY, -\BoxZ) -- (-\BoxW, \BoxTopY, \BoxZ); 
    \draw[black] (\BoxW, \BoxBot, \BoxZ) -- (\BoxW, \BoxBot, -\BoxZ) -- (\BoxW, \BoxTopY, -\BoxZ) -- (\BoxW, \BoxTopY, \BoxZ); 
    \draw[black] (-\BoxW, \BoxTopY, -\BoxZ) -- (\BoxW, \BoxTopY, -\BoxZ); 

    \coordinate (AxisOrg) at (-\BoxW - 0.8, \BoxBot-0.2, \BoxZ); 
    \draw[->, >=stealth, thick] (AxisOrg) -- ++(1.0, 0, 0) node[right] {\scriptsize $x$};
    \draw[->, >=stealth, thick] (AxisOrg) -- ++(0, 0, -1.6) node[right] {\scriptsize $y$};
    \draw[->, >=stealth, thick] (AxisOrg) -- ++(0, 1.0, 0) node[above] {\scriptsize $z$};

\end{tikzpicture}
}
\savebox{\MyRibbonTypeIPerturbed}{%

\begin{tikzpicture}[scale=1.0, z={(-0.3cm,-0.2cm)}, line join=round, line cap=round]

    % ==========================================
    % === 設定パラメータ ===
    % ==========================================
    \def\LineH{4.0}
    \def\DiskR{0.3}      
    \def\SphereR{0.55}   
    
    \def\AngSt{-75}      
    \def\AngEn{-105}     

    \def\DiskTwoZ{1.4}      
    \def\RibTwoRootX{1.2}   
    \def\RibTwoWidth{0.15}  
    \def\RibTwoAngSt{-45}   
    \def\RibTwoAngEn{-85}   

    \def\DiskY{1.0}       
    \def\RootY{-3.1}      
    \def\BoxTopY{3.0}     
    \def\DsY{-2.5}
    
    \def\BoxW{1.6}             
    \def\BoxBot{\RootY - 0.3}
    \def\BoxZ{\LineH}         
    \def\DosZ{0.5*\BoxZ} 

    % ==========================================
    % === 描画開始 ===
    % ==========================================

    % --- 1. Background ---
    \draw[dashed, gray!60] (\BoxW, \BoxTopY, -\BoxZ) -- (-\BoxW, \BoxTopY, -\BoxZ) -- (-\BoxW, \BoxBot, -\BoxZ) -- (\BoxW, \BoxBot, -\BoxZ);
    \draw[dashed, gray!60] (-\BoxW, \BoxBot, -\BoxZ) -- (-\BoxW, \BoxBot, \BoxZ);

    % --- 2. Line L_j ---
    \draw[thick, blue] (0, \DiskY, -\LineH) -- (0, \DiskY, 0);

    % --- 3. Root Rectangle ---
    \filldraw[fill=gray!20, draw=black] (-1.5, {\RootY - 0.3}, 0) rectangle (1.5, \RootY, 0);
    \node at (-1.8, {\RootY - 0.15}, 0) {\scriptsize $D_{0}$};

    \filldraw[fill=gray!30, draw=black] (-1.5, {\RootY - 0.3}, \DosZ ) rectangle (1.5, \RootY, \DosZ);
    \node at (-1.8, {\RootY - 0.15}, \DosZ) {\scriptsize $D_{0}'$};

    % =======================================================
    % ★ Sphere Q (奥側) & Red Band 描画エリア 
    % =======================================================

    % % --- 4a. Sphere Q Equator (Back Part: z < 0) ---
    % % 赤いバンドより先に描くことで、バンドの後ろに隠れるようにする
    % \begin{scope}[canvas is xz plane at y=\DsY]
    %     % z < 0 は sin(theta) < 0 なので 180度〜360度
    %     \draw[dashed, orange!80!black] (180:\SphereR) arc (180:360:\SphereR);
    % \end{scope}

    % 共通パラメータ計算
    \def\CutGap{0.30} 
    \def\CutYTop{\DsY + \CutGap}
    \def\CutYBot{\DsY - \CutGap}
    \pgfmathsetmacro{\Xr}{\DiskR * cos(\AngSt)}
    \pgfmathsetmacro{\Xl}{\DiskR * cos(\AngEn)}
    \pgfmathsetmacro{\YrTop}{\DiskY + \DiskR * sin(\AngSt)}
    \pgfmathsetmacro{\YlTop}{\DiskY + \DiskR * sin(\AngEn)}

    % --- 4b. Band 1 (Top & Bottom Parts: z=0) ---
    \begin{scope}[canvas is xy plane at z=0]
        % (1) 上部パーツ
        \fill[fill=red!3, opacity=0.9] 
            (\Xr, \YrTop) -- (\Xr, \CutYTop) -- (\Xl, \CutYTop) -- (\Xl, \YlTop) 
            arc (\AngEn:\AngSt:\DiskR) -- cycle;
        \draw[black] 
            (\Xr, \CutYTop) -- (\Xr, \YrTop) arc (\AngSt:\AngEn:\DiskR) -- (\Xl, \YlTop) -- (\Xl, \CutYTop);

        % (3) 下部パーツ
        \fill[fill=red!3, opacity=0.9] (\Xr, \RootY) rectangle (\Xl, \CutYBot);
        \draw[black] (\Xr, \CutYBot) -- (\Xr, \RootY) -- (\Xl, \RootY) -- (\Xl, \CutYBot);
    \end{scope}

    % --- 4b'. Band 1 (Middle Part: Bending to z < 0) ---
    \def\BendZ{-0.80} 
    \fill[fill=red!3, opacity=0.9] 
        (\Xr, \CutYTop, 0) 
        .. controls (\Xr, \DsY, \BendZ) .. (\Xr, \CutYBot, 0)
        -- (\Xl, \CutYBot, 0)
        .. controls (\Xl, \DsY, \BendZ) .. (\Xl, \CutYTop, 0)
        -- cycle;
    \draw[black] (\Xr, \CutYTop, 0) .. controls (\Xr, \DsY, \BendZ) .. (\Xr, \CutYBot, 0);
    \draw[black] (\Xl, \CutYTop, 0) .. controls (\Xl, \DsY, \BendZ) .. (\Xl, \CutYBot, 0);

    % =======================================================
    % ★ D_s (Disk) & Sphere Q (手前)
    % =======================================================

    % --- 4c. Disk D_s (Whole Disk) ---
    % バンドを描き終わった後に一気に描く
    \begin{scope}[canvas is zx plane at y=\DsY]
        \fill[green!10, opacity=0.8] (0,0) circle (\DiskR);
        \draw[green!60!black] (0,0) circle (\DiskR);
        \node[green!60!black] at (0, {-\DiskR -0.2}) {\scriptsize $D'$}; 
    \end{scope}

    % =======================================================
    % ★ D_s から D_{0,s} への新しいリボン
    % =======================================================
    \pgfmathsetmacro{\RibXsRight}{\DiskR * cos(\AngSt)}
    \pgfmathsetmacro{\RibXsLeft}{\DiskR * cos(\AngEn)}
    \pgfmathsetmacro{\RibZsRight}{sqrt(\DiskR*\DiskR - \RibXsRight*\RibXsRight)}
    \pgfmathsetmacro{\RibZsLeft}{sqrt(\DiskR*\DiskR - \RibXsLeft*\RibXsLeft)}
    
    \filldraw[fill=green!5, draw=black, opacity=0.8] 
        (\RibXsRight, \DsY, \RibZsRight) 
        .. controls (\RibXsRight, {\DsY - 0.3}, {\RibZsRight + 0.5}) and (\RibXsRight, {\RootY + 0.3}, {\DosZ - 0.5}) ..
        (\RibXsRight, \RootY, \DosZ)
        -- (\RibXsLeft, \RootY, \DosZ)
        .. controls (\RibXsLeft, {\RootY + 0.3}, {\DosZ - 0.5}) and (\RibXsLeft, {\DsY - 0.3}, {\RibZsLeft + 0.5}) ..
        (\RibXsLeft, \DsY, \RibZsLeft)
        -- cycle;

    %     % --- 4d. Sphere Q Equator (Front Part: z > 0) ---
    % \begin{scope}[canvas is xz plane at y=\DsY]
    %     % z > 0 は sin(theta) > 0 なので 0度〜180度
    %     \draw[dashed, orange!80!black] (0:\SphereR) arc (0:180:\SphereR);
    % \end{scope}

    % =======================================================

    % --- 7. Disk D_+ ---
    \begin{scope}[canvas is xy plane at z=0]
        \filldraw[fill=red!10, draw=red, opacity=0.8] (0, \DiskY) circle (\DiskR);
        \node[red] at ({\DiskR + 0.2}, {\DiskY + 0.2}) {\scriptsize $D_{+}$};
        \fill[red] (0, \DiskY) circle (1.5pt);
    \end{scope}

    % --- 8. Middle Line ---
    \draw[thick, blue] (0, \DiskY, 0) -- (0, \DiskY, \DiskTwoZ);

    % --- 9. Second Disk & Ribbon ---
    \begin{scope}[canvas is xy plane at z=\DiskTwoZ]
        \filldraw[fill=blue!10, draw=blue, opacity=0.9] (0, \DiskY) circle (\DiskR);
        \node[blue] at (0, {\DiskY}) [xshift=0.5cm, yshift=0.4cm] {\scriptsize $D_{-}$};
        \fill[blue] (0, \DiskY) circle (1.5pt);
    \end{scope}

    % Ribbon B_{new}
    \filldraw[fill=blue!3, draw=black, opacity=0.9] 
        (\RibTwoRootX, \RootY, 0)
        .. controls (\RibTwoRootX, {\RootY+1.5}, -3.0) and ({ \DiskR * cos(\RibTwoAngEn) }, {\DiskY - 1.5}, {\DiskTwoZ - 1.0}) ..
        ({ \DiskR * cos(\RibTwoAngEn) }, { \DiskY + \DiskR * sin(\RibTwoAngEn) }, \DiskTwoZ)
        -- plot[domain=\RibTwoAngEn:\RibTwoAngSt, variable=\t, samples=25] 
        ({ \DiskR * cos(\t) }, { \DiskY + \DiskR * sin(\t) }, \DiskTwoZ)
        .. controls ({ \DiskR * cos(\RibTwoAngSt) }, {\DiskY - 1.5}, 1.0) and ({\RibTwoRootX - \RibTwoWidth}, {\RootY+1.5}, 2.0) ..
        ({\RibTwoRootX - \RibTwoWidth}, \RootY, 0) 
        -- cycle;

    % --- 10. Q_r (Sphere Shading) ---
    % 全体を覆うシェーディングは最後に描画
    % \shade[ball color=orange, opacity=0.15] (0, \DsY, 0) circle (\SphereR);
    % \node[orange!80!black] at ({-\SphereR-0.1}, {\DsY+0.5}, 0) {\scriptsize $Q$};

    % --- 11. Line L_j (Top) ---
    \draw[thick, blue] (0, \DiskY, \DiskTwoZ) -- (0, \DiskY, \LineH) node[above] {$L$};

    % --- 12. Box 1-Skeleton (Front) ---
    \draw[black] (-\BoxW, \BoxBot, \BoxZ) rectangle (\BoxW, \BoxTopY, \BoxZ); 
    \draw[black] (-\BoxW, \BoxTopY, -\BoxZ) -- (-\BoxW, \BoxTopY, \BoxZ); 
    \draw[black] (\BoxW, \BoxBot, \BoxZ) -- (\BoxW, \BoxBot, -\BoxZ) -- (\BoxW, \BoxTopY, -\BoxZ) -- (\BoxW, \BoxTopY, \BoxZ); 
    \draw[black] (-\BoxW, \BoxTopY, -\BoxZ) -- (\BoxW, \BoxTopY, -\BoxZ); 

    % --- 13. Coordinate Axes ---
    \coordinate (AxisOrg) at (-\BoxW - 0.8, \BoxBot-0.2, \BoxZ); 
    \draw[->, >=stealth, thick] (AxisOrg) -- ++(1.0, 0, 0) node[right] {\scriptsize $x$};
    \draw[->, >=stealth, thick] (AxisOrg) -- ++(0, 0, -1.6) node[right] {\scriptsize $y$}; 
    \draw[->, >=stealth, thick] (AxisOrg) -- ++(0, 1.0, 0) node[above] {\scriptsize $z$};

\end{tikzpicture}

}
%abstract======================================

\begin{abstract}
    In this paper, we study the space of compactly supported embeddings between Euclidean spaces, $\Emb_c(\R^j, \R^n)$.
    By utilizing hairy graphs, we construct elements in the homotopy groups $\pi_{\bullet}(\bEmb_c(\R^{j}, \R^{n})) \otimes \Q$ corresponding to certain uni-trivalent graphs in the model.
    We then prove that these elements are nontrivial.
    Consequently, we show that the rational homotopy groups of $\Emb_c(\R^{n-2}, \R^n)$ are infinite-dimensional in infinitely many degrees when $n \ge 5$ is odd.
\end{abstract}

%===========================================

{\noindent\footnotesize {\rm Preprint}}\par
\vspace{15mm}
\maketitle
\vspace{-6mm}
\setcounter{tocdepth}{2}
\numberwithin{equation}{section}

\par\vspace{3mm}

\def\baselinestretch{1.06}\small\normalsize

\tableofcontents

\section{Introduction}

A \emph{long embedding} is an embedding
$\mathbb{R}^j \to \mathbb{R}^n$ which coincides,
outside the unit ball in $\mathbb{R}^j$, with the standard embedding
\[
    \iota\colon \mathbb{R}^j \to \mathbb{R}^n,\qquad
    (x_1,\dots,x_j) \mapsto (x_1,\dots,x_j,0,\dots,0).
\]
A long immersion is defined analogously.
We denote the space of long embeddings (resp.\ long immersions) by
$\Emb_c(\mathbb{R}^j,\mathbb{R}^n)$
(resp.\ $\Imm_c(\mathbb{R}^j,\mathbb{R}^n)$).
The motivation for this paper is to understand the homotopy type of
$\Emb_c(\mathbb{R}^j,\mathbb{R}^n)$.
By the Smale--Hirsch theorem, the homotopy type of
$\Imm_c(\mathbb{R}^j,\mathbb{R}^n)$ is well understood
(\cite{Sma59,Hir59}).
We therefore consider the homotopy fiber of the map
\[
    \Emb_c(\mathbb{R}^j,\mathbb{R}^n)
    \longrightarrow
    \Imm_c(\mathbb{R}^j,\mathbb{R}^n),
\]
and denote this space by $\bEmb_c(\mathbb{R}^j,\mathbb{R}^n)$.

In 2017, Fresse, Turchin, and Willwacher~\cite{FTW17,FTW20},
following Arone and Turchin~\cite{AT14,AT15}, showed that the graph
complex $HGC_{n,j}$ describes the rational homotopy type of
$\bEmb_c(\mathbb{R}^j,\mathbb{R}^n)$ when the codimension $n-j$ is
greater than $2$.
More precisely, they proved the following theorem.

\begin{Thm}[\cite{FTW17}]
    Assume that $n-j>2$.
    Then, after applying the degree shift in \cite{FTW17}, there is an
    isomorphism of graded $\mathbb{Q}$-vector spaces
    \[
        \left(
            \bigoplus_l H_l({}^*HGC_{n,j})
        \right)_{>0}
        \cong
        \bigoplus_{k>0}
            \pi_k(\bEmb_c(\mathbb{R}^j,\mathbb{R}^n),\iota)
            \otimes \mathbb{Q}.
    \]
    Here, ${}^*HGC_{n,j}$ denotes the dual of $HGC_{n,j}$, and the
    subscript $>0$ denotes the positive-degree part with respect to the
    shifted grading.
\end{Thm}

Their approach uses embedding calculus, developed by Goodwillie, Klein,
and Weiss~\cite{Wei99,GW99,GKW01}, together with rational homotopy
theory.

On the other hand, there is a more geometric approach based on
Kontsevich's configuration space integrals, which give integral invarinats of knots \cite{Kon94}.
This approach was further developed by Bott--Taubes~\cite{BT94}
and Kohno~\cite{Koh94}.

Motivated by Bott's suggestion to extend this framework to
higher-dimensional knotting phenomena~\cite{Bot96}, configuration space
integrals were subsequently developed for spaces of long embeddings
$\mathbb{R}^j \hookrightarrow \mathbb{R}^n$ by
Cattaneo--Rossi~\cite{CR05} and Watanabe~\cite{Wat07}, with further
refinements by Sakai~\cite{Sak10} and Sakai--Watanabe~\cite{SW12}.

More recently, Yoshioka constructed a combinatorial cochain complex
$DGC_{n,j}$, which is related to $HGC_{n,j}$ by a zigzag of
quasi-isomorphisms, and defined a cochain map
\[
  \overline{I}\colon DGC_{n,j}
  \longrightarrow
  A_{dR}\bigl(\bEmb_c(\mathbb{R}^j,\mathbb{R}^n)\bigr),
\]
called the \emph{modified configuration space integral}~\cite{Yos25b}.

Notably, this approach remains valid and effective even in codimension
two.

These developments suggest the following conjecture.

\begin{Conj} \label{conj:Nondegeneracy}
    If $n-j \ge 3$, then there exists a map
    \[
        c \colon H_l({}^{*}HGC_{n,j}(k,g))
        \longrightarrow
        \pi_{(g-1)(j-1)+(n-j-2)k+l}
        \bigl(\bEmb_c(\mathbb{R}^j,\mathbb{R}^n)\bigr) \otimes \R
    \]
    such that, for every pair
    $[H]\otimes[\gamma]\in
    H^{\bullet}(HGC_{n,j})\otimes H_{\bullet}({}^* HGC_{n,j})$,
    the natural pairing
    $\langle \overline{I}(H), c(\gamma) \rangle$
    coincides with the natural pairing
    $\langle H,\gamma\rangle$.
\end{Conj}

From now on, we focus on the case $l=0$, which we sometimes call the
top term.

Some results are already known in this direction.

\begin{Thm}[\cite{SW12,Yos25c}]
\label{weaker version of the conjecture}
    Assume that $n-j\ge 2$, $g\le 3$, and $j\ge 2$.
    Then the modified configuration space integral
    \[
    \overline{I} \colon
    H_{(j-1)(g-1)+(n-j-2)k}
    \bigl(\bEmb_c(\mathbb{R}^j,\mathbb{R}^n)\bigr)
    \otimes \mathbb{R}
    \longrightarrow
    H_{\mathrm{top}}({}^\ast HGC_{n,j}(k,g))
    \]
    is surjective.
    Furthermore, if $g\le 1$, or if $n-j=2$ and $g=2,3$, then its
    precomposition with the Hurewicz map is also surjective.
\end{Thm}

Sakai and Watanabe~\cite{SW12} proved this theorem for the case $g=1$.
Yoshioka~\cite{Yos25c} proved it for the cases $g=2,3$. Moreover, in the
case where $n-j=2$, he showed that the corresponding rational homotopy
groups are infinite-dimensional.

Related results were obtained independently by Budney--Gabai~\cite{BG21}
and Watanabe~\cite{Wat23b}: before Yoshioka's work, they proved
infinite-dimensionality results in the degrees corresponding to the case
$g=2$ and $n-j=2$.

For even $n\ge 6$ with $n-j=2$, Fernandes--Muñoz-Echániz~\cite{FM26}
obtained results analogous to those of Budney--Gabai. They also studied
nearby positive degrees: in the even-dimensional case, they proved
vanishing in all positive degrees below the degree detected by
Budney--Gabai. In the case where $n-j=2$ and $n$ is odd, they proved
finite-dimensionality in positive degrees up to two degrees below that
degree, and infinite-dimensionality one degree below it.

For $g\ge 4$, the computation of $ H_{\mathrm{top}} (^{*}HGC_{n,n-2}(\bullet,g))$ becomes more
difficult. Nevertheless, when $n$ is odd, one can prove the nontriviality of
certain graphs $G_{p,g}$ by evaluating them with the
$\mathfrak{sl}_2$ weight system introduced by Bar-Natan~\cite{Bar95}.
In particular, $H_{\mathrm{top}}({}^*HGC_{n,n-2}(\bullet,g))$ is infinite-dimensional.

The aim of this paper is to realize these graph homology classes as
elements of the rational homotopy groups of the space of long embeddings.
More precisely, we prove the following theorem.

\begin{Thm}
    Let $n\ge 5$ be an odd integer and let $g\ge 2$.
    Then the rational homotopy group
    \[
        \pi_{(n-3)(g-1)}
        \bigl(\bEmb_c(\mathbb{R}^{n-2},\mathbb{R}^n)\bigr)
        \otimes \mathbb{Q}
    \]
    is infinite-dimensional.
\end{Thm}

The construction is based on decomposing Yoshioka's
cycle~\cite{Yos25a,Yos25c} into local models and modifying these local
models so that the resulting cycles lie in the image of the Hurewicz map.
The modified local model agrees with the one used in
Watanabe's families of clasper surgeries~\cite{Wat09a}; this agreement
will be discussed in~\cite{Iri26}. 
Watanabe's families of clasper surgeries are high-dimensional analogues
of clasper surgery, also called $Y$-surgery, introduced independently by
Habiro~\cite{Hab00} and Goussarov~\cite{Gou99}.

\begin{Assum}
    Throughout this paper, we assume that $n$ and $j$ are odd and that
    $n-j\ge 2$.
\end{Assum}

\begin{Rem}
    The assumption that $n$ and $j$ are both odd is imposed mainly to simplify notations and arguments of signs.
    The results remain valid for the other parity cases as well, although
    treating them requires additional sign conventions and technical
    modifications which we do not discuss in this paper.

    Moreover, the construction is not intrinsically restricted to the
    particular graphs considered below.  It can be extended to arbitrary
    graphs, but doing so requires additional choices and verifications.
    We therefore restrict ourselves to the graphs needed for the proof of
    the main theorem.

    The main reason for working under the above assumption is that, in
    the case $j=n-2$, it provides a convenient setting in which
    $H_{\mathrm{top}}({}^*HGC_{n,n-2}(\bullet,g))$ is
    infinite-dimensional.
\end{Rem}

The paper is organized as follows.
In Section~\ref{sec:Preliminary}, we recall the fundamental notions,
such as chord diagrams and graph complexes, that are needed to construct
the cycles.
In Section~\ref{sec:Construction of the cycle}, we construct infinitely
many cycles $c_{p,g}$ from infinitely many nontrivial cycles $G_{p,g}$
in $H_{\mathrm{top}}({}^*HGC_{n,j}(2p+g-1,g))$ via ribbon presentations.
In Section~\ref{sec:Non-triviality}, we prove that the cycles $c_{p,g}$
are nontrivial by using Yoshioka's modified configuration space
integral.
In Section~\ref{Sec:The cycle is in the image of Hurewicz map}, we
reconstruct $c_{p,g}$ and show that they lie in the image of the
Hurewicz map.

\section*{Acknowledgments}

I would like to thank my supervisor Tadayuki Watanabe for his guidance and helpful discussions.
I also thank Leo Yoshioka for helpful email correspondence and Keiichi Sakai for an online discussion.
Finally, I am grateful to Koki Yamaguchi for his comments on the infinite dimensionality of the space of open Jacobi diagrams.

\section{Preliminaries} \label{sec:Preliminary}

\subsection{Basic Notations}

\noindent
We first fix the notation for the standard inclusion.

\begin{Not}
    We denote by $\iota$ the standard linear embedding
    $\R^{j} \hookrightarrow \R^{n}$ defined by
    \[
        \iota(x_1,\dots,x_j)
        =
        (x_1,\dots,x_j,0,\dots,0).
    \]
\end{Not}

\noindent
Using the standard inclusion $\iota$, we define the long version of
the embedding space. This formulation is technically convenient because
it replaces boundary conditions by a fixed behavior at infinity.

\begin{Def}
    A \emph{long embedding} is an embedding $f\colon \R^j \to \R^n$
    such that $f$ agrees with $\iota$ outside a compact subset of $\R^j$.
    We denote the space of long embeddings by
    $\Emb_c(\R^j,\R^n)$.
    This space is weakly equivalent to the space of embeddings fixed on
    the boundary $\Emb_{\partial}(D^j,D^n)$.
    Similarly, we denote by $\Imm_c(\R^j,\R^n)$ the space of long
    immersions.
\end{Def}

\noindent
To study the space of embeddings, which is often difficult to analyze
directly, we compare it with the corresponding space of immersions.
The difference between these two spaces is encoded homotopy theoretically
by the following homotopy fiber.

\begin{Def}
    We denote by $\bEmb_c(\R^j,\R^n)$ the homotopy fiber of the natural
    inclusion
    \[
        \Emb_c(\R^j,\R^n)
        \hookrightarrow
        \Imm_c(\R^j,\R^n)
    \]
    over the base point $\iota \in \Imm_c(\R^j,\R^n)$.
    Concretely, an element of $\bEmb_c(\R^j,\R^n)$ can be regarded as a
    pair $(f,\gamma)$, where $f \in \Emb_c(\R^j,\R^n)$ and
    $\gamma\colon [0,1]\to \Imm_c(\R^j,\R^n)$ is a path satisfying
    \[
        \gamma(0)=f,
        \qquad
        \gamma(1)=\iota .
    \]
    This space is called the space of \emph{embeddings modulo immersions}.
\end{Def}

\subsection{Chord diagrams}

We recall the definition of a chord diagram.

\begin{Def}[Chord diagram on directed lines]
    \label{chord diagram on directed lines}
    A \emph{chord diagram} $C$ of order $k$ on $s$ directed lines consists of the following data:
    \begin{itemize}
        \item Integers $t_i \ge 1$ ($i = 1, \dots, s$) satisfying
        \[
            \sum_{i=1}^{s} (t_i +1) = 2k.
        \]
        \item An ordered set of pairs $\{p_m\}_{m=1}^k$ forming a partition of the set of $2k$ points
        \[
            V(C) = \{ (i, l) \in \mathbb{Z}^2 \mid 1 \le i \le s,\; 0 \le l \le t_i  \}.
        \]
    \end{itemize}
    These data are required to satisfy the following condition:
    \begin{itemize}
        \item For each $1 \le i \le s$, the initial point $(i, 0)$ must be the first element of some pair, and the terminal point $(i, t_i)$ must be the second element of some pair.
    \end{itemize}
\end{Def}

We often regard each pair $p_m$ as an oriented edge (or oriented chord).
Such a chord diagram $C$ is depicted schematically in Figure~\ref{depiction of chord diagram}.
We define an ordering on the vertex set $V(C)$, called the \emph{induced ordering}, denoted by $(v_1, v_2, \dots, v_{2k})$. This ordering is determined by the chords such that for each $m \in \{1, \dots, k\}$, the vertices $v_{2m-1}$ and $v_{2m}$ correspond to the first and second elements of the $m$-th chord $p_m$, respectively.

\begin{figure}[h]
    \centering
    %begin-DepictionOfChordDiagram================================================
    \begin{tikzpicture}[
        scale=0.5,
        >=Latex,
        % --- スタイル定義 ---
        mid arrow/.style={postaction={decorate,decoration={
            markings,
            mark=at position 0.55 with {\arrow{>}}
        }}},
        dot/.style={circle, fill=black, inner sep=1.5pt},
        dashed edge/.style={thick, dashed},
        num/.style={fill=white, inner sep=1.5pt, font=\small}
    ]
    
        % ============================================
        % 1. 左側の図
        % ============================================
        \begin{scope}[shift={(0,0)}]
            \coordinate (m) at (0,-1);
            \coordinate (O) at (0, 0);
            \coordinate (Top) at (0, 6.5);
            \coordinate (P1) at (0, 1.0);
            \coordinate (P2) at (0, 2);
            \coordinate (P3) at (0, 3.0);
            \coordinate (P4) at (0, 4.0);
            \coordinate (P5) at (0, 5.0);
    
            \draw[->,thick] (m) -- (Top);
            % \node at (O) [below left] {O};
    
            \foreach \p in {O,P1, P2, P3, P4, P5} {
                \node[dot] at (\p) {};
            }
    
            % --- ★変更箇所: bend right -> out/in ---
            % 左側に膨らむように out=180, in=180 を指定
            % ラベル位置も right -> left に変更
            \draw[dashed edge, mid arrow] (P1) to[out=180, in=180, looseness=1.5] node[pos=0.5, left] {1} (P4);
            \draw[dashed edge, mid arrow] (P2) to[out=180, in=180, looseness=1.5] node[pos=0.5, left] {2} (P3);
        \end{scope}

        % ============================================
        % 2番目の図
        % ============================================
        \begin{scope}[shift={(3,0)}] 
            \coordinate (m2) at (0,-1);
            \coordinate (O2) at (0, 0);
            \coordinate (Top2) at (0, 6.5);
            \coordinate (P12) at (0, 1.0);
            
            \draw[->,thick] (m2) -- (Top2);
            \foreach \p in {O2, P12} { \node[dot] at (\p) {}; }
        \end{scope}
    
        % ============================================
        % 3番目の図
        % ============================================
        \begin{scope}[shift={(6,0)}] 
            \coordinate (m3) at (0,-1);
            \coordinate (O3) at (0, 0);
            \coordinate (Top3) at (0, 6.5);
            \coordinate (P13) at (0, 1.0);
            
            \draw[->,thick] (m3) -- (Top3);
            \foreach \p in {O3, P13} { \node[dot] at (\p) {}; }
        \end{scope}
        
        % ============================================
        % 4番目の図
        % ============================================
        \begin{scope}[shift={(9,0)}] 
            \coordinate (m4) at (0,-1);
            \coordinate (O4) at (0, 0);
            \coordinate (Top4) at (0, 6.5);
            \coordinate (P14) at (0, 1.0);
            
            \draw[->,thick] (m4) -- (Top4);
            \foreach \p in {O4, P14} { \node[dot] at (\p) {}; }
        \end{scope}
    
        % ============================================
        % ★変更箇所: 破線による接続 (out/in で曲げる)
        % ============================================
        
        % 3, 4, 5: 右(0度)に出て、左(180度)から入る -> S字カーブ
        \draw[->, dashed edge] (O)  --  node[midway, above] {3} (P12);
        \draw[->, dashed edge] (O2) --  node[midway, above] {4} (P13);
        \draw[->, dashed edge] (O3) --  node[midway, below] {5} (P14);
        
        % 6: 右端(O4) から 左端上部(P5) への戻り
        % 右(0度)に出て、右(0度)から入る -> 右側を大きく回るC字カーブ
        % looseness で膨らみ具合を調整
        \draw[->, dashed edge] (O4) --  node[midway, above] {6} (P5);
    
    \end{tikzpicture}
%end-depectionofchordDiagram==============================================
    \caption{Depiction of a chord diagram}
    \label{depiction of chord diagram}
\end{figure}
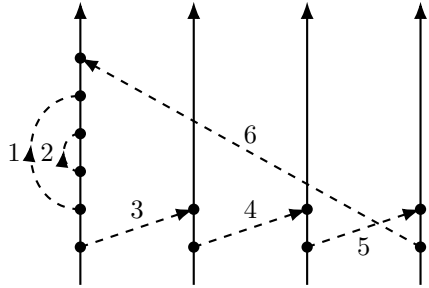

\subsection{Graph complexes}

% --- Introduction / Overview ---
We recall Yoshioka's modified configuration space integral and the graph complexes $HGC_{n,j}$, $PGC'_{n,j}$, $PGC_{n,j}$, and $DGC_{n,j}$.
These complexes are connected by a zigzag of projections which are quasi-isomorphisms:
\[ 
    HGC_{n,j} \xtwoheadleftarrow{\simeq} PGC_{n,j}' \xtwoheadrightarrow{\simeq} PGC_{n,j} \xtwoheadleftarrow{\simeq} DGC_{n,j}.
\]

% --- Definition of Plain Graph ---
\begin{Def}[Plain graph] \label{Def of plain graph}
    A \emph{plain graph} consists of three types of vertices:
    \emph{white vertices} \tikz[baseline=-0.6ex]{\draw (0,0) circle (0.08);},
    \emph{external black vertices} \tikz[baseline=-0.6ex]{\filldraw[black] (0,0) circle (0.08);},
    and \emph{internal black vertices} \tikz[baseline=-0.6ex]{\filldraw[black] (-0.07,-0.07) rectangle (0.07,0.07);};
    and two types of edges:
    \emph{solid edges} \tikz[baseline=-0.6ex]{\draw[thick] (0,0) -- (0.6,0);} (or \emph{$\eta$-edges})
    and \emph{dashed edges} \tikz[baseline=-0.6ex]{\draw[dashed, thick] (0,0) -- (0.6,0);} (or \emph{$\theta$-edges}).

    These vertices and edges must satisfy the following properties:
    \begin{itemize}
        \item Every internal black vertex has at least three dashed edges and no solid edges.
        \item Every white vertex has at least three dashed edges and no solid edges.
        \item Each connected component has at least one external vertex.
    \end{itemize}

    A plain graph is called \emph{admissible} if every external black vertex is incident to at least one dashed edge.

    A plain graph is called \emph{good} if its restriction to the solid edges is a disjoint union of lines such as 
    % \begin{center}
        \begin{tikzpicture}[baseline=-0.5ex]
            % 共通設定: 黒丸(vertex)と太線(edge)のスタイル
            \tikzset{
                v/.style={fill, circle, inner sep=1.5pt}, 
                e/.style={thick} 
            }
    
            % --- 左側: 辺が2つ (2 edges) ---
            % (0,0) から描画開始
            \draw[e] (0,0) -- (1.4,0);
            \foreach \x in {0, 0.7, 1.4} {
                \node[v] at (\x,0) {};
            }
            % \node at (0.7, -0.4) {\tiny 2 edges};
    
            % --- 右側: 辺が3つ (3 edges) ---
            % scope環境で座標を右(x=3.0)にずらす
            \begin{scope}[shift={(2.0, 0)}]
                \draw[e] (0,0) -- (2.1,0);
                \foreach \x in {0, 0.7, 1.4, 2.1} {
                    \node[v] at (\x,0) {};
                }
                % \node at (1.05, -0.4) {\tiny 3 edges};
            \end{scope}
        \end{tikzpicture}.
    % \end{center}
\end{Def}

% --- Notation for graph components ---
\noindent
We fix the notation for the components of a graph as follows.

\begin{Not}
    Let $\Gamma$ be a plain graph. We decompose the set of edges $E(\Gamma)$ and vertices $V(\Gamma)$ as follows:
    \[
        E(\Gamma) = E_{\theta}(\Gamma) \cup E_{\eta}(\Gamma), \qquad
        V(\Gamma) = W(\Gamma) \cup \underbrace{B_i(\Gamma) \cup B_e(\Gamma)}_{B(\Gamma)},
    \]
    where $E_{\theta}$ (resp.~$E_{\eta}$) denotes the set of dashed (resp.~solid) edges, and $W, B_i, B_e$ denote the sets of white, internal black, and external black vertices, respectively.
\end{Not}

\begin{Def}
    A \emph{labeled} plain graph is a plain graph $\Gamma$ equipped with a
    labeling of the vertex set $V(\Gamma)$ and an orientation of each edge in
    $E(\Gamma)$.
    We impose the following orientation relations: swapping two vertex labels or
    reversing the orientation of an edge changes the sign.
\end{Def}

\begin{Exa}
    Here is an example of a labeled plain graph with $g=5$, $k= 5$ ,$l=3$, $|B_e|=4$, $|B_i|=1$, $|W|=2$, $|E_{\theta}| =7$, and $ |E_{\eta}|=4$. 
%----------------------------begin{exampleofpaingraph}-------------------------------
\begin{figure}[h] \label{examplePlaingraph}
    \centering
    \begin{tikzpicture}[
        % --- 共通設定 ---
        % >={Stealth[length=2mm]}, % 矢印の形状
        >=Latex, % 矢印の形状
        % 辺の真ん中に矢印を置くスタイル
        mid arrow/.style={postaction={decorate,decoration={
            markings,
            mark=at position 0.55 with {\arrow{>}}
        }}},     
        % --- 頂点のスタイル定義 ---
        % White vertex: 白抜きの円
        white v/.style={circle, draw=black, fill=white, thick, inner sep=2pt},
        % External vertex: 黒塗りの円
        ext v/.style={circle, fill=black, inner sep=2pt},
        % Internal vertex: 四角形 (バツ印なし)
        int v/.style={rectangle, draw=black, fill=black, thick, minimum size=0.2cm},      
        % 線のスタイル
        dashed edge/.style={dashed, thick},
        solid edge/.style={thick}
    ]
        % --- 座標の定義 ---
        \coordinate (W1) at (-1.0, 2.0);
        \coordinate (W2) at (1.0, 2.0);
    
        \coordinate (E1) at (-2, 0);
        \coordinate (E2) at (-0.7, 0);
        \coordinate (E3) at (0.7, 0);
        \coordinate (E4) at (2, 0);
    
        \coordinate (I) at (0, -1.5);
    
        % --- ノードの配置 (ラベルを脇に配置) ---
        
        % 1, 2: 上に配置
        \node[white v, label=above:1] (w1) at (W1) {};
        \node[white v, label=above:2] (w2) at (W2) {};
    
        % 3, 4, 5, 6: 左右や斜めに配置
        \node[ext v, label=left:3] (e1) at (E1) {};
        \node[ext v, label=225:4] (e2) at (E2) {}; % 左下
        \node[ext v, label=315:5] (e3) at (E3) {}; % 右下
        \node[ext v, label=right:6] (e4) at (E4) {};
    
        % 7: 下に配置 (四角のみ)
        \node[int v, label=below:7] (i) at (I) {};
        
        % バツ印の描画コードを削除しました
    
        % --- 辺の描画 (矢印付き) ---
    
        % 1. 上段
        \draw[dashed edge, mid arrow] (w1) -- (w2);
    
        % 2. 中段への接続
        \draw[dashed edge, mid arrow] (w1) -- (e1);
        \draw[dashed edge, mid arrow] (e2) -- (w1);
        \draw[dashed edge, mid arrow] (w2) -- (e4);
        \draw[dashed edge, mid arrow] (w1) -- (e3); 
        \draw[dashed edge, mid arrow] (w2) -- (e2);
    
        % 3. 下段への接続
        \draw[solid edge, mid arrow] (i) -- (e1);
        \draw[solid edge, mid arrow] (e2) -- (i);
        \draw[solid edge, mid arrow] (i) -- (e3);

        %4. コード
        \draw[dashed edge, mid arrow] (e1) to [bend left=50] (e3);

        %5 中段
        \draw[solid edge, mid arrow] (e2) to [bend right = 40] (e4);
        
    \end{tikzpicture}
    \caption{An example of a labeled plain graph.}
\end{figure}
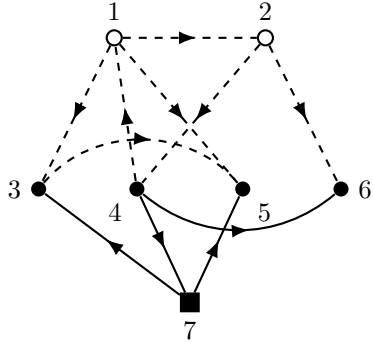
%end{examplePlaingraph}-------------------------------------------------------
\end{Exa}

% --- Convention on Orientation ---
\noindent
In the following definitions, we assume that $n$ and $j$ are odd integers. If they are even, the sign conventions would differ (see, for example, \cite[Definition 2.5]{Yos25b}).

% --- Specific types of graphs ---
Hairy graphs play a crucial role in constructing cycles in $\bEmb_c(\R^j,\R^n)$.
\begin{Def}[Hairy graph]
    A plain graph $\Gamma$ is called a \emph{hairy graph} if every black vertex is incident to exactly one dashed edge and no solid edges.
    Note that this implies $\Gamma$ has no internal black vertices (i.e., $B_i(\Gamma) = \emptyset$).
\end{Def}

\begin{Def}[\cite{Yos25b}]
    We define the vector spaces $PGC'_{n,j}, HGC_{n,j}$ as the quotient:
    \begin{align*}
        PGC'_{n,j} &= \frac{\mathbb{Q} \{ \text{admissible labeled connected plain graphs} \}}{\text{orientation relations} }\\
        HGC_{n,j} &= \frac{\mathbb{Q} \{ \text{admissible labeled connected hairy graphs} \}}{\text{orientation relations} } 
    \end{align*}        
\end{Def}

\noindent

% --- Differential ---
\noindent
Similarly, $PGC_{n,j}$ is defined as the quotient of $PGC'_{n,j}$
by the subspace spanned by graphs containing double edges or self-loops.

We first define a differential $d_{PGC'}$ on $PGC'_{n,j}$ by
\[
  d_{PGC'}(\Gamma)
  =
  \sum_{e\in E(\Gamma)} \operatorname{sign}(e)\,(\Gamma/e),
\]
where the sum is taken over all edges $e$ that are neither loop edges
nor chord edges. Here, a chord edge means an edge connecting two black
vertices. This differential descends to a differential $d_{PGC}$ on
$PGC_{n,j}$.

\begin{Prop}
    The pair $(PGC'_{n,j}, d_{PGC'})$ forms a cochain complex, i.e., $d_{PGC'}^2 = 0$.    
\end{Prop}

\begin{proof}
    We omit the proof. See \cite[Lemma 2.15]{Yos25b}.
\end{proof}

% --- Grading (Defect and Order) ---
\noindent
We now introduce bigradings on these complexes.

\begin{Def}
    The \emph{defect} $l(\Gamma)$ of a plain graph $\Gamma$ is defined as:
    \[ l(\Gamma) = 2|E_{\theta}(\Gamma)| - 3|W(\Gamma)| - |B(\Gamma)|. \]
    The \emph{order} $k(\Gamma)$ is defined as:
    \[ k(\Gamma) = |E_{\theta}(\Gamma)| - |W(\Gamma)|. \]
    We denote the first Betti number of $\Gamma$ by $g(\Gamma)$.
\end{Def}

\begin{Not}
    We denote by $PGC^{l}_{n,j}(k,g)$ the subspace of $PGC_{n,j}$ spanned by graphs $\Gamma$ satisfying $k(\Gamma)=k$, $g(\Gamma)=g$, and $l(\Gamma)=l$.
    Since the differential preserves the order and the first Betti number while increasing the defect by $1$, the family $\{PGC^{l}_{n,j}(k,g)\}_l$ forms a subcomplex of $PGC_{n,j}$.
\end{Not}

\noindent
In this paper, we focus on the cohomology $H^{\mathrm{top}}(HGC^{\bullet}_{n,j})$.
We call the defect $0$ subspace the \emph{top term} of the graph complex.
% --- Decorated Graphs (DGC) ---
\noindent
Finally, we introduce the complex of decorated graphs, which plays an important role in constructing a cocycle in $\bEmb_c(\R^j,\R^n)$.

\begin{Def}
    A \emph{decorated graph} is a plain graph equipped with an assignment of an element in
    \[
        Z_{n,j} := A_{n,j} \otimes BA_{n,j} \quad (\text{the acyclic left bar complex of } A_{n,j})
    \]
    to each external vertex (refer to \cite{Yos25b} for $A_{n,j}$).

    % (コメントアウト部分はそのまま)

    We denote the underlying plain graph by $P(\Gamma)$. The decorations are denoted by $D(\Gamma) = \bigotimes_{i=1}^{m} D_i(\Gamma)$, where $D_i(\Gamma)$ is the element attached to the $i$-th external vertex and $m = |B_e(P(\Gamma))|$.
\end{Def}

\noindent
The notion of an \emph{admissible} decorated graph is defined analogously. 
The space $DGC_{n,j}$ is generated by such admissible decorated graphs equipped with the differential $d_{DGC}$ (see \cite[Section 4.4]{Yos25b}). The gradings $g(\Gamma)$ and $k(\Gamma)$ for decorated graphs are also defined in a similar manner (see \cite[Notation 4.13 and 4.19, Remark 4.33]{Yos25b}).

\begin{Thm}[\cite{Yos25b}, Theorems 2.38, 2.46, 4.35]
    The natural projections 
    \[ HGC_{n,j} \xtwoheadleftarrow{p_1} PGC_{n,j}' \xtwoheadrightarrow{p_2} PGC_{n,j} \xtwoheadleftarrow{p_3} DGC_{n,j} \]
    are all quasi-isomorphisms.
\end{Thm}

\begin{proof}
    See \cite{Yos25b} for the detailed proof.
\end{proof}

\section{Construction of the cycle} \label{sec:Construction of the cycle}

In this section, we construct a nontrivial cycle $c_{p,g}$ associated with $G_{p,g}$ representing a nontrivial homology class in $H_{\mathrm{top}}({}^* HGC_{n,j}(2p+g-1,g))$.

Our construction strategy is summarized in the following diagram:

\begin{center}
\begin{tikzpicture}[
    node distance=0.8cm and 0.6cm,
    every node/.style={align=center, font=\small},
    arrow/.style={-latex, thick}
]
    % Nodes
    \node (G) [draw, rounded corners] {nontrivial cycle $G_{p,g}$ \\ in $H_{\mathrm{top}}({}^*HGC_{n,j})$};
    \node (D) [draw, rounded corners, right=of G] {Chord diagram \\ $D(G_{p,g})$};
    \node (P) [draw, rounded corners, right=of D] {Ribbon presentation \\ $P(G_{p,g})$};
    \node (C) [draw, rounded corners, right=of P, fill=gray!10] {Cycle $c_{p,g}$ in \\ $\bEmb_c(\R^j, \R^n)$};

    % Arrows
    \draw[arrow] (G) -- (D);
    \draw[arrow] (D) -- (P);
    \draw[arrow] (P) -- (C);
\end{tikzpicture}
\end{center}

In \S~\ref{subsec:RibbonPresentation}, we recall the notion of a \emph{ribbon presentation}. 
In \S~\ref{subsec:Cycles associated with chord diagrams}, we explain how to obtain a ribbon presentation from a chord diagram; this construction is a novel aspect of this paper.
In \S~\ref{subsec:Infiniteness_of_HGC}, we define $G_{p,g}$ and prove that its class in $H_{\mathrm{top}}(^{*}HGC_{n,j}(2p+g-1 ,g))$ is nontrivial.
Finally, in \S~\ref{subsec:The_chord_diagram_D(G(p,g))}, we explain the construction of the chord diagram $D(G_{p,g})$ associated with the hairy graph $G_{p,g}$.

\begin{Rem}
    In this paper, cycles are constructed using ribbon presentations, and their nontriviality is established in Section 4. 
    However, these cycles can also be described as families of claspers obtained by iterated plumbing, from which their nontriviality can be proved directly.
\end{Rem}

\subsection{Ribbon presentation}\label{subsec:RibbonPresentation}

We recall the notion of a ribbon presentation and review some of its basic properties.% This section is based on \cite[Section 1.2]{Yos25c}

\begin{Def}[Ribbon presentation \cite{HS01}] \label{Def of ribbon presentation}
    A \emph{ribbon presentation} $P= \mathcal{D} \cup \mathcal{B}$ is an immersed based oriented 2-disk in $\R^3$, where

    \begin{itemize}
        \item \( \mathcal{D} = D_0 \cup \dots\cup D_d \) is a union of disjoint disks.
        \item $\mathcal{B} = B_1 \cup \dots \cup B_{d'}, \quad  (B_i \cong I \times I) $  is a union of disjoint bands.

    \end{itemize}

    Each band connects distinct disks and the interior of a disk can intersect a band transversally.
    We take a base point in $\partial D_0$.
    We call $D_0$ the base disk of $P$.
    \end{Def}

% We consider the following additional structure.

% \begin{Def}
%     A \emph{labeled ribbon presentation} is a pair consisting of a ribbon presentation $P$ and a collection of 3-disks $\mathcal{Q} = Q_1 \cup \dots \cup Q_k$ in $\mathbb{R}^3$. For each 3-disk $Q_i$, let $I_i$ (resp. $J_i$) be the set of bands (resp. disks) that intersect $Q_i$. The local model in $Q_i$ is required to satisfy the following:
%     \begin{align*}
%         B &= \{ (x_1,x_2,x_3) \in \mathbb{R}^3 \mid |x_1| \le 1/2 ,\, |x_2| < 3 ,\, x_3 = p_B  \} \\
%         D &= \{ (x_1,x_2,x_3) \in \mathbb{R}^3 \mid |x_1|^2 + |x_3|^2 \le 1,\,  x_2 = q_D \}
%     \end{align*}
%     Here, $\{p_B\}_{B \in I_i}$ and $\{q_D\}_{D \in J_i}$ are sets of distinct real numbers satisfying $|p_B|, |q_D| < 1/1000$. The base point is assumed to lie on the boundary of $D_0$. The \emph{order} of a labeled ribbon presentation is the number $k$ of 3-disks in $\mathcal{Q}$.
% \end{Def}
    %ここに図を挿入
    
\begin{Def}
    An intersection of a disk with a band is called a \emph{crossing}.
    A disk without any crossings is called a \emph{node}.
    A \emph{leaf} is a disk with crossings.
\end{Def}

\begin{figure}[htbp]
    \centering
    \includegraphics[width=0.3\linewidth]{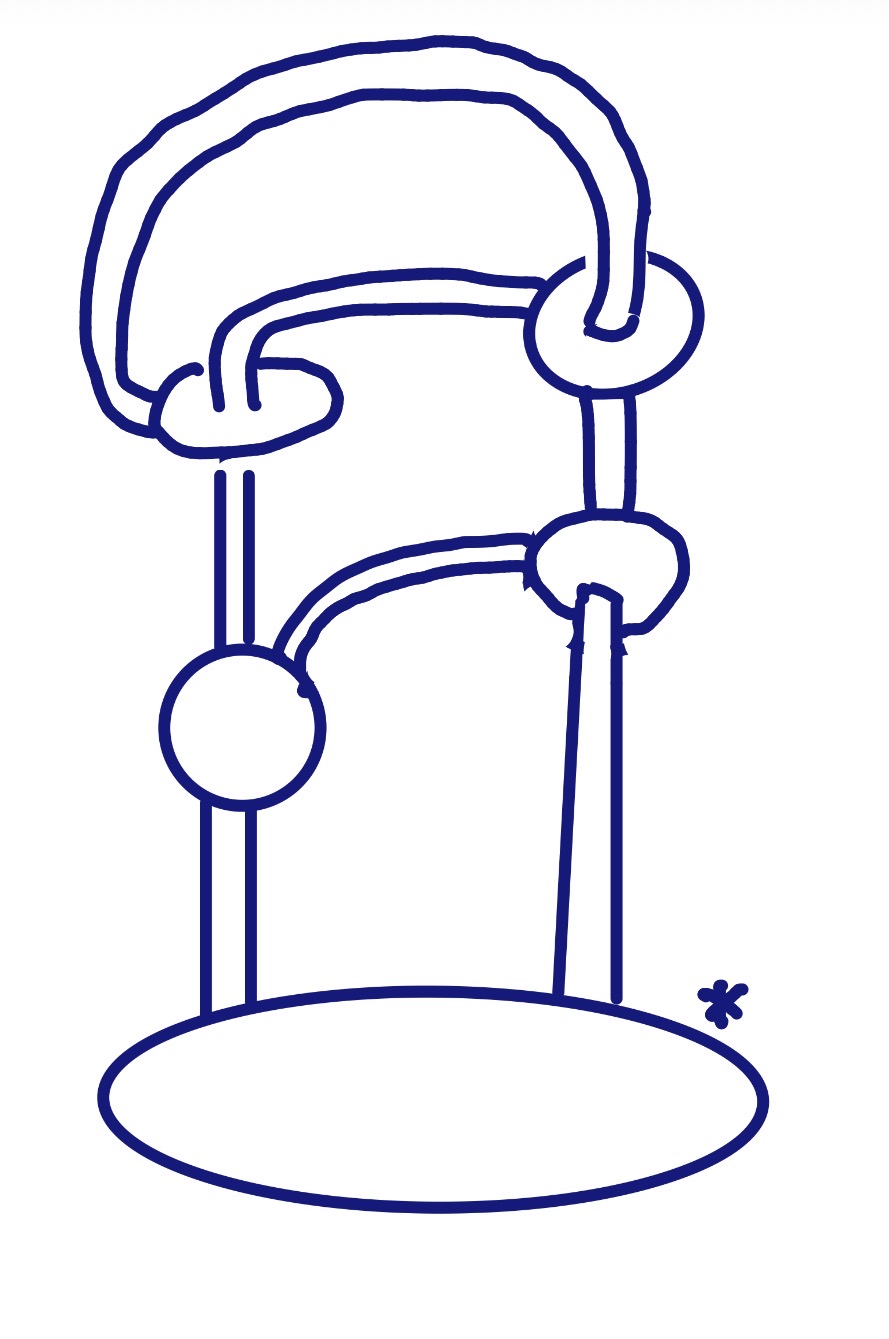}
    \caption{Example of a ribbon presentation}
    \label{fig:exampleofRibbonPresentation}
\end{figure}

\begin{Not} [Orientation of a crossing]
    Since $\mathcal{D} \cup \mathcal{B}$ is an oriented disk, each disk $D$ is oriented.  
    We orient the core of each band $B_i$ so that it goes from a leaf to the base disk $D_0$.  
    A crossing of a band with a disk (respectively, a crossing of a line with a disk) is called \emph{positive} 
    if the core of the band  gives the oriented normal vector of $D_{i}$ in $\R^3$.  
    Otherwise, it is called \emph{negative}.
\end{Not}

We recall deformations of ribbon presentations that preserve embeddings modulo immersions.

\begin{Not} [{\cite[Definition~3.2]{HS01}}]
   We define the S1, S3, S4 and S7 moves as illustrated in Figure \ref{fig:SMoves}.
% The gray-colored region in the diagram of the S4-move represents either a part of a band or a part of a line segment.

    \begin{figure}[htbp]
        \centering
    
        % ==================================================
        % 左上: S1-move
        % ==================================================
        \begin{subfigure}[b]{0.48\textwidth}
            \centering
            % scaleboxで図全体を縮小（0.5倍）
            \scalebox{0.4}{
                \begin{tikzpicture}
                    \def\R{1.2} \def\Cx{2.0} \def\Angle{15} \def\Sep{8.5}
                    % --- Left Figure ---
                    \begin{scope}[shift={(-\Sep/2, 0)}]
                        \draw[thick, dashed, line width=2pt] (-\Cx, \R) arc [start angle=90, end angle=270, radius=\R];
                        \draw[thick, line width=2pt] (-\Cx, -\R) arc [start angle=270, end angle=360, radius=\R];
                        \draw[thick, line width=2pt] (-\Cx, \R) arc [start angle=90, end angle=0, radius=\R];
                        \draw[thick, dashed, line width=2pt] (\Cx, -\R) arc [start angle=-90, end angle=90, radius=\R];
                        \draw[thick, line width=2pt] (\Cx, \R) arc [start angle=90, end angle=180, radius=\R];
                        \draw[thick, line width=2pt] (\Cx, -\R) arc [start angle=270, end angle=180, radius=\R];
                        \draw[thick, line width=2pt] ({-\Cx + \R*cos(\Angle)}, {\R*sin(\Angle)}) -- ({\Cx - \R*cos(\Angle)}, {\R*sin(\Angle)});
                        \draw[thick, line width=2pt] ({-\Cx + \R*cos(\Angle)}, {-\R*sin(\Angle)}) -- ({\Cx - \R*cos(\Angle)}, {-\R*sin(\Angle)});
                    \end{scope}
                    % --- Arrow ---
                    \draw[line width=2pt, <->, >=latex] (-0.5, 0) -- (0.5, 0); 
                    % --- Right Figure ---
                    \begin{scope}[shift={(\Sep/2, 0)}]
                        \draw[thick, dashed, line width=2pt] (-\Cx, \R) arc [start angle=90, end angle=270, radius=\R];
                        \draw[thick, dashed, line width=2pt] (\Cx, -\R) arc [start angle=-90, end angle=90, radius=\R];
                        \draw[thick, line width=2pt] (-\Cx, \R) -- (\Cx, \R);
                        \draw[thick, line width=2pt] (-\Cx, -\R) -- (\Cx, -\R);
                    \end{scope}
                \end{tikzpicture}
            }
            \caption{S1-move}
            \label{fig:S1}
        \end{subfigure}
        \hfill
        % ==================================================
        % 右上: S3-move
        % ==================================================
        \begin{subfigure}[b]{0.48\textwidth}
            \centering
            \scalebox{0.4}{
                \begin{tikzpicture}
                    \def\R{1.2} \def\Cx{1.3} \def\Angle{15} \def\Sep{7.0}
                    % --- Left Figure ---
                    \begin{scope}[shift={(-\Sep/2, 0)}]
                        \draw[thick, line width=2pt] (\R, 0) arc [start angle=0, end angle=360, radius=\R];
                        \draw[thick, line width=2pt] ({\R*cos(\Angle)}, {\R*sin(\Angle)}) -- ({\Cx + \R*cos(\Angle)}, {\R*sin(\Angle)});
                        \draw[thick, line width=2pt] ({\R*cos(\Angle)}, {-\R*sin(\Angle)}) -- ({\Cx + \R*cos(\Angle)}, {-\R*sin(\Angle)});
                        \draw[thick, line width=2pt] ({-\R*cos(\Angle)}, {\R*sin(\Angle)}) -- ({-\Cx - \R*cos(\Angle)}, {\R*sin(\Angle)});
                        \draw[thick, line width=2pt] ({-\R*cos(\Angle)}, {-\R*sin(\Angle)}) -- ({-\Cx - \R*cos(\Angle)}, {-\R*sin(\Angle)});
                    \end{scope}
                    % --- Arrow ---
                    \draw[line width=2pt, <->, >=latex] (-0.5, 0) -- (0.5, 0);
                    % --- Right Figure ---
                    \begin{scope}[shift={(\Sep/2, 0)}]
                        \draw[thick, line width=2pt] ({\Cx + \R*cos(\Angle)}, {\R*sin(\Angle)}) -- ({-\Cx - \R*cos(\Angle)}, {\R*sin(\Angle)});
                        \draw[thick, line width=2pt] ({\Cx + \R*cos(\Angle)}, {-\R*sin(\Angle)}) -- ({-\Cx - \R*cos(\Angle)}, {-\R*sin(\Angle)});
                    \end{scope}
                \end{tikzpicture}
            }
            \caption{S3-move}
            \label{fig:S2}
        \end{subfigure}
    
        \vspace{1cm} % 上下の間隔
    
        % ==================================================
        % 左下: S4-move
        % ==================================================
        \begin{subfigure}[b]{0.48\textwidth}
            \centering
            \scalebox{0.4}{
                \begin{tikzpicture}[line cap=round,line join=round]
                    % --- Left Figure ---
                    \begin{scope}[shift={(0,0)}]
                        \def\R{2.2} \def\gap{0.55} \def\off{0.16}
                        % Fill Under
                        \fill[fill=ribbonfill] ($(-\R,\R)+(\off,\off)$) -- ($(-\gap,\gap)+(\off,\off)$) -- ($(-\gap,\gap)+(-\off,-\off)$) -- ($(-\R,\R)+(-\off,-\off)$) -- cycle;
                        \fill[fill=ribbonfill] ($( \gap,-\gap)+(\off,\off)$) -- ($( \R,-\R)+(\off,\off)$) -- ($( \R,-\R)+(-\off,-\off)$) -- ($( \gap,-\gap)+(-\off,-\off)$) -- cycle;
                        % Draw Under lines
                        \draw[line width=1.8pt] ($(-\R,\R)+(\off,\off)$) -- ($(-\gap,\gap)+(\off,\off)$);
                        \draw[line width=1.8pt] ($(-\R,\R)+(-\off,-\off)$) -- ($(-\gap,\gap)+(-\off,-\off)$);
                        \draw[line width=1.8pt] ($( \gap,-\gap)+(\off,\off)$) -- ($( \R,-\R)+(\off,\off)$);
                        \draw[line width=1.8pt] ($( \gap,-\gap)+(-\off,-\off)$) -- ($( \R,-\R)+(-\off,-\off)$);
                        % Draw Over
                        \draw[line width=1.8pt] ($(-\R,-\R)+(-\off,\off)$) -- ($( \R, \R)+(-\off,\off)$);
                        \draw[line width=1.8pt] ($(-\R,-\R)+(\off,-\off)$) -- ($( \R, \R)+(\off,-\off)$);
                    \end{scope}
                    % --- Arrow ---
                    \begin{scope}[shift={(4.0,0)}]
                        \draw[line width=2pt, <->, >=latex] (-0.7, 0) -- (0.7, 0); 
                    % S4 label removed
                    \end{scope}
                    % --- Right Figure ---
                    \begin{scope}[shift={(8.0,0)}]
                        \def\R{2.2} \def\gap{0.55} \def\off{0.16} \def\Cr{0.45}
                        \coordinate (C1) at ($( \R,-\R)!0.5!( \gap,-\gap) $);
                        \coordinate (C2) at ($( -\R,\R)!0.5!( -\gap,\gap) $);
                        \coordinate (C3) at ($( -\R,-\R)!0.5!( -\gap,-\gap) $);
                        \path[name path=circ1] (C1) circle[radius=\Cr];
                        \path[name path=circ2] (C2) circle[radius=\Cr];
                        % Connecting Ribbons
                        \def\AngA{-35} \def\AngB{200} \def\AngC{-50} \def\AngD{220} 
                        \draw[white, line width=5pt] ($(C3)+(\AngA:\Cr)$) to[out=\AngA, in=\AngB] ($(C1)+(\AngB:\Cr)$);
                        \draw[white, line width=5pt] ($(C3)+(\AngC:\Cr)$) to[out=\AngC, in=\AngD] ($(C1)+(\AngD:\Cr)$);
                        \draw[line width=1.8pt] ($(C3)+(\AngA:\Cr)$) to[out=\AngA, in=\AngB] ($(C1)+(\AngB:\Cr)$);
                        \draw[line width=1.8pt] ($(C3)+(\AngC:\Cr)$) to[out=\AngC, in=\AngD] ($(C1)+(\AngD:\Cr)$);
                        \def\AngE{120} \def\AngF{220} \def\AngG{140} \def\AngH{200} 
                        \draw[white, line width=5pt] ($(C3)+(\AngE:\Cr)$) to[out=\AngE, in=\AngF] ($(C2)+(\AngF:\Cr)$);
                        \draw[white, line width=5pt] ($(C3)+(\AngG:\Cr)$) to[out=\AngG, in=\AngH] ($(C2)+(\AngH:\Cr)$);
                        \draw[line width=1.8pt] ($(C3)+(\AngE:\Cr)$) to[out=\AngE, in=\AngF] ($(C2)+(\AngF:\Cr)$);
                        \draw[line width=1.8pt] ($(C3)+(\AngG:\Cr)$) to[out=\AngG, in=\AngH] ($(C2)+(\AngH:\Cr)$);
                        % Main Ribbons
                        \draw[line width=1.8pt] ($(-\R,-\R)+(-\off,\off)$) -- ($(-\gap,-\gap)+(-\off,\off)$);
                        \draw[line width=1.8pt] ($(-\R,-\R)+(\off,-\off)$) -- ($(-\gap,-\gap)+(\off,-\off)$);
                        \draw[line width=1.8pt] ($( \gap, \gap)+(-\off,\off)$) -- ($( \R, \R)+(-\off,\off)$);
                        \draw[line width=1.8pt] ($( \gap, \gap)+(\off,-\off)$) -- ($( \R, \R)+(\off,-\off)$);
                        \path[name path=lineTop] ($( -\R,\R)+(\off,\off)$) -- ($( \R,-\R)+(\off,\off)$);
                        \path[name intersections={of=lineTop and circ2, sort by=lineTop, by={T2a, T2b}}];
                        \path[name intersections={of=lineTop and circ1, sort by=lineTop, by={T1a, T1b}}];
                        \path[name path=lineBot] ($( -\R,\R)+(-\off,-\off)$) -- ($( \R,-\R)+(-\off,-\off)$);
                        \path[name intersections={of=lineBot and circ2, sort by=lineBot, by={B2a, B2b}}];
                        \path[name intersections={of=lineBot and circ1, sort by=lineBot, by={B1a, B1b}}];
                        % Draw Parts
                        \fill[fill=ribbonfill] ($( -\R,\R)+(\off,\off)$) -- (T2a) -- (B2a) -- ($( -\R,\R)+(-\off,-\off)$) -- cycle;
                        \draw[line width=1.8pt, shorten >=-2pt] ($( -\R,\R)+(\off,\off)$) -- (T2a); 
                        \draw[line width=1.8pt, shorten >=-2pt] ($( -\R,\R)+(-\off,-\off)$) -- (B2a); 
                        \coordinate (T_start) at ($(T2b)!-10pt!(T1a)$);
                        \coordinate (T_end)   at ($(T1a)!-10pt!(T2b)$);
                        \coordinate (B_start) at ($(B2b)!-10pt!(B1a)$);
                        \coordinate (B_end)   at ($(B1a)!-10pt!(B2b)$);
                        \draw[line width=1.8pt, fill=ribbonfill] (T_start) -- (T_end) -- (B_end) -- (B_start) -- cycle;
                        \fill[fill=ribbonfill] (T1b) -- ($( \R,-\R)+(\off,\off)$) -- ($( \R,-\R)+(-\off,-\off)$) -- (B1b) -- cycle;
                        \draw[line width=1.8pt, shorten <=-2pt] (T1b) -- ($( \R,-\R)+(\off,\off)$);
                        \draw[line width=1.8pt, shorten <=-2pt] (B1b) -- ($( \R,-\R)+(-\off,-\off)$);
                        % Disks
                        \draw[white, line width=5pt] (C1) ++(-100:\Cr) arc (-100:85:\Cr);
                        \draw[line width=1.8pt] (C1) ++(-180:\Cr) arc (-180:80:\Cr);
                        \draw[white, line width=5pt] (C2) ++(30:\Cr) arc (30:180:\Cr);
                        \draw[line width=1.8pt] (C2) ++(0:\Cr) arc (0:260:\Cr);
                        \filldraw[fill=white, line width=1.8pt] (C3) circle[radius=\Cr];
                    \end{scope}
                \end{tikzpicture}
            }
            \caption{S4-move}
            \label{fig:S4}
        \end{subfigure}
        \hfill
        % ==================================================
        % 右下: S7-move
        % ==================================================
        \begin{subfigure}[b]{0.48\textwidth}
                \scalebox{0.4}{
                \begin{tikzpicture}[thick, >=latex, yscale=2.5]
                    % --- パラメータ設定 ---
                    \def\BigR{15}
                    \def\ArcAng{8}
                    \def\BandW{0.5}
                    \def\BandH{1.8}
                    \def\Tilt{45}
                    \def\Sep{8.5}
                    \def\ArcYOffset{1.5}
                    \def\RibbonAng{3.5}
                
                    % --- Left Figure ---
                    \begin{scope}[shift={(-\Sep/2, \ArcYOffset-\BigR)}]
                        % 弧の輪郭を太く
                        \draw [line width=2pt] ({90-\ArcAng}:\BigR) arc ({90-\ArcAng}:{90+\ArcAng}:\BigR);
                        % 左のバンド
                        \draw[line width=2pt] ({90+\RibbonAng}:\BigR) -- ++({180-\Tilt}:\BandH);
                        \draw[line width=2pt] ({90+\RibbonAng-\BandW/0.25}:\BigR) -- ++({180-\Tilt}:\BandH);
                        % 右のバンド
                        \draw[line width=2pt] ({90-\RibbonAng}:\BigR) -- ++({\Tilt}:\BandH);
                        \draw[line width=2pt] ({90-\RibbonAng+\BandW/0.25}:\BigR) -- ++({\Tilt}:\BandH);
                    \end{scope}
                
                    % --- Arrow ---
                     \draw[line width=2pt, <->, >=latex] (-0.8, \ArcYOffset + 0.8) -- (0.8, \ArcYOffset + 0.8);
                
                    % --- Right Figure ---
                % --- Right Figure ---
                    \begin{scope}[shift={(\Sep/2, \ArcYOffset-\BigR)}]
                        % 弧の輪郭を太く
                        \draw[line width = 2pt] ({90-\ArcAng}:\BigR) arc ({90-\ArcAng}:{90+\ArcAng}:\BigR);
                
                        % 1. 奥にあるバンド（左から右へ）を描画
                        \draw[line width=2pt] ({90+\RibbonAng}:\BigR) -- ++({\Tilt}:\BandH);
                        \draw[line width=2pt] ({90+\RibbonAng-\BandW/0.25}:\BigR) -- ++({\Tilt}:\BandH);
                
                        % 2. 手前のバンドの中身を白塗り
                        % 【修正点】
                        %  - 始点を弧から少し離す (+0.4) ことで、弧の線を消さないようにします。
                        %  - calcライブラリを使って座標を正確に計算します。
                        
                        \fill[white]
                            % 1点目: 右側の根元から少し進んだ点
                            ($ ({90-\RibbonAng}:\BigR) + ({180-\Tilt}:0.4) $) --
                            % 2点目: 右側の先端
                            ++({180-\Tilt}:{\BandH-0.4}) --
                            % 3点目: 左側の先端 (根元座標 + 全長ベクトル)
                            ($ ({90-\RibbonAng+\BandW/0.25}:\BigR) + ({180-\Tilt}:\BandH) $) --
                            % 4点目: 左側の根元から少し進んだ点
                            ($ ({90-\RibbonAng+\BandW/0.25}:\BigR) + ({180-\Tilt}:0.4) $) --
                            cycle;
                
                        % 3. 手前にあるバンドの輪郭線を描画（ここは元のまま全区間引きます）
                        \draw[line width=2pt] ({90-\RibbonAng}:\BigR) -- ++({180-\Tilt}:\BandH);
                        \draw[line width=2pt] ({90-\RibbonAng+\BandW/0.25}:\BigR) -- ++({180-\Tilt}:\BandH);
                    \end{scope}
                \end{tikzpicture} 
                }
            \caption{S7-move}
            \label{fig:S7}
        \end{subfigure}
    
        \caption{S1, S3, S4, S7-moves}
        \label{fig:SMoves}
        \end{figure}
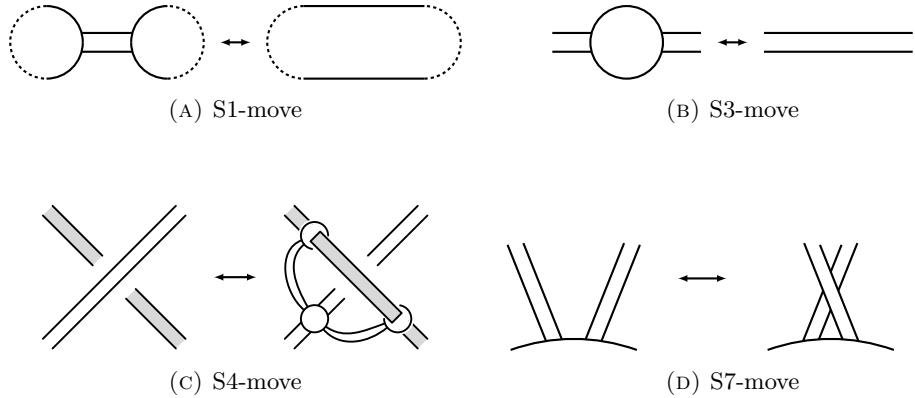
    
\end{Not}

\begin{Lem}
    Let $P$ be a ribbon presentation.
    By applying S1-moves finitely many times, we can arrange that every disk intersects at most one band.
\end{Lem}

\begin{proof}
    We apply the S1-moves as shown in Figure~\ref{fig:s1_move}.

  \begin{figure}[htbp]
        \centering
        \includegraphics[width=0.5\linewidth]{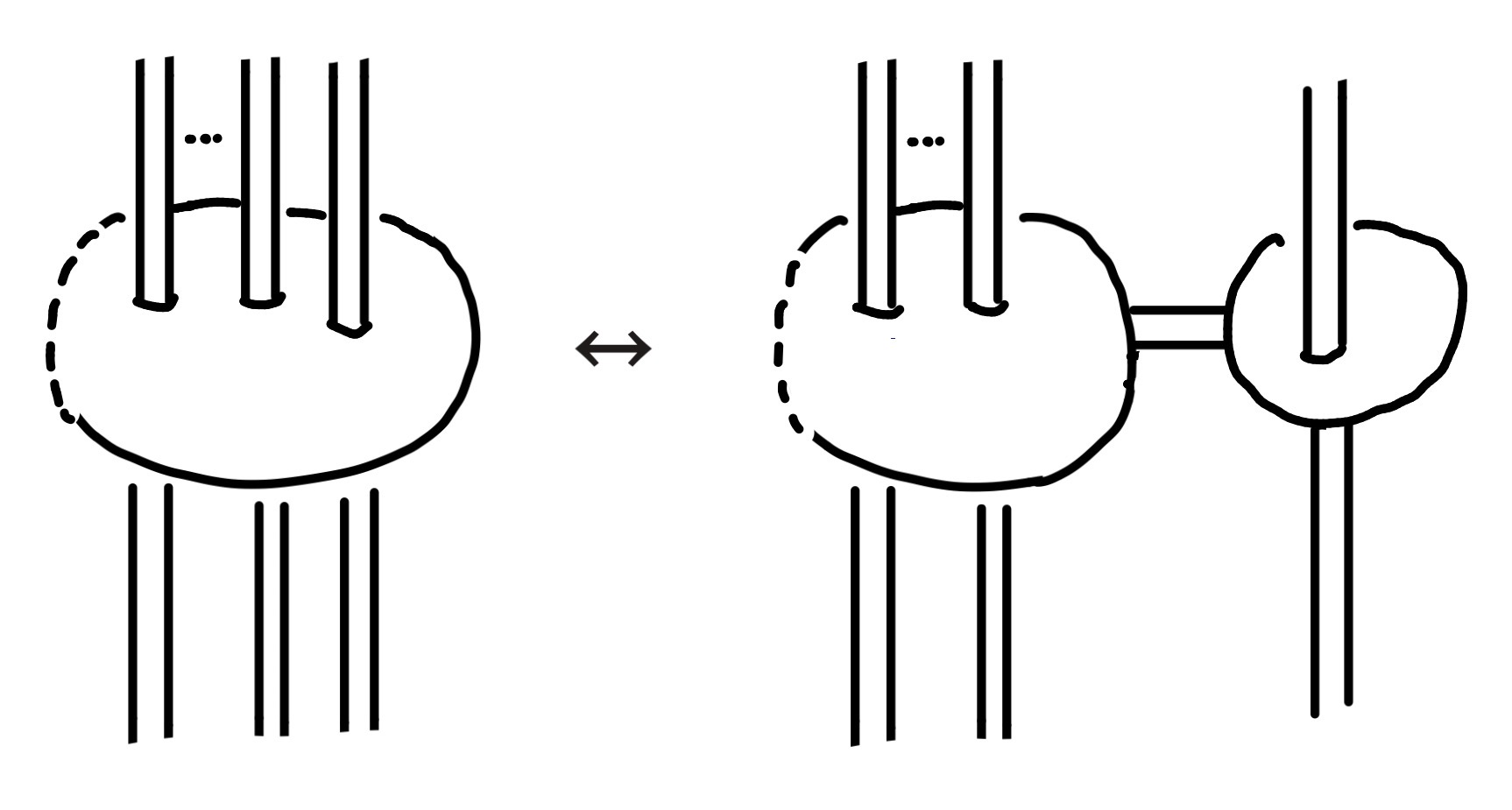}
        \caption{Applying S1-moves as shown in the figure.}
        \label{fig:s1_move}
    \end{figure}
\end{proof}

By this lemma, in this section, we assume that every disk of a ribbon presentation intersects at most one band.
Moreover, near each intersection, we use the local model
\begin{align*}
    B &= \{(x_1, x_2, 0) \mid |x_1| \le 3, |x_2| \le \frac{1}{2} \}, \\
    D &= \{(0,x_2, x_3 ) \mid x_2^2 +x_3^2 \le 1 \}.
\end{align*}

\begin{Not}
    We regard $\R^n$ as $\R^{3} \times \R^{n-j-2} \times \R^{j-1} $.
    Set
    \begin{equation*}
        V_P = \left(\mathcal{B} \times \left[ -\frac{1}{4}, \frac{1}{4} \right]^{j-1} \right) \bigcup \left(\mathcal{D} \times \left[  -\frac{1}{2}, \frac{1}{2}  \right]^{j-1} \right) \subset \R^3 \times \boldsymbol{0} \times \R^{j-1}
    \end{equation*}
    We write the manifold with boundary obtained by rounding the corners of $V_P$ as $V_P'$.
\end{Not}

\begin{Def}
    The long embedding $\varphi_{P}: \R^j \longrightarrow \R^n$ associated with $P$ is defined by the submanifold $(\partial V_P') \# \R^j$. At this stage, $\varphi_P$ is treated primarily as a submanifold; its explicit structure as an embedding map will be specified later in Definition \ref{def of canonical path}.
\end{Def}

Next, we define a family of embeddings obtained from $\varphi_P$. 

\begin{Not}\label{perturbated band}
    We denote by $S$ the sphere $S^{n-j-2}$ defined as
    \begin{equation*}
        S = S^{n-j-2} = \{ (x_3, \dots, x_{n-j+1}) \in \mathbb{R}^{n-j-1} \mid (x_3-1)^2 + \dots + x_{n-j+1}^2 = 1 \}.
    \end{equation*}
    Then, the \emph{perturbed band} $B(v)$ obtained by perturbing $B$ in the direction of $v \in S$ is defined by
    \begin{equation*}
        B(v) = \left\{ (x_1, x_2, \gamma(x_2)v) \in \mathbb{R}^2 \times \mathbb{R}^{n-j-1} \mid |x_1| \le 3, |x_2| < 1/2 \right\}.
    \end{equation*}
    Here, the function $\gamma$ is a smooth test function defined as
    \begin{equation*} 
        \gamma(y) = 
        \begin{cases}
            \exp\left( - \frac{y^2}{\sqrt{9-y^2}} \right) \quad &(|y| < 3), \\
            0 &(|y| \ge 3).
        \end{cases}
    \end{equation*}
\end{Not}

% \begin{Def}
%     Let $ v  \in S^{n-j-2}$ and $1\le i \le k$
%     We define $B_i(v)$ as the bands in $\R^3 \times \R^{n-j-2} \times \boldsymbol{0}$ obtained from bands $B$ having intersection with $Q_i$ by replacing each band $B \times \boldsymbol{0}$ with $B(v)$ at its crossing in $Q_i$.
% \end{Def}

\begin{Def}
    A \emph{labeled ribbon presentation of order $k$} is a pair consisting of a ribbon presentation $P$ and a surjection $\alpha \colon \{ \text{crossings in } P \} \twoheadrightarrow \{ 1, 2, \dots, k \}$.
\end{Def}

\begin{Not}
    Let $\boldsymbol{v}=(v_1,\dots,v_k) \in (S)^k$. We define the perturbed collection of bands $\mathcal{B}_{\boldsymbol{v}}$ as follows: 
    for each band $B$ and for each crossing at $B$ labeled by $i$, we replace the band $B$ with $B(v_i)$ at the crossing.
    We denote the union of all perturbed bands by $\mathcal{B}_{\boldsymbol{v}}$. 
    The \emph{perturbed ribbon presentation} associated with $P$ and $\boldsymbol{v}$ is defined as $P_{\boldsymbol{v}} = \mathcal{D} \cup \mathcal{B}_{\boldsymbol{v}}$. Furthermore, we define the following neighborhood $V_{P_{\boldsymbol{v}}}$ of the image:
    \begin{equation*}
        V_{P_{\boldsymbol{v}}} = \left( \mathcal{B}_{\boldsymbol{v}} \times \left[ -\frac{1}{4}, \frac{1}{4} \right]^{j-1} \right) \cup \left( \mathcal{D} \times \left[ -\frac{1}{2}, \frac{1}{2} \right]^{j-1} \right).
    \end{equation*}
\end{Not}

\begin{Def}
    For $\boldsymbol{v} \in (S)^k$, a long embedding $\varphi_{P_{\boldsymbol{v}}}: \R^j \longrightarrow \R^n $
    is defined as 
    $\varphi_{P_{\boldsymbol{v}}} = \partial V_{P_{\boldsymbol{v}}}' \# \R^j $.
    Here, $V_{P_{\boldsymbol{v}}}'$ is obtained from $V_{P_{\boldsymbol{v}}}$ by smoothing its corners.
    Note that $\varphi_{P_{\boldsymbol{0}}} = \varphi_P $.
\end{Def}

\begin{Def} \label{def of canonical path}
    We define a cycle 
    \[
    c_P: (S^{n-j-2})^{k} 
    \longrightarrow 
    \bEmb_{c}\bigl(\R^j, \R^n \bigr)
    \]
    as follows. 
    Consider the following sequence of resolutions of crossings, 
    which are possible in the space of immersions:
    \begin{itemize}
        \item[($m1$)] 
        Pull each disk $D_{i'}$ corresponding to a crossing 
        with a band of $\mathcal{B}$ 
        in the $x_1$ direction, so that no $D_{i'}$, 
        $B_{i''}$ remain intersecting.
        \item[($m2$)] 
        Pull back $D_{i'}$ and $B_{i''}$ to the base disk $D_0$.
    \end{itemize}
    This operation yields a path from $\varphi_{P_{\boldsymbol{v}}}$ to the trivial embedding $\R^j \longrightarrow \R^n$.  
    Using this path and the standard coordinates of $\R^j$, 
    we equip each point of the submanifold $\varphi_{P_{\boldsymbol{v}}}$ with a coordinate system.  
    The same path also defines a lift
    \[
    \widetilde{\varphi}_{\boldsymbol{v}} : 
    (S^{n-j-2})^k \to \bEmb_{c}(\R^j, \R^n )
    \]
    of ${\varphi}_{\boldsymbol{v}}$.
    We define the cycle associated with $P$ as 
    \[ c_P: (S^{n-j-2})^k \longrightarrow \bEmb_c(\R^j,\R^n), \quad v \mapsto \widetilde{\varphi}_{P_{\boldsymbol v}}  \].
\end{Def}

\subsection{Cycles associated with chord diagrams} \label{subsec:Cycles associated with chord diagrams}

We first construct a modified ribbon cycle associated with a chord diagram.
In \S~\ref{Sec:The cycle is in the image of Hurewicz map}, 
we will reconstruct this generalized ribbon cycle to explicitly demonstrate how the cycle constructed below arises from the homotopy groups of the space of long embeddings.

Let $C$ be a chord diagram of order $k$, and let $\{p_i\}_i$ denote its ordered chords.
Let $g-1$ be the number of chords whose initial points do not lie on the $x$-axis.
Based on these data, in this section we construct a cycle $\psi = \psi(C)$ of degree $k(n-j-2)+(g-1)(j-1)$.

First, we introduce type $\mathrm{I}$ and type $\mathrm{II}$ planetary-like systems, which are analogues of the planetary systems introduced in \cite{Yos25a}.
We fix a sufficiently small parameter $\epsilon$ satisfying $0 < \epsilon < \frac{1}{10000}$.

\begin{Def}[Type $\mathrm{I}$ planetary-like system] \label{def:TypeIPlanetaryLikeSystem}
    A \emph{type $\mathrm{I}$ planetary-like system} $\mathcal{S}_{\mathrm{I}}$ is a subset $D_+ \cup D_- \cup (I \times S) \subset \mathbb{D}^j$, where $D_+$ and $D_-$ are disks of radius $\epsilon^2$ centered at the origin and at $(1/2, 0, \dots, 0)$, respectively.
    The set $I \times S$ is an annulus defined by
    \begin{equation*}
        I \times S = \{ x \in \mathbb{D}^j \mid \tfrac{1}{2}\epsilon \le \|x\| \le \epsilon \}.
    \end{equation*}
    (See Figure~\ref{fig:TypeIPranetaryLikeSystem}.)
\end{Def}

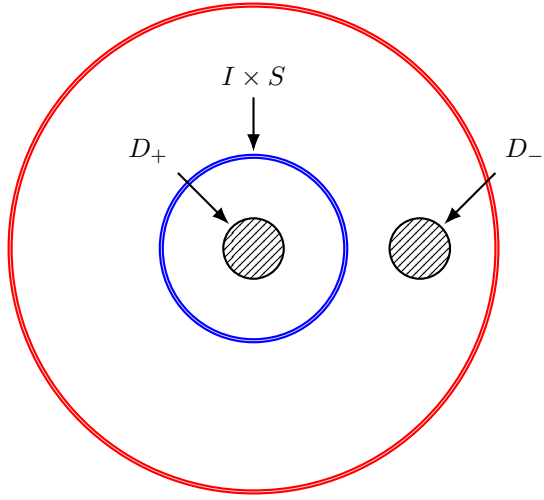
\begin{figure}[h]
    \centering
    \begin{tikzpicture}[
      scale=0.4,
        >=Latex, % 矢印の形状
        % --- スタイル定義 ---
        annulus/.style n args={3}{
            draw=#1, thick, 
            fill=#1!10,
            even odd rule
        },
        hatched disk/.style={
            draw=#1, thick,
            pattern=north east lines,
            pattern color=#1
        }
    ]
    \begin{scope}[shift={(0,0)}]
            % 赤いアニュラス (外側)
            \path[annulus={red}{3.2}{2.4}] (0,0) circle (8.1) (0,0) circle (8.0);
            
            \begin{scope}[shift={(0, 0)}]
                % 青いアニュラス (I x S)
                \path[annulus={blue}{0.9}{0.6}] (0,0) circle (3.1) (0,0) circle (3.0);
                
                % 真ん中のディスク (D_+)
                \path[hatched disk=black] (0,0) circle (1.0);
                \coordinate (CenterT2) at (0,0); 
            \end{scope}
            
            \begin{scope}[shift={(5.5, 0)}]
                % 右のディスク (D_-)
                \path[hatched disk=black] (0,0) circle (1.0);
                \coordinate (CenterR1) at (0,0);
            \end{scope}
    \end{scope}

    % --- 以下、ラベルと矢印の追加 ---
    
    % 1. D_+ (真ん中のディスクへの矢印)
    %    位置: 左上から中心へ
    \draw[->, thick] (-2.5, 2.5) node[above left] {$D_+$} -- (-0.8, 0.8);

    % 2. I x S (青いアニュラスへの矢印)
    %    位置: 上方向から半径3.05の地点へ (角度90度付近)
    \draw[->, thick] (0, 5.0) node[above] {$I \times S$} -- (0, 3.2);

    % 3. D_- (右のディスクへの矢印)
    %    位置: 右上からディスク中心(5.5, 0)の縁へ
    \draw[->, thick] (8.0, 2.5) node[above right] {$D_-$} -- (6.3, 0.8);

\end{tikzpicture}
    \caption{Type~$\mathrm{I}$ planetary-like system.}
    \label{fig:TypeIPranetaryLikeSystem}
\end{figure}

\begin{Def}[Type $\mathrm{II}$ planetary-like system]\label{def:TypeIIPlanetaryLikeSystem}
    A \emph{type $\mathrm{II}$ planetary-like system} $\mathcal{S}_{\mathrm{II}}$ is a subset $D_+ \cup D_- \cup D'_+ \cup D'_- \subset \mathbb{D}^j$, where $D_+, D_-, D'_+$, and $D'_-$ are disks, each with radius $\epsilon^2$. The centers of these disks are given as follows:
    \begin{itemize}
        \item $D_+$: $(0,0, \dots, 0)$
        \item $D_-$: $(-\frac{1}{2}, 0, \dots, 0)$
        \item $D'_+$: $(\epsilon, 0, \dots, 0)$
        \item $D'_-$: $(\frac{1}{2}, 0, \dots, 0)$
    \end{itemize}
\end{Def}

\begin{figure}[h]
    \centering
    \begin{tikzpicture}[
      scale=2.0,
        >=Latex, % 矢印の形状
        % --- スタイル定義 ---
        annulus/.style n args={3}{
            draw=#1, thick, 
            fill=#1!10,
            even odd rule
        },
        hatched disk/.style={
            draw=#1, thick,
            pattern=north east lines,
            pattern color=#1
        }
    ]
    \begin{scope}[shift={(0,0)}]
            % 赤いアニュラス (外枠)
            \path[annulus={red}{0.9}{0.6}] (0,0) circle (1.4) (0,0) circle (1.39);
            
            % 青いアニュラス (内側の軌道)
            % 第1層 (半径0.6付近)
            \path[annulus={blue}{0.9}{0.6}] (0,0) circle (0.54) (0,0) circle (0.5);
            
            % --- パーツ類 ---
            
            % 左の青いdisk (D'_+) : 座標 (0.52, 0)
            \path[hatched disk=blue] (0:0.52) circle (0.1);
            
            % 真ん中のdisk (D_+) : 座標 (0, 0)
            \path[hatched disk=black] (0,0) circle (0.1);
            
            % 左側のdisk (D_-) : 座標 (-1.0, 0)
            \path[hatched disk=black] (-1.0,0) circle (0.1);
            
            % 一番右の青いdisk (D'_-) : 座標 (1.0, 0)
            \path[hatched disk=blue] (1.0,0) circle (0.1);
    \end{scope}
    
    % --- ラベルと矢印の追加 ---
    
    % 1. D_- (左側の黒ディスク)
    \draw[->, thick] (-1.2, 0.4) node[above] {$D_-$} -- (-1.05, 0.1);

    % 2. D_+ (真ん中の黒ディスク)
    \draw[->, thick] (-0.15, 0.4) node[above] {$D_+$} -- (-0.05, 0.1);

    % 3. D'_+ (左の青いディスク / 内側)
    \draw[->, thick] (0.52, 0.4) node[above] {$D'_+$} -- (0.52, 0.13);

    % 4. D'_- (一番右の青いディスク / 外側)
    \draw[->, thick] (1.2, 0.4) node[above] {$D'_-$} -- (1.05, 0.1);

    \end{tikzpicture}
    \caption{Type~$\mathrm{II}$ planetary-like system.}
    \label{fig:TypeIIPlanetaryLikeSystem}
\end{figure}
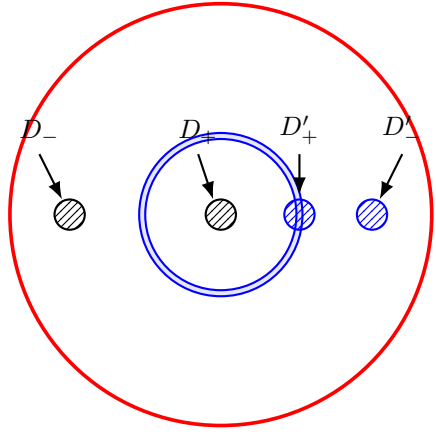

A planetary-like system is defined as the composition of Type $\mathrm{I}$ and Type $\mathrm{II}$ planetary-like systems.

% =========================================================
% Definition: Monostellar System
% =========================================================
\begin{Def}[Monostellar planetary-like system]
    The class of \emph{monostellar planetary-like systems} is defined inductively as follows.

    \begin{enumerate}
        \item (Base cases) Type $\mathrm{I}$ and Type $\mathrm{II}$ systems are monostellar planetary-like systems.

        For a system $\mathcal{T}$ of type $T \in \{\mathrm{I}, \mathrm{II}\}$, we fix the following \emph{structure embeddings} $e^{T}_+, e^{T}_- : \mathbb{D}^j \to \mathbb{D}^j$:
        \begin{itemize}
            \item $e^{T}_+$: The standard orientation-preserving embedding onto the disk $D_+ \subset \mathcal{T}$.
            \item $e^{T}_-$: The composition of the reflection across the hyperplane $x_1=0$ with the standard embedding onto the disk $D_- \subset \mathcal{T}$.
        \end{itemize}

        \item (Composition) If $\mathcal{S}$ is a monostellar planetary-like system and $\mathcal{S}'$ is a system of type $T \in \{\mathrm{I}, \mathrm{II}\}$, their \emph{composition} $\mathcal{S} \circ \mathcal{S}'$ is defined as the union of the images:
        \begin{equation*}
            \mathcal{S} \circ \mathcal{S}' = e^{T}_+(\mathcal{S}) \cup e^{T}_-(\mathcal{S}) \cup (\mathcal{S}' \setminus (D_+ \cup D_-)).
        \end{equation*}
        (Note that this definition implies that $\mathcal{S}$ is embedded into $D_+$ standardly and into $D_-$ with reflection).
    \end{enumerate}

    For iterated compositions, we omit parentheses and write $\mathcal{S}_1 \circ \mathcal{S}_2 \circ \dots \circ \mathcal{S}_k$.
\end{Def}

\begin{Rem}
    The order of composition defined here is \emph{opposite to} the standard convention for the little disks operad.
    Our convention is chosen to be consistent with the composition of ribbon presentations (see Section \ref{Sec:The cycle is in the image of Hurewicz map}).
\end{Rem}

 \begin{figure}[htbp]
        \centering
        \begin{tikzpicture}[
            scale=0.8, >=Latex,
            annulus/.style n args={3}{draw=#1, thick, fill=#1!10, even odd rule},
            hatched disk/.style={draw=#1, thick, pattern=north east lines, pattern color=#1}
        ]
            % --- Definitions of contents ---
            \newcommand{\InnerContentSII}{
                \path[annulus={blue}{}{}] (0,0) circle (0.6) (0,0) circle (0.5);
                \path[hatched disk=blue] (30:0.55) circle (0.15);
                \path[hatched disk=black] (0,0) circle (0.2);
                \path[hatched disk=black] (-1.0,0) circle (0.2);
                \path[hatched disk=blue] (1.0,0) circle (0.2);
            }
            \newcommand{\MiddleContentSI}{
                \path[annulus={blue}{}{}] (0,0) circle (3.1) (0,0) circle (3.0);
                \begin{scope}[shift={(0,0)}, scale=1.96] \InnerContentSII \end{scope}
                \begin{scope}[shift={(5.5,0)}, scale=1.41, xscale=-1] \InnerContentSII \end{scope}
            }
            
            % --- Main Drawing ---
            \begin{scope}[shift={(0,0)}]
                \path[annulus={red}{}{}] (0,0) circle (8.1) (0,0) circle (8.0);
                \path[annulus={blue}{}{}] (0,0) circle (3.1) (0,0) circle (3.0);
                \begin{scope}[shift={(0,0)}, scale=0.4] \MiddleContentSI \end{scope}
                \begin{scope}[shift={(5.5, 0)}, scale=0.3, xscale=-1] \MiddleContentSI \end{scope}
            \end{scope}
            
            \node[below] at (0, -8.5) {$\mathcal{S}_{\mathrm{II}} \circ \mathcal{S}_{\mathrm{I}} \circ \mathcal{S}_{\mathrm{I}}$};
        \end{tikzpicture}
        \caption{Example of compositions of planetary-like systems. The red circle is the boundary of $\mathbb{D}^j$.}
        \label{fig:exampleOfCompositionOfPlanetaryLikeSystem}
    \end{figure}
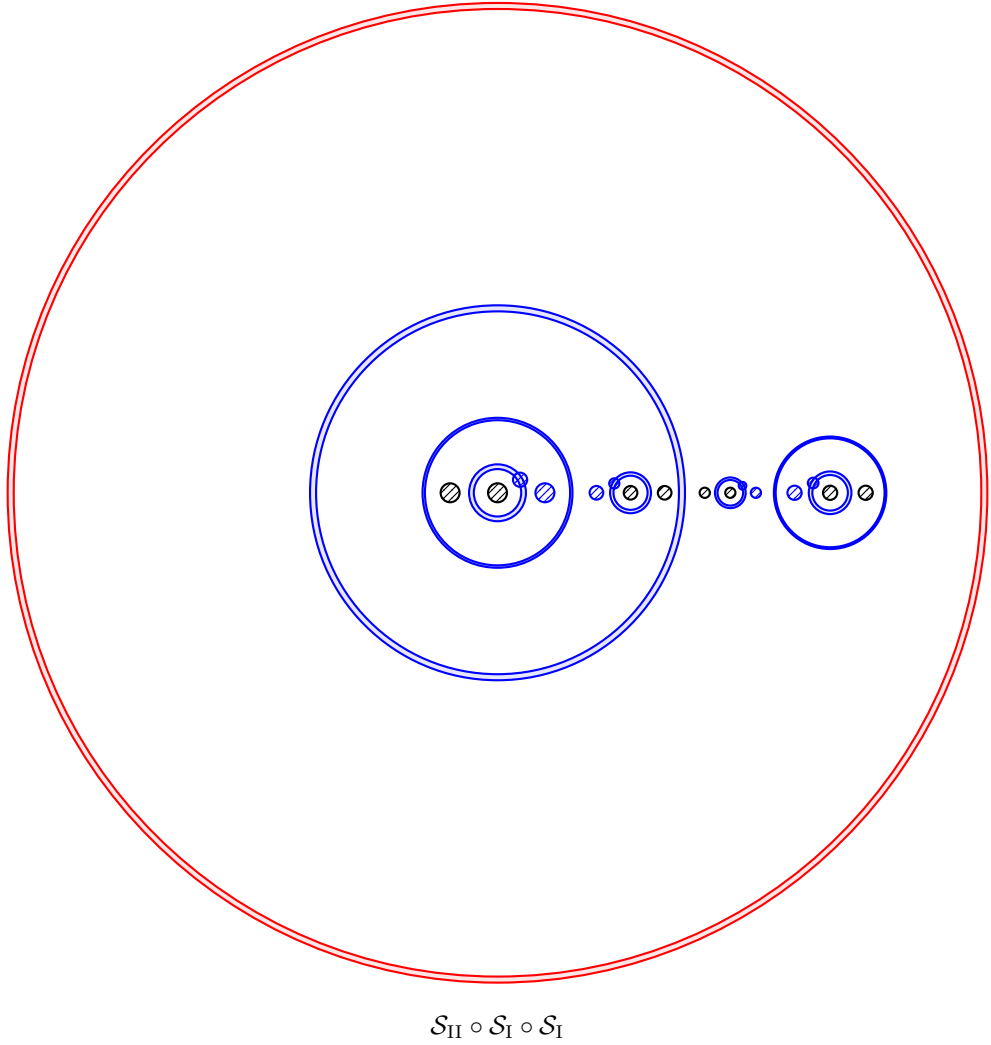
% =========================================================
% Definition: System associated with Chord Diagram
% =========================================================
\begin{Def}[Planetary-like system associated with a chord diagram]
    Let $C$ be a chord diagram on $s$ directed lines. For each line $i$ ($1 \le i \le s$) and each vertex index $l$ with $1 \le l \le t_i $, we assign a monostellar planetary-like system $\mathcal{S}_{i,l}$ as follows:
    \begin{itemize}
        \item If $(i,l)$ is the second element of a pair in $C$ (i.e., it is the endpoint of an ingoing chord), we set $\mathcal{S}_{i,l} = \mathcal{S}_{\mathrm{I}}$.
        \item Otherwise, i.e., if it is the endpoint of an outgoing chord, we set $\mathcal{S}_{i,l} = \mathcal{S}_{\mathrm{II}}$.
    \end{itemize}
    We define the composite system $\mathcal{S}_{i}$ for the $i$-th line by:
    \begin{equation*}
        \mathcal{S}_{i} = \mathcal{S}_{i,1} \circ \dots \circ \mathcal{S}_{i, t_i}.
    \end{equation*}

    Now, for each $i \in \{1, \dots, s\}$, let $e_i: \mathbb{D}^j \to \mathbb{R}^j$ be the embedding onto the disk of radius $\epsilon^2$ centered at $c_i = (i, 0, \dots, 0)$, defined by $e_i(x) = \epsilon^2 x + c_i$.

    The \emph{planetary-like system} $\mathcal{S}_C$ associated with $C$ is defined as:
    \begin{equation*}
        \mathcal{S}_C = \bigcup_{i=1}^{s} e_i(\mathcal{S}_i).
    \end{equation*}
\end{Def}

Next, we introduce a rotation of the disk $D'_+$ in the Type $\mathrm{II}$ planetary-like system.

\begin{Def}[Rotated Type $\mathrm{II}$ planetary-like system]
    Let $\theta \in S^{j-1}$. The \emph{rotated type $\mathrm{II}$ planetary-like system}, denoted by $\mathcal{S}_{\mathrm{II}}(\theta)$, is defined by replacing the disk $D'_+$ in $\mathcal{S}_{\mathrm{II}}$ with a disk $D'_+(\theta)$ of the same radius centered at $\epsilon \theta$.
\end{Def}

Compositions of rotated systems are defined in the same manner.
We construct a family of planetary-like systems parameterized by $(S^{j-1})^{\times (g-1)}$ by replacing the Type $\mathrm{II}$ system associated with the $i$-th chord (specifically, for chords whose initial points are not on the $x$-axis) with the rotated system $\mathcal{S}_{\mathrm{II}}(\theta_i)$.

% =========================================================
% Definition: Geometric Components (S_alpha, T_alpha)
% =========================================================
\begin{Def}
    Let $C$ be a chord diagram of order $k$ and $1 \le \alpha \le k$.
    Let $(i, l)$ and $(i', l')$ be the terminal and initial points of the $\alpha$-th chord, respectively.

    We define the geometric components associated with the $\alpha$-th chord as follows:
    \begin{itemize}
        \item The target annulus $S_{\alpha}$ is the image of $I \times S$ corresponding to the terminal point $(i, l)$, (i.e. a union of $2^{t_i-l}$-annuli  $(e_+ \cup e_-)^{t_i-l}(I \times S) $)
        \item The source disk $T_{\alpha}$ is the image of $D'_+ \cup D'_-$ corresponding to the initial point $(i', l')$ (i.e. $(e_+ \cup e_-)^{t_i-l}(D_+' \cup D_-')$) if $l >0$, and if $l=0$, then $T_{\alpha}$ is $(e_{+}\cup e_{-})^{t_i}(D_+ \cup D_-)$.
    \end{itemize}

    We classify the components of these sets into two types:
    \begin{itemize}
        \item \textbf{Primary part:} The component constructed exclusively by iterations of the map $e_+$ (i.e., the "backbone" of the tree structure).
        For example, the primary part of $S_\alpha$ is $e_i(e_+^{t_i-l}(I \times S))$.
        \item \textbf{Secondary part:} Any component involving at least one map $e_-$. These correspond to the "branches" created by the planetary-like system.
    \end{itemize}
\end{Def}

We construct a ribbon cycle that is compatible with the planetary-like system.

\begin{Mysec}[Disk-and-band systems associated with Type $\rI$ and Type $\rII$ planetary-like systems]
   Let $\mathcal{R}_{\rI}$ (resp. $\mathcal{R}_{\rII}$) denote a \emph{disk-and-band system} embedded in $\R^2$.
    Each system consists of two bands and three disks, as illustrated in Figure~\ref{fig:TypeIDisksAndBandSystem} (resp. Figure~\ref{fig:TypeIIDisksAndBandSystem}).
    
    Note that the components of these disk-and-band systems correspond to those of the planetary-like systems $\mathcal{S}_{\rI}$ and $\mathcal{S}_{\rII}$ (Definitions \ref{def:TypeIPlanetaryLikeSystem} and \ref{def:TypeIIPlanetaryLikeSystem}) as follows:
    \begin{itemize}
        \item The disks labeled $D_+^2$ and $D_-^2$ (and $D'^2_+, D'^2_-$ for Type $\rII$) correspond to the respective disks in the planetary-like systems.
        The red arrow in Figure~\ref{fig:TypeIIDisksAndBandSystem} depicts the rotation of $D'_+$ around $D_+$ in the Type~$\rII$ planetary-like system $\mathcal{S}_{\rII}$.
        
        \item The band $B_{\rI}$ in Type $\rI$ corresponds to the annulus $I \times S$.
    \end{itemize}
    
    Furthermore, note that one end of the band $B_0$ is free (i.e., not attached to any disk), serving as a connection point for composition.
    
   \begin{figure}[htbp]
    \centering
    % ==================================================
    % 左: Type I
    % ==================================================
    \begin{subfigure}[b]{0.48\textwidth}
        \centering
        \scalebox{0.4}{
            \begin{tikzpicture}
                % --- パラメータ定義 ---
                \def\R{1.2}      % 円の半径
                \def\Dist{3.8}   % 円の中心間の距離
                \def\HBA{18}     % Half Band Angle
                \def\LW{2pt}     % 線の太さ

                \tikzset{main/.style={thick, line width=\LW}}

                % --- 中央の円 ---
                \draw[main] (0, 0) circle (\R);

                % --- 右上の円 (D_-) ---
                \coordinate (C1) at (45:\Dist);
                \begin{scope}[shift={(C1)}]
                    \draw[main] (0, 0) circle (\R);
                \end{scope}
                \node[font=\large] at ($(C1) + (45:\R+0.6)$) {$D^2_-$};

                % --- 左上の円 (D_+) ---
                \coordinate (C2) at (135:\Dist);
                \begin{scope}[shift={(C2)}]
                    \draw[main] (0, 0) circle (\R);
                \end{scope}
                \node[font=\large] at ($(C2) + (135:\R+0.6)$) {$D^2_+$};

                % --- バンド ---
                \coordinate (C0) at (-90:\Dist);
                \draw[main] (-90-\HBA:\R) -- ($(C0) + (90+\HBA:\R)$);
                \draw[main] (-90+\HBA:\R) -- ($(C0) + (90-\HBA:\R)$);
                \node[font=\large, below] at ($(C0) +(90:\R) $) {$B_0$};

                \draw[main] (45-\HBA:\R) -- ($(C1) + (225+\HBA:\R)$);
                \draw[main] (45+\HBA:\R) -- ($(C1) + (225-\HBA:\R)$);
                \draw[main] (135+\HBA:\R) -- ($(C2) + (315-\HBA:\R)$);
                \draw[main] (135-\HBA:\R) -- ($(C2) + (315+\HBA:\R)$);

                % \emph{ここを添削}: 左側のバンドにラベルを追加
                \node[font=\large] at ($(0,0)!0.5!(C2)$) {$B_{\mathrm{I}}$};

            \end{tikzpicture}
        }
        \caption{$\mathcal{R}_{\mathrm{I}}$}
        \label{fig:TypeIDisksAndBandSystem}
    \end{subfigure}
    \hfill % 図の間隔を調整
    % ==================================================
    % 右: Type II
    % ==================================================
    \begin{subfigure}[b]{0.48\textwidth}
        \centering
        \scalebox{0.4}{
            \begin{tikzpicture}
                % --- パラメータ定義 ---
                \def\R{0.8}      % 円の半径
                \def\Dist{3.8}   % 円の中心間の距離
                \def\HBA{18}     % Half Band Angle (バンド幅の半分)
                \def\LW{2pt}     % 線の太さ
                \def\BandLen{2.5} % 下に伸びるバンドの長さ
                \def\ArR{1.1}    % 矢印の回転半径 (円より少し外側)

                \tikzset{main/.style={thick, line width=\LW}}

                % --- 中央の円 ---
                \draw[main] (0, 0) circle (\R);

                % --- 周囲のDiskとバンド (ループ処理) ---
                % \emph{ここを添削}: 数式部分を {} で囲むことでエラーを回避し、^2 を追加
                \foreach \ang/\lab/\labpos in {
                    15/{D'^2_-}/15,
                    60/{D'^2_+}/60,
                    120/{D^2_+}/120,
                    165/{D^2_-}/165
                } {
                    \coordinate (C) at (\ang:\Dist);
                    
                    % Disk描画
                    \begin{scope}[shift={(C)}]
                        \draw[main] (0, 0) circle (\R);
                    \end{scope}

                    % バンド描画
                    \draw[main] (\ang-\HBA:\R) -- ($(C) + (\ang+180+\HBA:\R)$);
                    \draw[main] (\ang+\HBA:\R) -- ($(C) + (\ang+180-\HBA:\R)$);

                    % ラベル配置
                    \node[font=\large] at ($(C) + (\labpos:\R+0.6)$) {$\lab$};
                }

                % --- 下に伸びるバンド (B_0) ---
                \draw[main] (270-\HBA:\R) -- ++(0, -\BandLen);
                \draw[main] (270+\HBA:\R) -- ++(0, -\BandLen);
                \node[font=\large] at (0, -\R-\BandLen-0.6) {$B_0$};

                % ==================================================
                % 矢印 (Wrapping Arrow)
                % ==================================================
                \draw[thick,red, line width=1.5pt] (75:\ArR) arc [start angle=75, end angle=100, radius=\ArR];
                \draw[thick,red, ->, >=latex, line width=1.5pt] (140:\ArR) to[out=210, in=-50, looseness=1.2] (55:\ArR);

            \end{tikzpicture}
        }
        \caption{$\mathcal{R}_{\mathrm{II}}$}
        \label{fig:TypeIIDisksAndBandSystem}
    \end{subfigure}

    \caption{}
    \label{fig:DisksAndBands}
\end{figure}
\end{Mysec}

Let $\mathcal{R}'$ be a disk-and-band system of Type $\rI$ or Type $\rII$.
The composition $\mathcal{R} \circ \mathcal{R}'$ is defined by attaching $\mathcal{R}$ to the disk $D^2_+$ in $\mathcal{R}'$ via the band $B_0$, and attaching the reflection of a copy of $\mathcal{R}$ to $D^2_-$.

\begin{figure}[htbp]
    \centering
    \scalebox{0.45}{ 
        \begin{tikzpicture}
            % ==================================================
            % パラメータ定義
            % ==================================================
            \def\RadRoot{1.2}      % Level 0 (Root) 半径
            \def\RadChild{0.8}     % Level 1 (Child) 半径
            \def\RadGrand{0.5}     % Level 2 (Grandchild) 半径
            
            \def\DistRoot{4.5}     % Root -> Child の距離
            \def\DistChild{2.8}    % Child -> Grandchild の距離
            
            \def\HBA{18}           % Half Band Angle
            \def\LW{2pt}           % 線幅
            
            \def\ScaleChild{0.7}   % Level 1 (Child) の縮小率
            \pgfmathsetmacro{\RealRadChild}{\RadChild * \ScaleChild}

            % --- カラー定義 ---
            \definecolor{ColorDprime}{RGB}{180, 0, 0}
            \definecolor{ColorDouter}{RGB}{180, 230, 255}
            \definecolor{ColorBandRootLeft}{RGB}{180, 0, 0}
            \definecolor{ColorBandMid}{RGB}{180, 230, 255}
            \definecolor{LabelRed}{RGB}{255, 0, 0}
            \definecolor{LabelBlue}{RGB}{0, 0, 255}

            \tikzset{main/.style={thick, line width=\LW}}

            % ==================================================
            % ラベルセット定義
            % ==================================================
            \newcommand{\SetLabelsLL}{\def\LabA{8}\def\LabB{1}\def\LabC{1}\def\LabD{8}}
            \newcommand{\SetLabelsLR}{\def\LabA{5}\def\LabB{4}\def\LabC{4}\def\LabD{5}}
            \newcommand{\SetLabelsRR}{\def\LabA{6}\def\LabB{3}\def\LabC{3}\def\LabD{6}}
            \newcommand{\SetLabelsRL}{\def\LabA{7}\def\LabB{2}\def\LabC{2}\def\LabD{7}}

            % ==================================================
            % [部品1] Level 2: Type II System (孫)
            % ==================================================
            \newcommand{\DrawLevelTwoTypeII}{
                % 本体中心円
                \draw[main, fill=white] (0, 0) circle (\RadGrand);
                
                % --- D'+, D'- (赤) ---
                \foreach \ang/\labmac in {15/\LabA, 60/\LabB} {
                    \coordinate (Pos) at (\ang:\DistChild);
                    \draw[main] (\ang-\HBA:\RadGrand) -- ($(Pos) + (\ang+180+\HBA:\RadGrand)$);
                    \draw[main] (\ang+\HBA:\RadGrand) -- ($(Pos) + (\ang+180-\HBA:\RadGrand)$);
                    \begin{scope}[shift={(Pos)}]
                        \draw[main, fill=ColorDprime] (0, 0) circle (\RadGrand);
                        \node[font=\bfseries\small, text=LabelRed, anchor=center] at (\ang:\RadGrand+0.35) {\labmac};
                    \end{scope}
                }
                
                % --- D+, D- (青) ---
                \foreach \ang/\labmac in {120/\LabC, 165/\LabD} {
                    \coordinate (Pos) at (\ang:\DistChild);
                    \draw[main] (\ang-\HBA:\RadGrand) -- ($(Pos) + (\ang+180+\HBA:\RadGrand)$);
                    \draw[main] (\ang+\HBA:\RadGrand) -- ($(Pos) + (\ang+180-\HBA:\RadGrand)$);
                    \begin{scope}[shift={(Pos)}]
                        \draw[main, fill=ColorDouter] (0, 0) circle (\RadGrand);
                        
                        % ラベル位置の微調整 (1と7は外側にずらす)
                        \def\LabelDist{0.35}
                        \ifnum\labmac=1 \def\LabelDist{0.6} \fi
                        \ifnum\labmac=7 \def\LabelDist{0.6} \fi
                        
                        \node[font=\bfseries\small, text=LabelBlue, anchor=center] at (\ang:\RadGrand+\LabelDist) {\labmac};
                    \end{scope}
                }

                % 矢印
                \def\LocalArR{\RadGrand * 1.5}
                \draw[thick,red, line width=1.5pt] (75:\LocalArR) arc [start angle=75, end angle=100, radius=\LocalArR];
                \draw[thick,red, ->, >=latex, line width=1.5pt] (140:\LocalArR) to[out=210, in=-50, looseness=1.2] (55:\LocalArR);
            }

            % ==================================================
            % [部品2] Level 1: Type I System (子)
            % ==================================================
            \newcommand{\DrawLevelOneComplex}[2]{ 
                \draw[main, fill=white] (0, 0) circle (\RadChild);

                % --- Left Branch ---
                \coordinate (LTwoLeft) at (135:\DistChild);
                \fill[ColorDouter] (135-\HBA:\RadChild) 
                                  -- ($(LTwoLeft) + (315+\HBA:\RadGrand)$)
                                  -- ($(LTwoLeft) + (315-\HBA:\RadGrand)$)
                                  -- (135+\HBA:\RadChild) -- cycle;
                \draw[main] (135-\HBA:\RadChild) -- ($(LTwoLeft) + (315+\HBA:\RadGrand)$);
                \draw[main] (135+\HBA:\RadChild) -- ($(LTwoLeft) + (315-\HBA:\RadGrand)$);

                \begin{scope}[shift={(LTwoLeft)}, rotate=45]
                    #1 
                    \DrawLevelTwoTypeII
                \end{scope}

                % --- Right Branch ---
                \begin{scope}[xscale=-1]
                    \coordinate (LTwoRight) at (135:\DistChild);
                    \draw[main] (135-\HBA:\RadChild) -- ($(LTwoRight) + (315+\HBA:\RadGrand)$);
                    \draw[main] (135+\HBA:\RadChild) -- ($(LTwoRight) + (315-\HBA:\RadGrand)$);
                    \begin{scope}[shift={(LTwoRight)}, rotate=45]
                        #2 
                        \DrawLevelTwoTypeII
                    \end{scope}
                \end{scope}
                \draw[main, fill=white] (0, 0) circle (\RadChild);
            }

            % ==================================================
            % [Main] Level 0: Root
            % ==================================================
            \draw[main, fill=white] (0, 0) circle (\RadRoot);
            \coordinate (RootC0) at (-90:\DistRoot);
            \draw[main] (-90-\HBA:\RadRoot) -- ($(RootC0) + (90+\HBA:\RadRoot)$);
            \draw[main] (-90+\HBA:\RadRoot) -- ($(RootC0) + (90-\HBA:\RadRoot)$);
            \node[font=\Huge] at ($(RootC0) + (0, -0.8)$) {$B_0$};

            % Attach Left
            \coordinate (LOneLeft) at (135:\DistRoot);
            \fill[ColorDprime] (135-\HBA:\RadRoot)
                                   -- ($(LOneLeft) + (315+\HBA:\RealRadChild)$)
                                   -- ($(LOneLeft) + (315-\HBA:\RealRadChild)$)
                                   -- (135+\HBA:\RadRoot) -- cycle;
            \draw[main] (135-\HBA:\RadRoot) -- ($(LOneLeft) + (315+\HBA:\RealRadChild)$);
            \draw[main] (135+\HBA:\RadRoot) -- ($(LOneLeft) + (315-\HBA:\RealRadChild)$);
            \begin{scope}[shift={(LOneLeft)}, rotate=45, scale=\ScaleChild]
                \DrawLevelOneComplex{\SetLabelsLL}{\SetLabelsLR}
            \end{scope}

            % Attach Right
            \begin{scope}[xscale=-1]
                \coordinate (LOneRight) at (135:\DistRoot);
                \draw[main] (135-\HBA:\RadRoot) -- ($(LOneRight) + (315+\HBA:\RealRadChild)$);
                \draw[main] (135+\HBA:\RadRoot) -- ($(LOneRight) + (315-\HBA:\RealRadChild)$);
                \begin{scope}[shift={(LOneRight)}, rotate=45, scale=\ScaleChild]
                    % \emph{ここを添削}: 引数の順序を修正
                    % xscale=-1 により、#1(Left Branch)が視覚的な右側(外側)へ、
                    % #2(Right Branch)が視覚的な左側(内側)へ移動するため、
                    % 外側用(RL)を第1引数、内側用(RR)を第2引数に渡す。
                    \DrawLevelOneComplex{\SetLabelsRL}{\SetLabelsRR}
                \end{scope}
            \end{scope}
            \draw[main, fill=white] (0, 0) circle (\RadRoot);

        \end{tikzpicture}
    }
    \caption{$\mathcal{R}_{\mathrm{II}} \circ \mathcal{R}_{\mathrm{I}} \circ \mathcal{R}_{\mathrm{I}}$ with labeled disks}
    \label{fig:ExampleOfCompositionsOfDyskBandsSystem}
\end{figure}
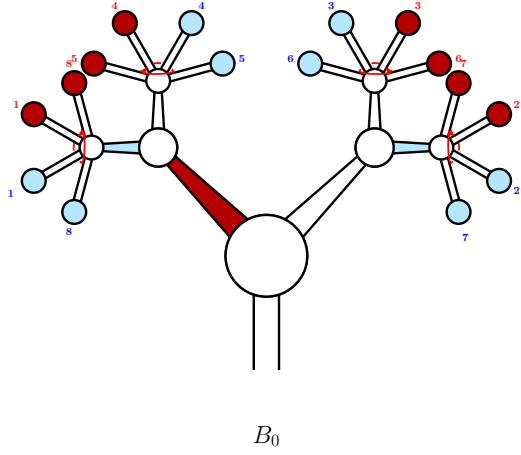

\begin{Con} \label{ConstructionOfCycleAssociatedWithChordDiagram}
    Let $C$ be a chord diagram.
    We define $c_C \colon (S^{j-1})^{g-1} \times (S^{n-j-2})^k \to \bEmb_c(\R^j, \R^n)$ as follows.

    We assign Type~$\rI$ or Type~$\rII$ to each vertex in $V(C)$, excluding the vertices lying on the $x$-axis.
    Analogously to the base cases, we define a composed system of embedded disks and bands, denoted by $(\mathcal{R})_{i}$, by composing Type~$\rI$ or Type~$\rII$ systems along the $i$-th line.
    The total system $\mathcal{R}$ is obtained by attaching $(\mathcal{R})_i$ to the base disk $D_0$ via $B_0$.

    Next, we construct a ribbon crossing in $\R^3$.
    Fix a chord $\alpha$.
    We consider the disks corresponding to the $j$-disks in $T_{\alpha}$.
    We arrange these disks so that they form crossings with the family of bands corresponding to $S_{\alpha}$ such that the disk $e_{s_i} e_{s_{i-1}}\dots e_{s_2} (D'_{s_1})$ (or $e_{s_i} e_{s_{i-1}}\dots e_{s_2} (D_{s_1})$) intersects these bands with sign $(-1)^{\# \{ a \mid s_a =- \}}$.
    These crossings are labeled by $\alpha$.
    Finally, by rotating the stems originating from $D_{+}'$ for each Type~$\rII$ planetary-like system, we obtain a family of embeddings $(S^{j-1})^{g-1} \times (S^{n-j-2})^k \to \bEmb_c(\R^j, \R^n)$.

    When $n-j = 2$, we must specify the relative ordering of intersections near the root of the bands.
    Let $D_1 = e_{s_i}\dots e_{s_2}(D'_{s_{1}})$ and $D_2 = e_{s'_i}\dots e_{s'_2}(D'_{s'_{1}})$ be distinct disks in $T_{\alpha}$ with indices $s_a, s'_a \in \{ +, - \}$. 

    Suppose that the index sequences coincide for all $a < b$ (i.e., $s_{a} = s'_a$), and differ at index $b$ with $s_b = -$ and $s'_b = +$.
    Then, $D_1$ intersects the bands closer to the root than $D_2$ does if and only if the number of negative indices preceding $b$ is odd, i.e.,
    \[
        \# \{a < b \mid s_a = - \} \equiv 1 \pmod{2}.
    \]
\end{Con}

The cycle constructed in Construction~\ref{ConstructionOfCycleAssociatedWithChordDiagram} has the following property, which we will use in \S~\ref{sec:Non-triviality} to show the nontriviality of the cycle.

\begin{Lem}
    The map $c_C \colon (S^{j-1})^{g-1} \times (S^{n-j-2})^k \to \bEmb_c(\R^j, \R^n)$ satisfies the following conditions.

    \begin{itemize}
        \item For each $\alpha$, the preimage of the neighborhood of the ribbon crossings labeled by $\alpha$ is contained in $S_{\alpha} \cup T_{\alpha}$.
        Moreover, the preimage of the boundary of the fattened disks of the crossings is contained in $T_{\alpha}$, and the preimage of the boundary of the thickened bands of the crossings is contained in $S_{\alpha}$.
    \end{itemize}
\end{Lem}

\begin{Rem}
    The primary part (obtained by removing all secondary parts) recovers the cycle constructed in \cite[Section~6.2]{Yos25a}.
\end{Rem}

\begin{Exa}
    Here is an example of a chord diagram $C_0$, shown in Figure~\ref{fig:TheChordDiagramConstructionOfCycle}.
    The associated planetary-like system is shown in Figure~\ref{fig:exampleOfCompositionOfPlanetaryLikeSystem} and the associated ribbon presentation is shown in Figure~\ref{fig:ExampleOfCompositionsOfDyskBandsSystem}.
    In Figure~\ref{fig:ExampleOfCompositionsOfDyskBandsSystem}, disks and bands of the same color intersect, and the disks with smaller labels intersect closer to the roots of bands.  
       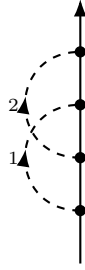
\begin{figure}[h]
        \centering
        \begin{tikzpicture}[
            xscale=0.5, yscale=0.7, >=Latex,
            mid arrow/.style={postaction={decorate,decoration={markings, mark=at position 0.55 with {\arrow{>}}}}},
            dot/.style={circle, fill=black, inner sep=1.5pt},
            dashed edge/.style={thick, dashed},
            lbl/.style={font=\scriptsize, inner sep=2pt}
        ]
            \begin{scope}[shift={(0,0)}]
                \coordinate (m) at (0,-1);
                \coordinate (Top) at (0, 4);
                \coordinate (O) at (0, 0);
                \coordinate (P1) at (0, 1);
                \coordinate (P2) at (0, 2); 
                \coordinate (P3) at (0, 3);
                
                \draw[->, thick] (m) -- (Top);
                \draw[dashed edge, mid arrow] (P1) to[out=180, in=180, looseness=2.5] node[pos=0.5, left, lbl] {$2$} (P3);
                \draw[dashed edge, mid arrow] (O) to[out=180, in=180, looseness=2.5] node[pos=0.5, left, lbl] {$1$} (P2);
                
                \foreach \p in {O,P1, P2, P3} { \node[dot] at (\p) {}; }
            \end{scope}
        \end{tikzpicture}
        \caption{The chord diagram $C_0$}
        \label{fig:TheChordDiagramConstructionOfCycle}
    \end{figure}
\end{Exa}

\subsection{Infiniteness of the top term of HGC}\label{subsec:Infiniteness_of_HGC}

We write $^{*}HGC_{n,j}$ for the corresponding chain complex
\footnote{ Let \(\Gamma\) and \(\Gamma'\) be two identical labeled admissible hairy graphs
with no odd automorphisms. We regard \(\Gamma\) as an element of
\(HGC_{n,j}\) and \({}^*\Gamma\) as the corresponding element of
\({}^*HGC_{n,j}\). The pairing is defined so that
\[
\langle \Gamma, {}^*\Gamma\rangle = |\operatorname{Aut}(\Gamma)|.
\]   };

\[HGC_{n,j} =\Hom( ^{*}HGC_{n,j} , \Q ) \]

We write $\mathcal{B}= H_{\mathrm{top}}(^{*}HGC_{n,j})$, which is the space of open Jacobi diagrams.

Next, we recall $\mathfrak{sl}_2$ weight system $W_{\mathfrak{sl}_2}$. Instead of the construction using Lie algebra tensors, $W_{\mathfrak{sl}_2}$ can be defined recursively using the local relations \cite{CV97} shown in Figure~\ref{fig:sl2_definitions}.

\begin{Def}\label{def:sl2_weight}
    The \emph{$\mathfrak{sl}_2$ weight system}, denoted by $W_{\mathfrak{sl}_2}$, is a map from the space of open Jacobi diagrams \footnote{The weight system is defined on the space of open Jacobi diagrams  which may be disconnected or have connected components without external vertices} to scalars (or a polynomial ring in $\hbar$). It is uniquely determined by the following five conditions, which are depicted in Figure~\ref{fig:sl2_definitions}:
    \begin{enumerate}
        \item[(A)] This relation allows one to resolve an edge connecting two trivalent vertices into a difference of chord configurations.
        \item[(B)] An isolated loop (bubble) consisting of dashed edges with vertices can be removed by multiplying the weight by $4\hbar$.
        \item[(C)] The value of a single strut (an isolated chord with external vertices) is normalized to $C$.
        % --- 追加箇所 ---
        \item[(D)] A simple closed dashed curve without vertices evaluates to $3$ (the dimension of $\mathfrak{sl}_2$).
        \item[(E)] The weight of the disjoint union of two diagrams $\Gamma$ and $\Gamma'$ is the product of their individual weights.
        % ----------------
    \end{enumerate}
    Since any admissible graph reduces to a scalar via a finite sequence of these operations, $W_{\mathfrak{sl}_2}$ is well-defined.
    In Figure~\ref{fig:sl2_definitions}, the symbol $\simeq$ indicates that the evaluation by $W_{\mathfrak{sl}_2}$ yields the same value for both sides.
\end{Def}

%begin-sl_2relations-------------------------------------------------------------------------------------------------
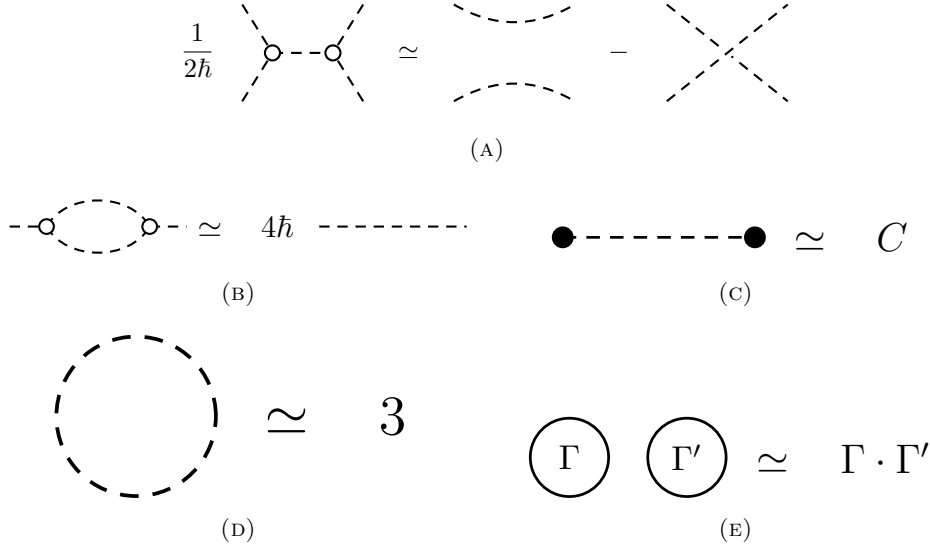
\begin{figure}[htbp]
    \centering
    % --- 共通スタイル定義 ---
    \tikzset{
        v/.style={circle, draw=black, fill=white, thick, inner sep=2pt},
        ext v/.style={circle, draw=black, fill=black, thick, inner sep=2pt},
        w/.style={thick, dash pattern=on 4pt off 3pt},
        solid w/.style={thick},
        % 抽象的なグラフΓを表すスタイル
        blob/.style={circle, draw=black, thick, minimum size=0.8cm}
    }
    
    % ==========================================
    % 1段目: sl_2 relation
    % ==========================================
    \begin{subfigure}[b]{1.0\textwidth}
        \centering
        \begin{align*}
            \frac{1}{2\hbar}
            \quad
            \begin{tikzpicture}[baseline=-0.5ex, scale=0.8]
                \draw[w] (-1, 0.8)  -- (-0.5, 0); \draw[w] (-1, -0.8) -- (-0.5, 0);
                \draw[w] (1, 0.8) -- (0.5, 0);    \draw[w] (1, -0.8)  -- (0.5, 0);
                \draw[w] (-0.5, 0) -- (0.5, 0);   \node[v] at (-0.5, 0) {}; \node[v] at (0.5, 0) {};
            \end{tikzpicture}
            \quad &\simeq \quad
            \begin{tikzpicture}[baseline=-0.5ex, scale=0.8]
                \draw[w] (-1, 0.8)  to[out=-30, in=210] (1, 0.8);
                \draw[w] (-1, -0.8) to[out=30, in=150] (1, -0.8);
            \end{tikzpicture}
            \quad - \quad
            \begin{tikzpicture}[baseline=-0.5ex, scale=0.8]
                \draw[w] (-1, 0.8) -- (1, -0.8);
                \draw[line width=4pt, white] (-1, -0.8) -- (1, 0.8); 
                \draw[w] (-1, -0.8)  -- (1, 0.8);
            \end{tikzpicture}
        \end{align*}
        \caption{}
    \end{subfigure}

    \par\vspace{0.3cm} % 行間

    % ==========================================
    % 2段目: Loop reduction & Normalization
    % ==========================================
    \begin{subfigure}[b]{0.48\textwidth}
        \centering
        \resizebox{\textwidth}{!}{%
            \begin{tikzpicture}[baseline=(current bounding box.center)]
                \begin{scope}
                    \coordinate (el) at (-1.2, 0); \coordinate (vl) at (-0.7, 0); 
                    \coordinate (vr) at (0.7, 0);  \coordinate (er) at (1.2, 0);  
                    \draw[w] (el) -- (vl); \draw[w] (vr) -- (er);
                    \draw[w] (vl) to[bend left=60] (vr); \draw[w] (vl) to[bend right=60] (vr);
                    \node[v] at (vl) {}; \node[v] at (vr) {}; 
                \end{scope}
                \node at (2, 0) {\Large $\simeq \quad 4 \hbar$};
                \begin{scope}[shift={(4,0)}]
                    \draw[w] (-1, 0) -- (1, 0);
                \end{scope}
            \end{tikzpicture}%
        }
        \caption{}
    \end{subfigure}
    \hfill
    \begin{subfigure}[b]{0.48\textwidth}
        \centering
        \resizebox{0.8\textwidth}{!}{%
            \begin{tikzpicture}[baseline=(current bounding box.center)]
                \begin{scope}
                    \draw[w] (-1, 0) -- (1, 0);
                    \node[ext v] at (-1, 0) {}; \node[ext v] at (1, 0) {}; 
                \end{scope}
                \node at (2, 0) {\Large $\simeq \quad C$};
            \end{tikzpicture}%
        }
        \caption{}
    \end{subfigure}

    \par\vspace{0.3cm} % 行間

    % ==========================================
    % 3段目: Vacuum loop & Product formula (追加)
    % ==========================================
    % (d) Vacuum loop (verticesなし) = 3
    \begin{subfigure}[b]{0.48\textwidth}
        \centering
        \resizebox{0.8\textwidth}{!}{%
            \begin{tikzpicture}[baseline=(current bounding box.center)]
                % 頂点のない破線の円
                \draw[w] (0,0) circle (0.6);
                % 等号
                \node at (1.5, 0) {\Large $\simeq \quad 3$};
            \end{tikzpicture}%
        }
        \caption{}
        \label{fig:vacuum_loop}
    \end{subfigure}
    \hfill
    % (e) Product formula: Gamma Gamma' = Gamma . Gamma'
    \begin{subfigure}[b]{0.48\textwidth}
        \centering
        \resizebox{0.9\textwidth}{!}{%
            \begin{tikzpicture}[baseline=(current bounding box.center)]
                % 左辺: 並置されたグラフ
                \node[blob] (G1) at (-0.6, 0) {$\Gamma$};
                \node[blob] (G2) at (0.6, 0) {$\Gamma'$};
                
                % 等号
                \node at (2.2, 0) {\Large $\simeq \quad \Gamma \cdot \Gamma'$};
            \end{tikzpicture}%
        }
        \caption{}
        \label{fig:product_formula}
    \end{subfigure}

    \caption{The $\mathfrak{sl}_2$ relations and values.}
    \label{fig:sl2_definitions}
\end{figure}

We introduce $G_{p,g}$, which is the key graph used to construct the nontrivial cycle. 

\begin{Def}
    The hairy graph \(G_{p,g}\) is defined as follows.
    It consists of a circle and $g-1$ parallel edges inside the circle (forming a graph with the first Betti number $g$). Furthermore, $2p$ hairs are attached to the upper arc of the circle.
\end{Def}

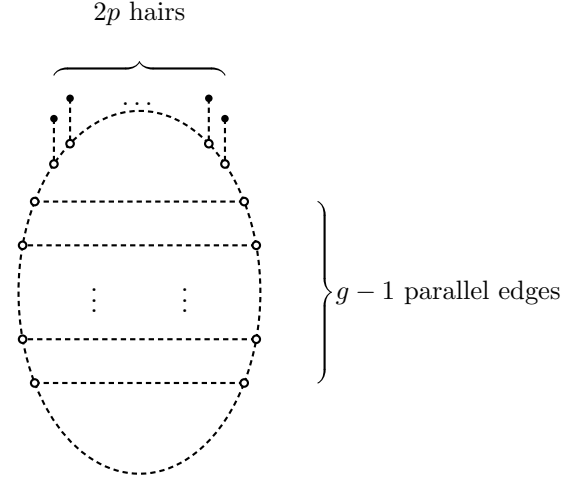
\begin{figure}[h]
    \centering
\begin{tikzpicture}[
    scale=1.0,
    baseline=(current bounding box.center),
    % --- 矢印の設定 ---
    >={Straight Barb[line width=0.5pt, angle=60:5pt]},  
    % --- 線の中間に矢印を置くスタイル ---
    mid arrow/.style={postaction={decorate,decoration={
        markings,
        mark=at position 0.55 with {\arrow{>}}
    }}},
    % --- 頂点のスタイル ---
    ext/.style={circle, fill=black, inner sep=1pt},
    int/.style={circle, draw=black, fill=white, thick, inner sep=1pt},
    dashed edge/.style={thick, dash pattern=on 2pt off 1.5pt},
    % --- ラベル用スタイル ---
    lbl/.style={font=\tiny, inner sep=1pt},
    % --- ブレース（中括弧）のスタイル ---
    brace/.style={thick, decorate, decoration={calligraphic brace, amplitude=5pt, raise=5pt}}
]

    % --- 左辺 ---
    \begin{scope}[shift={(0,0)}]
        
        \def\rx{1.6} \def\ry{2.4}

        % ==========================================
        % 1. 座標の定義
        % ==========================================
        % Chords (右側)
        \def\angRone{30}   \coordinate (R1) at ({\rx*cos(\angRone)}, {\ry*sin(\angRone)});
        \def\angRtwo{15}   \coordinate (R2) at ({\rx*cos(\angRtwo)}, {\ry*sin(\angRtwo)});
        \def\angRthree{345} \coordinate (R3) at ({\rx*cos(\angRthree)}, {\ry*sin(\angRthree)}); % -10度
        \def\angRfour{330} \coordinate (R4) at ({\rx*cos(\angRfour)}, {\ry*sin(\angRfour)});   % -20度

        % Chords (左側: 180-theta)
        \def\angLone{150}  \coordinate (L1) at ({\rx*cos(\angLone)}, {\ry*sin(\angLone)});
        \def\angLtwo{165}  \coordinate (L2) at ({\rx*cos(\angLtwo)}, {\ry*sin(\angLtwo)});
        \def\angLthree{195} \coordinate (L3) at ({\rx*cos(\angLthree)}, {\ry*sin(\angLthree)});
        \def\angLfour{210} \coordinate (L4) at ({\rx*cos(\angLfour)}, {\ry*sin(\angLfour)});

        % Hairs (上部)
        \def\angBone{45}   \coordinate (B1) at ({\rx*cos(\angBone)}, {\ry*sin(\angBone)});
        \def\angBtwo{55}   \coordinate (B2) at ({\rx*cos(\angBtwo)}, {\ry*sin(\angBtwo)});
        \def\angBthree{125} \coordinate (B3) at ({\rx*cos(\angBthree)}, {\ry*sin(\angBthree)});
        \def\angBfour{135} \coordinate (B4) at ({\rx*cos(\angBfour)}, {\ry*sin(\angBfour)});

        % Hairの先端 (T1~T4)
        \path (B1) ++(0, 0.6) coordinate (T1);
        \path (B2) ++(0, 0.6) coordinate (T2);
        \path (B3) ++(0, 0.6) coordinate (T3);
        \path (B4) ++(0, 0.6) coordinate (T4);

        % ==========================================
        % 2. 楕円の分割描画 (反時計回り矢印付き)
        % ==========================================
        % 点の順番(CCW): R2->R1->B1->B2...
        
        % 右上 (I)
        \draw[dashed edge] (R2) arc [start angle=\angRtwo, end angle=\angRone, x radius=\rx, y radius=\ry];
        \draw[dashed edge] (R1) arc [start angle=\angRone, end angle=\angBone, x radius=\rx, y radius=\ry];
        \draw[dashed edge] (B1) arc [start angle=\angBone, end angle=\angBtwo, x radius=\rx, y radius=\ry];
        
        % B2-B3の間
        \draw[dashed edge] (B2) arc [start angle=\angBtwo, end angle=\angBthree, x radius=\rx, y radius=\ry];
        
        % 左上 (I/II境界)
        \draw[dashed edge] (B3) arc [start angle=\angBthree, end angle=\angBfour, x radius=\rx, y radius=\ry];
        \draw[dashed edge] (B4) arc [start angle=\angBfour, end angle=\angLone, x radius=\rx, y radius=\ry];
        
        % 左側 (II)
        \draw[dashed edge] (L1) arc [start angle=\angLone, end angle=\angLtwo, x radius=\rx, y radius=\ry];
        % L2-L3の間
        \draw[dashed edge] (L2) arc [start angle=\angLtwo, end angle=\angLthree, x radius=\rx, y radius=\ry];
        \draw[dashed edge] (L3) arc [start angle=\angLthree, end angle=\angLfour, x radius=\rx, y radius=\ry];
        
        % 下部 (II -> I)
        \draw[dashed edge] (L4) arc [start angle=\angLfour, end angle=\angRfour, x radius=\rx, y radius=\ry];
        
        % 右下 (I)
        \draw[dashed edge] (R4) arc [start angle=\angRfour, end angle=\angRthree, x radius=\rx, y radius=\ry];
        % R3-R2の間
        \draw[dashed edge] (R3) arc [start angle=\angRthree, end angle=\angRtwo+360, x radius=\rx, y radius=\ry];

        % ==========================================
        % 3. Hairの描画
        % ==========================================
        % Hairのインデックス表示用マクロ
        % \i: ループ変数, \disp: 表示する番号
        \foreach \i/\disp in {1/2p, 2/2p-1, 3/2, 4/1} {
            % 矢印: 先端(T)から根元(B)へ
            \draw[dashed edge] (T\i) -- (B\i);
            
            % 根元 B: (I) を付与
            % ラベル位置調整: B3,B4は左寄り、B1,B2は右寄りに配置
            \node[int] at (B\i) {};
            
            % 先端 T: 番号を付与
            \node[ext] at (T\i) {};
        }

        % ==========================================
        % 4. Chordの描画
        % ==========================================
        \foreach \i in {1,2,3,4} {
            \draw[dashed edge] (R\i) -- (L\i);
            
        %     % 左 L: (II) を付与
            \node[int] at (L\i) {};
            
        %     % 右 R: (I) を付与
            \node[int] at (R\i) {};
        }

        % ==========================================
        % 5. 省略記号 (...)
        % ==========================================
        \node at (0, 2.5) {$\dots$}; % B2-B3
        \node at (-0.6, 0) {$\vdots$}; % L2-L3
        \node at (0.6, 0) {$\vdots$};  % R2-R3

        % ==========================================
        % 6. 本数の明示 (Brace)
        % ==========================================
        
        % (a) Hairの本数 (2p hairs)
        % T4(1番) から T1(2p番) までの範囲をブレースで囲む
        \coordinate (T1north) at ($(T1)+(0,0.4)$);
        \coordinate (T4north) at ($(T4)+(0,0.4)$);
        \draw[brace] (T4north) -- (T1north) node[midway, above=20pt] {$2p$ hairs};

        % (b) Chordの本数 (g-1 chords)
        % R1(上) から R4(下) までの範囲を右側でブレースで囲む
        % 座標を少し右にずらして配置
        \coordinate (BraceTop) at ($(R1)+(0.8,0)$);
        \coordinate (BraceBot) at ($(R4)+(0.8,0)$);
        \draw[brace] (BraceTop) -- (BraceBot) node[midway, right=8pt] {$g-1$ parallel edges};

    \end{scope}

\end{tikzpicture}
    \caption{The hairy graph $G_{p,g}$.}
    \label{fig:Gpg}
\end{figure}

\begin{Prop} \label{Prop:InfinitenessOfHGC}
    Let $n$ be an odd integer with $n \ge 5$, let $g$ be any positive integer, and let $p$ be a positive integer.
    Then, the class $[G_{p,g}]$ is a nontrivial element in $H_{\mathrm{top}}(^{*}HGC_{n,n-2}(2p+g-1,g))$.
    In particular, the space $H_{\mathrm{top}}(^{*}HGC_{n,n-2}(\bullet, g))$ is infinite-dimensional.
\end{Prop}

\begin{proof}
    We demonstrate that for any $g$ and $p$, the class $[G_{p,g}]$ is nontrivial in $HGC_{n,n-2}$ (specifically in the $g$-loop part), using the $\mathfrak{sl}_2$ weight system $W_{\mathfrak{sl}_2}$.

    For the inductive step, applying the loop reduction to the bottom loop yields:
    \[
        W_{\mathfrak{sl}_2}(G_{p,g}) = 4\hbar W_{\mathfrak{sl}_2}(G_{p,g-1}).
    \]

    For the base case $g=1$, we consider the relation involving the hairs. We obtain the following relation:
    
    % --- 図1: 関係式 ---
    \[
        \begin{tikzpicture}[
        scale=1.2,
        baseline=(current bounding box.center),
        % --- スタイル定義 ---
        % Hairの先端(external vertex): 黒塗り
        ext/.style={circle, fill=black, inner sep=1.5pt},
        % Hairの付け根(internal vertex): 白塗り、黒枠
        int/.style={circle, draw=black, fill=white, thick, inner sep=2pt},
        % 辺のスタイル
        base/.style={thick, dashed},                 % ベースの円/楕円
        hair/.style={thick, dashed},                 % Hair本体
        dashed edge/.style={thick, dashed}   % 相互作用（青色部分→黒破線）
    ]
    
        % =============================================
        % 1. 左辺: 3本のHairを持つ図
        % =============================================
        \begin{scope}[shift={(0,0)}]
            % ベースの楕円
            \draw[base] (0,0) ellipse (1.2cm and 0.6cm);
            
            % Hairの根元と先端の座標
            \coordinate (R1) at (-0.8, 0.45); \coordinate (T1) at (-1.0, 1.2);
            \coordinate (R2) at (-0.5,0.55);  \coordinate (T2) at (-0.7,1.3);
            \coordinate (R3) at (-0.1, 0.6);   \coordinate (T3) at (-0.15, 1.3);
            \coordinate (R4) at (0.8, 0.45);  \coordinate (T4) at (1.0, 1.2);
            
            % Hairの描画
            \draw[hair] (R1) -- (T1);
            \draw[hair] (R2) -- (T2);
            \draw[hair] (R3) -- (T3);
            \draw[hair] (R4) -- (T4);
            
            % --- 省略記号 (...) の追加 ---
            % R3とR4の間（上側）に配置
            \node at (0.35, 0.8) {$\cdots$};
    
            % 頂点
            \foreach \p in {R1, R2, R3, R4} \node[int] at (\p) {};
            \foreach \p in {T1, T2, T3, T4} \node[ext] at (\p) {};
        \end{scope}
    
        % =============================================
        % 等号と係数
        % =============================================
        \node at (2.0, 0.5) {\Large $= \quad 2\hbar$};
    
        % =============================================
        % 2. 中央の図: 1本目と2本目が接続
        % =============================================
        \begin{scope}[shift={(4.0,0)}]
            \draw[base] (0,0) ellipse (1.2cm and 0.6cm);
            
            \coordinate (R1) at (-0.8, 0.45); \coordinate (T1) at (-1.0, 1.2);
            \coordinate (R2) at (-0.4,0.55);  \coordinate (T2) at (-0.6,1.3);
            \coordinate (R3) at (-0.2, 0.55);   \coordinate (T3) at (-0.3, 1.3);
            \coordinate (R4) at (0.8, 0.45);  \coordinate (T4) at (1.0, 1.2);
            
            % \draw[hair] (R1) -- (T1);
            % \draw[hair] (R2) -- (T2);
            \draw[hair] (R3) -- (T3);
            \draw[hair] (R4) -- (T4);
            
            % --- 省略記号 (...) の追加 ---
            \node at (0.3, 0.8) {$\cdots$};
    
            % 【青色部分】1本目と2本目を破線で結ぶ (U字型)
            \draw[dashed edge] (T1) to[bend right=45] (T2);
            
            \foreach \p in {R3, R4} \node[int] at (\p) {};
            \foreach \p in {T1, T2, T3, T4} \node[ext] at (\p) {};
        \end{scope}
    
        % =============================================
        % 符号と係数
        % =============================================
        \node at (6.0, 0.5) {\Large $- \quad 2\hbar$};
    
        % =============================================
        % 3. 右辺: 1本目と2本目が交差接続
        % =============================================
        \begin{scope}[shift={(8.0,0)}]
            
            % --- 定数定義 ---
            \def\rx{1.2} % 横半径
            \def\ry{0.6} % 縦半径
    
            % --- 座標定義 (角度で指定) ---
            % R1: 135度 (左上)
            \coordinate (R1) at ({\rx*cos(135)}, {\ry*sin(135)}); 
            \coordinate (T1) at (-1.0, 1.2);
            
            % R2: 110度 (R1の少し右)
            \coordinate (R2) at ({\rx*cos(110)}, {\ry*sin(110)}); 
            \coordinate (T2) at (-0.6,1.3);
            
            % R3: 100度付近 (元のコードのsin(70)に合わせて調整)
            \coordinate (R3) at ({\rx*cos(100)}, {\ry*sin(70)});    
            \coordinate (T3) at (-0.3, 1.3);
            
            % R4: 45度
            \coordinate (R4) at ({\rx*cos(45)}, {\ry*sin(45)});    
            \coordinate (T4) at (1.0, 1.2);
    
            % --- 楕円の描画 (R1とR2の間を消す) ---
            \draw[base] (R1) arc [
                start angle=135,
                end angle=470,
                x radius=\rx,
                y radius=\ry
            ];
    
            % --- Hairの描画 ---
            \draw[hair] (R3) -- (T3);
            \draw[hair] (R4) -- (T4);
            
            % --- 省略記号 (...) の追加 ---
            \node at (0.35, 0.8) {$\cdots$};
    
            % --- 相互作用 (交差する破線) ---
            \draw[dashed edge] (T1) to[bend right=20] (R2);
            \draw[dashed edge] (T2) to[bend left=20] (R1);
            
            % --- 頂点の配置 ---
            \foreach \p in { R3, R4} \node[int] at (\p) {};
            \foreach \p in {T1, T2, T3, T4} \node[ext] at (\p) {};
    
        \end{scope}
    
    \end{tikzpicture}
    \]

    The second term vanishes due to the symmetry relation.
    Thus, we have the recurrence relation $W_{\mathfrak{sl}_2}(G_{p,1}) = (2\hbar C) \cdot W_{\mathfrak{sl}_2}(G_{p-1,1})$.
    
    Finally, we check the initial case explicitly:
    % --- 図2: G_{1,1} の計算 ---
    \[
    \begin{tikzpicture}[
        scale=1.2,
        baseline=(current bounding box.center),
        ext/.style={circle, fill=black, inner sep=1.5pt},
        int/.style={circle, draw=black, fill=white, thick, inner sep=2pt},
        base/.style={thick, dashed},
        hair/.style={thick, dashed},
        dashed edge/.style={thick, dashed}
    ]
        % 1. 左辺
        \begin{scope}[shift={(0,0)}]
            \draw[base] (0,0) ellipse (0.6cm and 0.6cm);
            \coordinate (R1) at (-0.4, 0.45); \coordinate (T1) at (-0.5, 1.2);
            \coordinate (R2) at (0.4, 0.45);  \coordinate (T2) at (0.5, 1.2);
            \draw[hair] (R1) -- (T1);
            \draw[hair] (R2) -- (T2);
            \foreach \p in {R1, R2} \node[int] at (\p) {};
            \foreach \p in {T1, T2} \node[ext] at (\p) {};
        \end{scope}
    
        % 等号
        \node at (1.5, 0.5) {$ = 2\hbar$};
    
        % 2. 中央
        \begin{scope}[shift={(2.8,0)}]
            \draw[base] (0,0) ellipse (0.6cm and 0.6cm);
            \coordinate (T1) at (-0.5, 1.2);
            \coordinate (T2) at (0.5, 1.2);
            \draw[dashed edge] (T1) to[bend right=45] (T2);
            \foreach \p in {T1, T2} \node[ext] at (\p) {};
        \end{scope}
    
        % 符号
        \node at (4.3, 0.5) {$ - 2\hbar$};

        % 3. 右辺
        \begin{scope}[shift={(5.6,0)}]
            \def\rx{0.6} \def\ry{0.6}
            \coordinate (R1) at ({\rx*cos(135)}, {\ry*sin(135)}); 
            \coordinate (T1) at (-0.5, 1.2);
            \coordinate (R2) at ({\rx*cos(45)}, {\ry*sin(45)});    
            \coordinate (T2) at (0.5, 1.2);
            \draw[base] (R1) arc [start angle=135, end angle=405, x radius=\rx, y radius=\ry];
            \draw[dashed edge] (T1) to[bend right=20] (R2);
            \draw[dashed edge] (T2) to[bend left=20] (R1);
            \foreach \p in {T1, T2} \node[ext] at (\p) {};
        \end{scope}

        % 結果の式
        \node[right] at (6.6, 0.5) {$ = 6\hbar C - 2\hbar C = 4\hbar C.$};
    \end{tikzpicture}
    \]
    This implies the weight is non-zero, which completes the proof.
\end{proof}

\subsection{The chord diagrams \texorpdfstring{$D(G_{p,g})$}{D(G(p,g))}} \label{subsec:The_chord_diagram_D(G(p,g))}

The chord diagram $D(G_{p,g})$ on oriented lines is obtained from the hairy graph $G_{p,g}$ as follows.
See Figure \ref{fig:ResultingChordDiagramDGpg}.

\begin{itemize}
    \item First, orient the edges of $G_{p,g}$ as shown in Figure~\ref{fig:orientedGpg}.
    Vertices on the left side have two incoming edges (called type $(\mathrm{I})$), while vertices on the right side have two outgoing edges (called type $(\mathrm{II})$).

    \item Replace each hair with an oriented line attached to two open chords:
    %図始まり========================================================================
    \begin{tikzpicture}[
    scale=0.5,
    >=Latex,
    mid arrow/.style={postaction={decorate,decoration={
        markings,
        mark=at position 0.55 with {\arrow{>}}
    }}},
    dot/.style={circle, fill=black, inner sep=1.5pt},
    dashed edge/.style={thick, dashed}
]

    \begin{scope}[shift={(0,0)}]
        % --- 座標定義 (T0, T1, T2 は削除) ---
        \coordinate (m) at (-1,0);
        \coordinate (O) at (0, 0);
        \coordinate (Top) at (3, 0);
        \coordinate (P1) at (1.0, 0);
        % P2は使われていないため削除しても良いが残します
        \coordinate (P2) at (2.0, 0); 

        % --- メインの横線 ---
        \draw[->,thick] (m) -- (Top);

        % --- 点 ---
        \foreach \p in {O,P1} {
            \node[dot] at (\p) {};
        }

        % --- 接続 (座標定義なしで記述) ---
        
        % 1. O から (0, 2) へ
        % ++(0, 2) は「直前の点(O)から x+0, y+2 の位置」という意味
        \draw[dashed edge, mid arrow] (O) -- ++(0, 1.5);
        
        % 2. (1, 2) から P1 へ
        % 始点を直接 (1, 2) と書くか、P1から見て相対的に指定する
        % ここでは「P1の真上(y=2)からP1へ」描きたいので以下のように書けます
        \path (P1) ++(0, 2) coordinate (temp); % 一時的な点を作る場合
        \draw[dashed edge, mid arrow] (P1) ++(0, 1.5) -- (P1);
        
        % もしくはもっと単純に、P1のx座標を利用して直接座標を書く
        % \draw[dashed edge, mid arrow] (1, 2) -- (P1);

    \end{scope}

\end{tikzpicture}.
%=============================図終わり================================
    Exceptionally, replace the rightmost hair on the upper edge of $G_{p,g}$ with the following diagram:
    \[
    %=======================begin-ChordsOn1LineWithgMinus1Chords===========================
    \begin{tikzpicture}[
        scale=0.4,
        >=Latex,
        mid arrow/.style={postaction={decorate,decoration={
            markings,
            mark=at position 0.55 with {\arrow{>}}
        }}},
        dot/.style={circle, fill=black, inner sep=1.5pt},
        dashed edge/.style={thick, dashed},
        brace/.style={thick, decorate, decoration={calligraphic brace, amplitude=5pt, raise=5pt}}
    ]
    
        \begin{scope}[shift={(0,0)}]
            % --- 座標定義 ---
            \coordinate (m) at (-1,0);
            \coordinate (O) at (0, 0);
            \coordinate (P1) at (1.0, 0);
            \coordinate (P2) at (2.0, 0); 
            \coordinate (P3) at (5.0, 0);
            \coordinate (P4) at (6.0, 0);
            \coordinate (P5) at (7.0, 0);
            \coordinate (Top) at (8.5, 0);
    
            % --- メインの横線 ---
            \draw[->,thick] (m) -- (Top);
    
            % ---- chords (円弧) ------
            % 【修正箇所】
            % out=90, in=90 で真上に持ち上げ、loosenessで高さを調整
            % looseness=1.5 くらいにすると大きく膨らみます
            \draw[dashed edge, mid arrow] (P1) to[out=90, in=90, looseness=1.5] (P4);
            \draw[dashed edge, mid arrow] (P2) to[out=90, in=90, looseness=1.5] (P3);
    
            % --- 点 ---
            \foreach \p in {O,P1, P2, P3, P4, P5} {
                \node[dot] at (\p) {};
            }
    
            % --- 省略記号 ---
            \coordinate (MidP2P3) at ($(P2)!0.5!(P3)$);
            
            % 2. 円弧の省略記号 (縦の点々)
            % アーチが高くなったので、点の位置も少し上に上げました (+0.5 -> +1.5)
            \node at ($(MidP2P3)+(0, 0.8)$) {$\vdots$};

            % --- 接続 (Hairs) ---
            \draw[dashed edge, mid arrow] (O) -- ++(0, 3);
            \draw[dashed edge, mid arrow] (P5) ++(0, 3) coordinate (P5_top) -- (P5);
    
            % ==========================================
            % 3. 本数の明示 (Brace)
            % ==========================================
            
            % chords の本数
            % 円弧がさらに高くなったので、ブレースの位置を上に逃がしました (3.5 -> 4.5)
            \coordinate (BraceLeft) at ($(P1)+(0, 2)$);
            \coordinate (BraceRight) at ($(P4)+(0, 2)$);
    
            \draw[brace] (BraceLeft) -- (BraceRight) node[midway, above=10pt] {$g-1$ chords};
    
        \end{scope}
    
    \end{tikzpicture}%end=======ChordOn1LineWithgMinus1Chords==================================================
    \]

    \item Connect the open ends of the chords according to the connectivity of the graph $G_{p,g}$.

    \item Label the solid lines counterclockwise starting from the leftmost hair.
    Label the $g-1$ nested chords $1, \dots, g-1$ (outer first). The remaining chords follow the order of the solid lines containing their starting points.
    \end{itemize}
    
\begin{figure}[h]
    \centering
    \begin{tikzpicture}[
        scale=1.5,
        baseline=(current bounding box.center),
        % --- 矢印の設定 ---
        >=Latex,
        % >={Straight Barb[line width=0.5pt, angle=60:5pt]},  
        % --- 線の中間に矢印を置くスタイル ---
        mid arrow/.style={postaction={decorate,decoration={
            markings,
            mark=at position 0.55 with {\arrow{>}}
        }}},
        % --- 頂点のスタイル ---
        ext/.style={circle, fill=black, inner sep=1pt},
        int/.style={circle, draw=black, fill=white, thick, inner sep=1pt},
        dashed edge/.style={thick, dash pattern=on 2pt off 1.5pt},
        % --- ラベル用スタイル ---
        lbl/.style={font=\tiny, inner sep=1pt},
        % --- ブレース（中括弧）のスタイル ---
        brace/.style={thick, decorate, decoration={calligraphic brace, amplitude=5pt, raise=5pt}}
    ]
    
        % --- 左辺 ---
        \begin{scope}[shift={(0,0)}]
            
            \def\rx{1.6} \def\ry{2.4}
    
            % ==========================================
            % 1. 座標の定義
            % ==========================================
            % Chords (右側)
            \def\angRone{30}   \coordinate (R1) at ({\rx*cos(\angRone)}, {\ry*sin(\angRone)});
            \def\angRtwo{15}   \coordinate (R2) at ({\rx*cos(\angRtwo)}, {\ry*sin(\angRtwo)});
            \def\angRthree{345} \coordinate (R3) at ({\rx*cos(\angRthree)}, {\ry*sin(\angRthree)}); % -10度
            \def\angRfour{330} \coordinate (R4) at ({\rx*cos(\angRfour)}, {\ry*sin(\angRfour)});   % -20度
    
            % Chords (左側: 180-theta)
            \def\angLone{150}  \coordinate (L1) at ({\rx*cos(\angLone)}, {\ry*sin(\angLone)});
            \def\angLtwo{165}  \coordinate (L2) at ({\rx*cos(\angLtwo)}, {\ry*sin(\angLtwo)});
            \def\angLthree{195} \coordinate (L3) at ({\rx*cos(\angLthree)}, {\ry*sin(\angLthree)});
            \def\angLfour{210} \coordinate (L4) at ({\rx*cos(\angLfour)}, {\ry*sin(\angLfour)});
    
            % Hairs (上部)
            \def\angBone{45}   \coordinate (B1) at ({\rx*cos(\angBone)}, {\ry*sin(\angBone)});
            \def\angBtwo{55}   \coordinate (B2) at ({\rx*cos(\angBtwo)}, {\ry*sin(\angBtwo)});
            \def\angBthree{125} \coordinate (B3) at ({\rx*cos(\angBthree)}, {\ry*sin(\angBthree)});
            \def\angBfour{135} \coordinate (B4) at ({\rx*cos(\angBfour)}, {\ry*sin(\angBfour)});
    
            % Hairの先端 (T1~T4)
            \path (B1) ++(0, 0.6) coordinate (T1);
            \path (B2) ++(0, 0.6) coordinate (T2);
            \path (B3) ++(0, 0.6) coordinate (T3);
            \path (B4) ++(0, 0.6) coordinate (T4);

            % ==========================================
            % 2. 楕円の分割描画 (反時計回り矢印付き)
            % ==========================================
            % 点の順番(CCW): R2->R1->B1->B2...
            
            % 右上 (I)
            \draw[dashed edge, mid arrow] (R2) arc [start angle=\angRtwo, end angle=\angRone, x radius=\rx, y radius=\ry];
            \draw[dashed edge, mid arrow] (R1) arc [start angle=\angRone, end angle=\angBone, x radius=\rx, y radius=\ry];
            \draw[dashed edge, mid arrow] (B1) arc [start angle=\angBone, end angle=\angBtwo, x radius=\rx, y radius=\ry];
            
            % B2-B3の間
            \draw[dashed edge, mid arrow] (B2) arc [start angle=\angBtwo, end angle=\angBthree, x radius=\rx, y radius=\ry];
            
            % 左上 (I/II境界)
            \draw[dashed edge, mid arrow] (B3) arc [start angle=\angBthree, end angle=\angBfour, x radius=\rx, y radius=\ry];
            \draw[dashed edge, mid arrow] (B4) arc [start angle=\angBfour, end angle=\angLone, x radius=\rx, y radius=\ry];
            
            % 左側 (II)
            \draw[dashed edge, mid arrow] (L1) arc [start angle=\angLone, end angle=\angLtwo, x radius=\rx, y radius=\ry];
            % L2-L3の間
            \draw[dashed edge, mid arrow] (L2) arc [start angle=\angLtwo, end angle=\angLthree, x radius=\rx, y radius=\ry];
            \draw[dashed edge, mid arrow] (L3) arc [start angle=\angLthree, end angle=\angLfour, x radius=\rx, y radius=\ry];
            
            % 下部 (II -> I)
            \draw[dashed edge, mid arrow] (L4) arc [start angle=\angLfour, end angle=\angRfour, x radius=\rx, y radius=\ry];
            
            % 右下 (I)
            \draw[dashed edge, mid arrow] (R4) arc [start angle=\angRfour, end angle=\angRthree, x radius=\rx, y radius=\ry];
            % R3-R2の間
            \draw[dashed edge, mid arrow] (R3) arc [start angle=\angRthree, end angle=\angRtwo+360, x radius=\rx, y radius=\ry];

            % ==========================================
            % 3. Hairの描画
            % ==========================================
            % Hairのインデックス表示用マクロ
            % \i: ループ変数, \disp: 表示する番号
            \foreach \i/\disp in {1/2p, 2/2p-1, 3/2, 4/1} {
                % 矢印: 先端(T)から根元(B)へ
                \draw[dashed edge, mid arrow] (T\i) -- (B\i);
                
                % 根元 B: (I) を付与
                % ラベル位置調整: B3,B4は左寄り、B1,B2は右寄りに配置
                \node[int, label={[lbl, xshift=(\i>2 ? 3pt : -3pt)]below:$(\mathrm{I})$}] at (B\i) {};
                
                % 先端 T: 番号を付与
                \node[ext, label={[lbl]above:$\disp$}] at (T\i) {};
            }

            % ==========================================
            % 4. Chordの描画
            % ==========================================
            \foreach \i in {1,2,3,4} {
                \draw[dashed edge, mid arrow] (R\i) -- (L\i);
                
                % 左 L: (II) を付与
                \node[int, label={[lbl]left:$(\mathrm{I})$}] at (L\i) {};
                
                % 右 R: (I) を付与
                \node[int, label={[lbl]right:$(\mathrm{II})$}] at (R\i) {};
            }

            % ==========================================
            % 5. 省略記号 (...)
            % ==========================================
            \node at (0, 2.5) {$\dots$}; % B2-B3
            \node at (-0.6, 0) {$\vdots$}; % L2-L3
            \node at (0.6, 0) {$\vdots$};  % R2-R3

            % ==========================================
            % 6. 本数の明示 (Brace)
            % ==========================================
            
            % (a) Hairの本数 (2p hairs)
            % T4(1番) から T1(2p番) までの範囲をブレースで囲む
            \coordinate (T1north) at ($(T1)+(0,0.4)$);
            \coordinate (T4north) at ($(T4)+(0,0.4)$);
            \draw[brace] (T4north) -- (T1north) node[midway, above=20pt] {$2p$ hairs};
    
            % (b) Chordの本数 (g-1 chords)
            % R1(上) から R4(下) までの範囲を右側でブレースで囲む
            % 座標を少し右にずらして配置
            \coordinate (BraceTop) at ($(R1)+(0.8,0)$);
            \coordinate (BraceBot) at ($(R4)+(0.8,0)$);
            \draw[brace] (BraceTop) -- (BraceBot) node[midway, right=8pt] {$g-1$ parallel edges};
    
        \end{scope}
    
    \end{tikzpicture}
    \caption{Orientation of $G_{p,g}$}
    \label{fig:orientedGpg}
\end{figure}
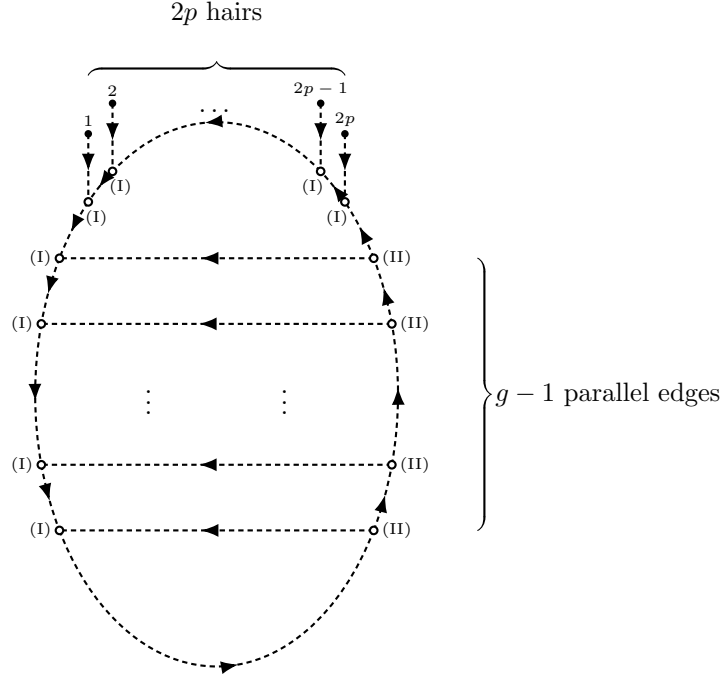

\begin{figure}
    \centering
    %begin{fig:ResultingChordDiagramDGpg}=======================================
   \begin{tikzpicture}[
        scale=0.5,
        >=Latex,
        % --- 共通スタイル ---
        mid arrow/.style={postaction={decorate,decoration={
            markings,
            mark=at position 0.5 with {\arrow{>}}
        }}},
        dot/.style={circle, fill=black, inner sep=1.5pt},
        dashed edge/.style={thick, dashed},
        brace/.style={thick, decorate, decoration={calligraphic brace, amplitude=5pt, raise=5pt}},
        lbl/.style={font=\scriptsize, inner sep=2pt}
    ]
    
        % =========================================================
        % 1. 左の図 (g-1 chords) - 縦向き
        % =========================================================
        \begin{scope}[shift={(0,0)}]
            % --- 座標定義 ---
            \coordinate (m) at (0,-1);
            \coordinate (Top) at (0, 11.5);
            \coordinate (O) at (0, 0);
            
            \coordinate (P1) at (0, 1.0);
            \coordinate (P2) at (0, 2.0); 
            \coordinate (P6) at (0, 4.2);
            \coordinate (P7) at (0, 5.8);
            \coordinate (P3) at (0, 8.0);
            \coordinate (P4) at (0, 9.0);
            \coordinate (P5) at (0, 10.0);
    
            % --- メインの縦線 (3分割) ---
            \draw[thick] (m) -- ($(P2)+(0, 0.6)$);
            \draw[thick] ($(P6)-(0, 0.6)$) -- ($(P7)+(0, 0.6)$);
            \draw[->, thick] ($(P3)-(0, 0.6)$) -- (Top);
    
            % \node at (O) [below left] {O};
    
            % --- Chords (番号追加) ---
            \draw[dashed edge, mid arrow] (P1) to[out=180, in=180, looseness=1.5] node[pos=0.5, left, lbl] {$1$} (P4);
            \draw[dashed edge, mid arrow] (P2) to[out=180, in=180, looseness=1.5] node[pos=0.5, left, lbl] {$2$} (P3);
            \draw[dashed edge, mid arrow] (P6) to[out=180, in=180, looseness=1] node[pos=0.5, left, lbl, yshift=5pt] {$g-1$} (P7);
    
            % --- 点 ---
            \foreach \p in {O,P1, P2, P3, P4, P5, P6, P7} {
                \node[dot] at (\p) {};
            }
    
            % --- 縦方向の省略記号 (\vdots) ---
            \node at ($(P2)!0.5!(P6)$) {$\vdots$};
            \node at ($(P7)!0.5!(P3)$) {$\vdots$};
            
            % --- ★復活: 横方向の省略記号 (\cdots) ---
            % Chord 2 と Chord g の間の空間に配置
            \coordinate (MidHeight) at (0, 5.0); % 中心の高さ
            \node at ($(MidHeight)-(1.8, 0)$) {$\cdots$};
    
            % --- ブレース (g-1 chords) ---
            \coordinate (BraceBottom) at ($(P1)-(3.8, 0)$);
            \coordinate (BraceTop) at ($(P4)-(3.8, 0)$);
            \draw[brace] (BraceBottom) -- (BraceTop) node[midway, left=10pt] {$g-1$ chords};
        \end{scope}

        % =========================================================
        % 2. 右の図 (2番目の線)
        % =========================================================
        \begin{scope}[shift={(3,0)}]
            \coordinate (Start) at (0, -1);
            \coordinate (End) at (0, 3);
            \draw[->,thick] (Start) -- (End);
    
            \coordinate (Q1_bot) at (0, 0);
            \coordinate (Q2_bot) at (0, 1);
            \node[dot] at (Q1_bot) {};
            \node[dot] at (Q2_bot) {};
            
            \draw[dashed edge, mid arrow] (O) -- node[midway, above, sloped, lbl] {$g$} (Q2_bot);
        \end{scope}

        % =========================================================
        % 3. ３番目の図 (3番目の線)
        % =========================================================
        \begin{scope}[shift={(6,0)}]
            \coordinate (Start2) at (0, -1);
            \coordinate (End2) at (0, 3);
            \draw[->,thick] (Start2) -- (End2);
    
            \coordinate (Q12_bot) at (0, 0);
            \coordinate (Q22_bot) at (0, 1);
            \node[dot] at (Q12_bot) {};
            \node[dot] at (Q22_bot) {};
            
            \draw[dashed edge, mid arrow] (Q1_bot) -- node[midway, above,sloped, lbl] {$g+1$} (Q22_bot);
            \draw[dashed edge, mid arrow] (Q12_bot)  -- node[midway, sloped, yshift=5pt, xshift=5pt, lbl] {$g+2$} ++(1.8, +0.9)(Q12_bot) ;
        \end{scope}
    
        % --- ドット ---
        \node at (9, 1.0) {\Large $\dots$};

        % =========================================================
        % 4. 一番右の図 (最後の線)
        % =========================================================
        \begin{scope}[shift={(12,0)}]
            \coordinate (Start3) at (0, -1);
            \coordinate (End3) at (0, 3);
            \draw[->,thick] (Start3) -- (End3);
    
            \coordinate (Q13_bot) at (0, 0);
            \coordinate (Q23_bot) at (0, 1);
            \node[dot] at (Q13_bot) {};
            \node[dot] at (Q23_bot) {};
            
            \draw[dashed edge, mid arrow] (Q13_bot) to[out=0, in=0, looseness=1.2] node[pos=0.5,sloped,xshift=-5pt, yshift=5pt , lbl] {$g+p-1$} (P5);
            \draw[dashed edge, mid arrow] (Q23_bot) ++(-1.8, -0.9)  -- node[above, sloped, lbl, xshift=-0pt] {$g+p-2$} (Q23_bot);
        \end{scope}

        % =========================================================
        % 全体のブレース
        % =========================================================
        \coordinate (TopBraceLeft) at ($(End)+(-0.1, 0.1)$);
        \coordinate (TopBraceRight) at ($(End3)+(0.1, 0.1)$);
        \draw[brace] (TopBraceLeft) -- (TopBraceRight) node[midway, above=10pt] {$2p-1$ lines};
    
    \end{tikzpicture}
    %end{fig:ResultingChordDiagramDGpg}========================================
    \caption{The resulting chord diagram $D(G_{p,g})$}
    \label{fig:ResultingChordDiagramDGpg}
\end{figure}
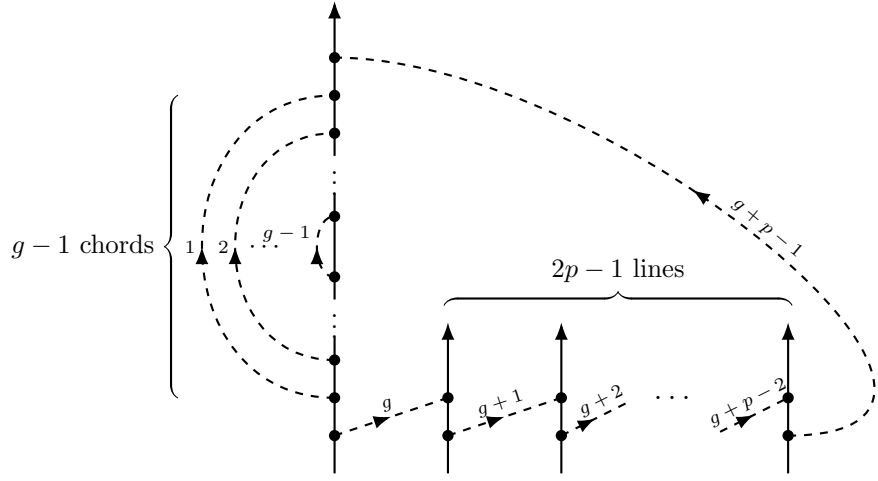

\begin{Not}
    We denote the ribbon presentation associated with $D(G_{p,g})$ by $P(G_{p,g})$ and we denote the resulting cycle by $c_{p,g}$.
\end{Not}

\section{Non-triviality of the cycle} \label{sec:Non-triviality}

In this section, our goal is to establish the nontriviality of the cycle $c_{p,g}$ in $H_{2p+g-1}( \bEmb_c(\R^j, \R^n) )$.
This section essentially follows the approach of \cite[Section~7]{Yos25a}.

First, we recall the basic properties of the \emph{modified configuration space integral} in \S~\ref{subsec:Basic_properties_of_the_modified configuration_space_integral}. 
Next, in \S~\ref{subsec:counting_formula}, we verify that the counting formula \cite[Theorem~7.14]{Yos25a}, which is the key to proving the nontriviality, remains valid in our setting.
Finally, we prove the nontriviality of $c_{p,g}$ in \S~\ref{subsec:Proof_of_the_non_triviality_of_psi(G(k,g))}.

\subsection{Basic properties of the modified configuration space integral} \label{subsec:Basic_properties_of_the_modified configuration_space_integral}

We recall the definition of Yoshioka's modified configuration space integral.
Recall that an element $\psi \in \bEmb_c(\R^j,\R^n)$ is a family of immersions $(\psi_u)_{u \in [0,1]}$ such that $\psi_1 =\iota$ and $\psi_0 \in \mathrm{Emb}_c(\R^j,\R^n)$.

To define the modified configuration space integral, we first introduce the following notation.

\begin{Not}
    Let $E_{s,t} = E_{s,t}(\R^j,\R^n)$ be defined by the following pullback square:
    \[\begin{tikzcd}
        {E_{s,t}} &&& {C_{s+t}(\R^n)} \\
        \\
        {\bEmb_c(\R^j,\R^n) \times C_{s}(\R^j) } &&& {C_s(\R^n)}
        \arrow[from=1-1, to=1-4]
        \arrow[from=1-1, to=3-1]
        \arrow["\lrcorner"{anchor=center, pos=0.125}, draw=none, from=1-1, to=3-4]
        \arrow["{\text{restriction map}}", from=1-4, to=3-4]
        \arrow["{\text{ev}_{u=0}}"', from=3-1, to=3-4]
    \end{tikzcd}\]
    Here, $C_k(\R^d)$ denotes the Fulton-MacPherson compactification of the configuration space $\mathrm{Conf}_k(\R^d)$.
    We denote the typical fiber over $\psi \in \bEmb_c(\R^j,\R^n)$ by $C_{s,t}(\psi)$.
\end{Not}

\begin{Not}
    Let $\mathrm{Inj}(\R^j,\R^n)$ denote the space of $\R$-linear injective maps $\R^j \to \R^n$.
    Let $\Gamma$ be a labeled decorated graph with $B(P(\Gamma))=s$ and $W(P(\Gamma))=t$.
    The associated map
    \[ F_{\Gamma}\colon E_{s,t} \longrightarrow (S^{j-1})^{|E_{\eta}(\Gamma)|} \times (S^{n-1})^{|E_{\theta}(\Gamma)|} \times (P(\mathrm{Inj}(\R^j,\R^n)))^{|B_e(\Gamma)|} \]
    is defined as follows:
    \begin{itemize}
        \item For $e=(p,q) \in E_{\theta}(\Gamma)$, the map to the factor $S^{n-1}$ is given by the composition:
        \[
            {E_{s,t}} \longrightarrow {{C_{s+t}(\R^n)}} \xrightarrow{\text{Gauss}} {S^{n-1}},
        \]
        where the Gauss map is defined by
        \[ (x_v)_{v \in V(\Gamma)} \mapsto \frac{x_q -x_p}{\|x_q-x_p \|}. \]

        \item For $e=(p,q) \in E_{\eta}(\Gamma)$, the map to the factor $S^{j-1}$ is given by the composition of the following diagram:
        \[
        E_{s,t} \longrightarrow \bEmb_c(\R^j,\R^n) \times C_s(\R^j) \twoheadrightarrow  C_s(\R^j) \xrightarrow{\text{Gauss}} S^{j-1},
        \]
        where the Gauss map is given by
        \[ (x_v)_{v \in B(\Gamma)} \mapsto \frac{x_q-x_p}{\| x_q - x_p \|}. \]

        \item For $v \in B_{e}(\Gamma)$, the map to the factor $P(\mathrm{Inj}(\R^j,\R^n))$ is defined as
        \[
            E_{s,t} \longrightarrow \bEmb_c(\R^j,\R^n) \times C_s(\R^j) \xrightarrow{\text{ev}_{s=v}} P(\mathrm{Inj}(\R^j,\R^n)).
        \]
    \end{itemize}
\end{Not}

\begin{Def}[Modified configuration space integral]
    Let $a_{d}\colon S^{d} \xrightarrow{\cong} S^d$ be the antipodal map.
    Fix volume forms $\omega_{j-1} \in A_{dR}(S^{j-1})$ and $\omega_{n-1} \in A_{dR}(S^{n-1})$ such that $a_{j-1}^{*}\omega_{j-1} = (-1)^j \omega_{j-1}$ and $a_{n-1}^{*} \omega_{n-1} = (-1)^n \omega_{n-1}$.

    Let $\pi$ be the natural projection:
    \[
    \pi\colon E_{s,t} \longrightarrow \bEmb_c(\R^j,\R^n) \times C_{s}(\R^j) \twoheadrightarrow  \bEmb_c(\R^j,\R^n).
    \]
    Then the modified configuration space integral $\overline{I}\colon DGC_{n,j} \rightarrow A_{dR}(\bEmb_c(\R^j,\R^n))$ is defined as
    \[
        \overline{I}(\Gamma) = \pi_{\ast}\left(F_{\Gamma}^{*}\left(\bigwedge_{e \in E_{\eta}} \omega_{j-1} \wedge \bigwedge_{e \in E_{\theta}} \omega_{n-1} \wedge \bigwedge_{v \in B_e} J(D_v(\Gamma)) \right)\right).
    \]
    Here, the differential form $J(D_{v}(\Gamma)) \in A_{dR}(P(\mathrm{Inj}(\R^j,\R^n)))$ depends only on the decoration $D_v(\Gamma)$ at $v$  and the fixed differential forms $\omega_{j-1}$ and $\omega_{n-1}$.
    (See \cite[Section 5.4]{Yos25b} for the definition.)
\end{Def}

In fact, the following theorem holds.

\begin{Thm}\cite[Theorem 5.20]{Yos25b}
    The modified configuration space integral $\overline{I}\colon DGC_{n,j} \longrightarrow A_{dR}(\bEmb_c(\R^j,\R^n))$ is a cochain map.
\end{Thm}

Next, we review basic properties of the configuration integral.
The rest of this section follows \cite[Section 7.1]{Yos25a}.

% \subsection{Modified some paring lemmas about Configuration space integral}

\begin{Lem}[Symmetry lemma]
    Let $\Gamma$ be a labeled plain graph, where the labeling includes the orientation data of each edge,
    and let $\rho $ be an automorphism of $\Gamma$.
    Consider the involution on $C_{s,t}$ given by
    \[
        \sigma(x_1,x_2,\dots,x_s,x_{s+1},\dots,x_{s+t})
        =
        (x_{\rho^{-1}(1)},x_{\rho^{-1}(2)},\dots,x_{\rho^{-1}(s)},
         x_{\rho^{-1}(s+1)},\dots,x_{\rho^{-1}(s+t)}).
    \]
    Let $X,Y \subset C_{s,t}$ be submanifolds satisfying $\sigma(X)=Y$.
    Then
    \[
        \int_X I(\Gamma) = \epsilon(\rho)\, \int_Y I(\Gamma),
    \]
    where $\epsilon(\rho) \in \{-1,1\}$ is the sign induced by the action of $\rho$ on the labeling of $\Gamma$.
\end{Lem}

\begin{proof}
    The following identity holds up to sign:
    \[ \int_Y \overline{I}(\Gamma) =  \int_X \sigma^{*} \overline{I}(\Gamma) = \int_X \overline{I}(\Gamma). \]
    The sign difference coincides with $\epsilon(\rho)$.
\end{proof}

\begin{Lem}[Localizing lemma {\cite{SW12}}]\label{Localizing lemma}
    Let $\Gamma$ be a decorated graph of order $k$ with $B(P(\Gamma)) = s$ and $W(P(\Gamma)) = t$.
    Let $\psi$ be the cycle associated with a chord diagram $C$.
    For each $i$, let $C_0(i)\subset C_{s,t}(\psi)$ be the locus of
    configurations for which either $S_i$ or $T_i$ contains no black
    vertices. Then
    \[
        \int_{C_0(i)}\overline I(\Gamma)\longrightarrow 0
        \qquad (y_i\to 0).
    \]
    Consequently, the same conclusion holds for
    $C_0=\bigcup_i C_0(i)$.
\end{Lem}

\begin{proof}
    On the locus $C_0(i)$, at least one of the two local
    pieces $S_i$ and $T_i$ contains no black vertices. Hence the associated
    configuration space integral is unaffected, in the limit $y_i\to 0$,
    by allowing the corresponding ribbon crossing to pass through a
    self-intersection.
    
    After this degeneration, the contribution from this local region
    vanishes. Indeed, it is either zero for dimensional reasons, or the
    contributions from the two resolutions of the local crossing cancel
    each other. 
\end{proof}

\begin{Not}
    Recall that we have classified the geometric components associated with each chord into target annuli and source disks.
    For each vertex label \(\alpha \in \{1, \dots, 2k\}\), we define the associated subset $X_\alpha \subset \mathbb{R}^j$ as follows:
    \[
        X_\alpha =
        \begin{cases}
            T_{\frac{\alpha +1}{2} } & (\text{if } \alpha \text{ is odd}),\\[2mm]
            S_{ \frac{\alpha}{2}} & (\text{if } \alpha \text{ is even}).
        \end{cases}
    \]
    We denote by \(X^+_\alpha\) the primary part of \(X_\alpha\).
    Recall that \(X^+_\alpha\) is the component constructed exclusively by iterations of the map $e_+$ (i.e., the "backbone" of the tree structure).
\end{Not}

By the Localizing Lemma, it suffices to consider graphs with $s = 2k$ and $t = 0$
(in particular, such graphs have no nontrivial decoration)
and configurations in which the black vertex $x_{\sigma^{-1}(\alpha)}$ lies in $X_\alpha$ for each $\alpha$.
We denote the set of such configurations by $C_{\sigma}$, and the ordering of $B(\Gamma)$ is given by the composition
\[
    \sigma \colon B(\Gamma) \xrightarrow[\cong]{\sigma} V(C) \xrightarrow[\cong]{\tau} \{1,2,\dots, 2k \}.
\]
(Recall that the ordering $\tau$ of $V(C)$ is defined such that the $(2i-1)$-th element is the initial vertex and the $(2i)$-th element is the terminal vertex of the $i$-th chord.)

\begin{Lem}[Pairing lemma {\cite[Lemma~4.9]{SW12}}]
    If $x_{\sigma^{-1}(2i)} \in S_i$ and $x_{\sigma^{-1}(2i-1)} \in T_i$
    are not connected by a dashed edge, then
    \[
        \int_{C_{\sigma}} \overline{I}(\Gamma) \longrightarrow 0
        \qquad (y_i \to 0).
    \]
\end{Lem}

\begin{proof}
    The proof is analogous to that of Lemma~\ref{Localizing lemma}.
\end{proof}

Before stating the next lemma, we clarify the definition of the components $Y$ and $Y'$.
Recall that our planetary-like system is constructed by iterating the maps $e_+$ and $e_-$.
We refer to any subset of the form $e_{\mu_1} \circ \dots \circ e_{\mu_m}(D)$ (where $D \in \{D_+, D_-, I \times S\}$) as an \emph{elementary component}.

\begin{Lem}[Same system lemma]
    Let $\Gamma$ be a decorated graph of defect $0$ and order $k$ such that
    \[
        B_e(P(\Gamma)) = 2k .
    \]
    Let $C_{\mathrm{sys}}$ denote the set of configurations for which there exist
    two elementary components $Y,Y' \subset \mathbb{R}^j$ and points
    $y \in Y$, $y' \in Y'$ satisfying the following conditions:
    \begin{itemize}
        \item $y$ and $y'$ are connected by solid edges in $\Gamma$;
        \item $Y$ is contained in some $X_\alpha$, and $Y'$ is contained in some
              $X_\beta$, where the case $\alpha=\beta$ is allowed;
        \item there exist disjoint balls $B_Y$ and $B_{Y'}$ such that the orbit of
              $Y$ is contained in $B_Y$, and the orbit of $Y'$ is contained in
              $B_{Y'}$.
    \end{itemize}
    Then
    \[
        \int_{C_{\mathrm{sys}}} \overline{I}(\Gamma)
        \longrightarrow 0
    \]
    as the radii of $B_Y$ and $B_{Y'}$ tend to $0$.
\end{Lem}
    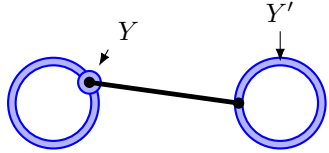
\begin{figure}[htbp]
        \centering
        \begin{tikzpicture}[
            scale=0.5,
            >=Latex,
            annulus/.style={draw=blue, thick, fill=blue!30, even odd rule},
            small disk/.style={draw=blue, thick, fill=blue!30},
            label line/.style={thin, draw=black},
            point/.style={circle, fill=black, inner sep=1.5pt},
        ]
            % --- Left component (Y) ---
            \begin{scope}[shift={(-3, 0)}]
                \path[annulus] (0,0) circle (1.2) (0,0) circle (1.0);
                \coordinate (LocalP1) at (30:1.1);
                \path[small disk] (LocalP1) circle (0.3);
                \coordinate (P1) at (LocalP1);
                \draw[<-, label line] (P1) ++(60:0.5) -- ++(60:0.5) node[above right] {$Y$};
            \end{scope}

            % --- Right component (Y') ---
            \begin{scope}[shift={(3, 0)}]
                \path[annulus] (0,0) circle (1.2) (0,0) circle (1.0);
                \coordinate (P2) at (180:1.1);
                \coordinate (AnnulusTop) at (90:1.1);
                \draw[<-, label line] (AnnulusTop) -- ++(90:0.8) node[above] {$Y'$};
                \node[point] at (P1) {};
                \node[point] at (P2) {};
            \end{scope}

            % --- Connection ---
            \draw[black, line width=1.8pt] (P1) -- (P2);
        \end{tikzpicture}
        \caption{Configuration where $Y$ and $Y'$ are not linked.}
        \label{fig:same_system_lemma}
    \end{figure}

\begin{proof}
    As the radii shrink, the direction vector
    \[
        \frac{x - x'}{\|x - x'\|} \in S^{j-1}
    \]
    varies only within an arbitrarily small neighborhood of a point on \(S^{j-1}\);
    hence the contribution to the integral tends to zero.
\end{proof}

\begin{Lem}[Ingoing lemma]
    Let $\Gamma$ be as above.
    Let \(C_{\mathrm{in}}\) be the set of configurations for which there exist elementary components
    \(Y, Y', Y''\) such that
    \begin{itemize}
        \item There are points \(y \in Y\), \(y' \in Y'\), and \(y'' \in Y''\),
        \item $y$ is connected by solid edges to both \(y'\) and \(y''\),
        \item both \(Y'\) and \(Y''\) lie inside \(Y\) (as sub-components in the recursive structure).
    \end{itemize}
    Then
    \[
        \int_{C_{\mathrm{in}}} \overline{I}(\Gamma)
        \longrightarrow 0
    \]
    as the radii of \(Y'\) and \(Y''\) tend to \(0\).

    \begin{figure}[htbp]
        \centering
        \begin{subfigure}{0.48\textwidth}
            \centering
            \scalebox{0.5}{
            \begin{tikzpicture}[
                scale=0.5, >=Latex,
                annulus/.style={draw=blue, thick, fill=blue!10, even odd rule},
                point/.style={circle, fill=black, inner sep=1.5pt},
                label line/.style={thin, draw=black, shorten >= 2pt}
            ]
                % Y (Outer)
                \path[annulus] (0,0) circle (5.0) (0,0) circle (4.8);
                \coordinate (LabelY) at (135:5.5);
                \draw[<-, label line] (135:5.05) -- (LabelY) node[above left] {$Y$};

                % Y' (Middle)
                \path[annulus] (0,0) circle (1.6) (0,0) circle (1.4);
                \coordinate (LabelYprime) at (95:1.8);
                \draw[<-, label line] (95:1.65) -- (LabelYprime) node[above] {$Y'$};

                % Y'' (Inner)
                \path[annulus] (0,0) circle (0.6) (0,0) circle (0.4);
                \coordinate (LabelYdoubleprime) at (30:0.8);
                \draw[<-, label line] (30:0.65) -- (LabelYdoubleprime) node[right] {$Y''$};

                % Connections
                \coordinate (P1) at (110:4.9);
                \coordinate (P2) at (110:1.5);
                \coordinate (P3) at (50:0.5);

                \draw[black, line width=1.2pt] (P1) -- (P2);
                \draw[black, line width=1.2pt] (P1) -- (P3);

                \node[point] at (P1) {};
                \node[point] at (P2) {};
                \node[point] at (P3) {};
                \end{tikzpicture}}
            \label{fig:IngoingLemmaA}
        \end{subfigure}
        \begin{subfigure}{0.3\textwidth}
            \centering
            \scalebox{0.5}{
            \begin{tikzpicture}[
                scale=0.5, >=Latex,
                annulus/.style={draw=blue!80!black, thick, fill=blue!10, fill opacity=0.6, even odd rule},
                disk/.style={draw=blue!80!black, thick, fill=blue!70, fill opacity=0.6},
                label line/.style={thin, draw=black, shorten >= 2pt}
            ]
                % Left group (Y and Y'')
                \begin{scope}[shift={(0, 0)}]
                    \path[annulus] (0,0) circle (5.0) (0,0) circle (4.8);
                    \draw[<-, label line] (135:5.05) -- ++(135:0.5) node[above left] {$Y$};
                    \path[annulus] (0,0) circle (0.8) (0,0) circle (0.6);
                    \draw[<-, label line] (135:0.85) -- ++(135:1.5) node[above left] {$Y''$};
                    \coordinate (LocalP_Disk) at (30:0.7);
                    \path[disk] (LocalP_Disk) circle (0.3);
                    \coordinate (P2) at (LocalP_Disk);
                \end{scope}

                % Right group (Y')
                \begin{scope}[shift={(2, 0)}]
                    \path[annulus] (0,0) circle (0.8) (0,0) circle (0.6);
                    \draw[<-, label line] (45:0.85) -- ++(45:0.6) node[above right] {$Y'$};
                    \coordinate (P3) at (60:0.7);
                \end{scope}

                % Connections
                \coordinate (P1) at (60:4.9);
                \draw[black, line width=1.5pt] (P1) -- (P2);
                \draw[black, line width=1.5pt] (P1) -- (P3);

                \fill[black] (P1) circle (2pt);
                \fill[black] (P2) circle (2pt);
                \fill[black] (P3) circle (2pt);
            \end{tikzpicture}
            }
            \label{fig:IngoingLemmaB}
        \end{subfigure}
        \caption{Examples of nested components $Y, Y', Y''$ in the Ingoing Lemma.}
    \end{figure}
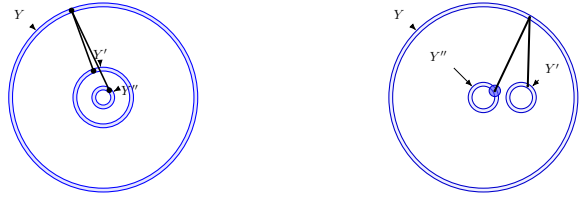
\end{Lem}

\begin{proof}
    Since $\omega_{j-1} \wedge \omega_{j-1} = 0$, the integral vanishes in this limit.
\end{proof}

\subsection{Counting formula} \label{subsec:counting_formula}

To prove the nontriviality of the cycle, it remains only to check that the counting formula still holds. 

We recall the graph-chord pairing.

\begin{Def}[Graph--chord pairing {\cite[Definition~3.7]{Yos25a}}] \label{def:graph_chord_pairing}
    Let $C$ be a chord diagram on $s$ directed lines of order $k$, and let $\Gamma$ be a labeled plain graph of order $\ k$.  
    The pairing $\langle \Gamma, C\rangle$ is defined as follows.

    We count only those permutations 
    \[
        \sigma \colon B(\Gamma) \longrightarrow V(C)
    \]
    satisfying the following conditions:
    \begin{itemize}
        \item[$(\mathrm{I})$] 
            The number of black vertices matches the number of vertices of the chord diagram,  
            i.e.\ $|B(\Gamma)| = |V(C)| = 2k$.

        \item[$(\mathrm{II})$] 
            The permutation $\sigma$ induces a map 
            \[
                \overline{\sigma} \colon E(\Gamma) \longrightarrow E(C).
            \]

        \item[$(\mathrm{III})$] 
            If two black vertices $v_1, v_2$ lie in the same solid component of $\Gamma$,  
            then their images $\sigma(v_1)$ and $\sigma(v_2)$ lie on the same directed line of~$C$.

        \item[$(\mathrm{IV})$] 
            If a vertex $w \in V(C)$ is not on the $x$–axis,  
            then there exists a unique vertex $w'$ located below $w$ on the same oriented line such that  
            $\sigma^{-1}(w)$ is connected to $\sigma^{-1}(w')$ in $\Gamma$.
    \end{itemize}

    The sign of $\sigma$ is $+1$ if and only if  
    the orientation on $\Gamma$ obtained by transporting the orientation of $C$  
    along the assignment $\sigma \colon B(\Gamma)\to V(C)$  
    agrees with the original orientation of $\Gamma$.
\end{Def}

\begin{Lem}[{\cite[Theorem~7.11]{Yos25a}}, {\cite[Lemma~4.12]{SW12}}] \label{lem:pairing_is_1}
    Let $\Gamma$ be a labeled decorated graph.
    Let $\sigma \colon B(\Gamma) \xrightarrow{\cong} V(C)$ be a bijection satisfying all the conditions in Definition~\ref{def:graph_chord_pairing}.
    Let $C^+ \subset C_\sigma$ be the subconfiguration determined by the condition
    \[
        x_{\sigma^{-1}(\alpha)} \in X^+_{\alpha} \quad (\text{for all } \alpha).
    \]
    Then,
    \[
        \int_{C^+} \overline{I}(\Gamma) = \pm 1,
    \]
    where the sign in the equation depends on the sign of $\sigma$ in Definition~\ref{def:graph_chord_pairing}.
\end{Lem}

\begin{proof}
    See \cite[Theorem~7.11]{Yos25a}. Note that the primary part of our cycle coincides with Yoshioka's cycle.
\end{proof}

The following lemma is the key step in the evaluation.

\begin{Lem}[Counting formula, analogue of {\cite[Theorem~7.14]{Yos25a}}]
    Let
    \[
        w = \sum_{\Gamma} \frac{w_{\Gamma}}{|\operatorname{Aut}(\Gamma)|}\,\Gamma
    \]
    be a cocycle in \(DGC_{n,j}\).
    Let $G(C)$ be the set of plain graphs $\Gamma$ such that $|\langle \Gamma, C \rangle| \neq 0$.
    We assume that each graph in $G(C)$ is equipped with the labeling induced by $C$.
    Then,
    \[
        \langle \overline{I}(w),\, c_{C} \rangle
        = \sum_{\Gamma \in G(C)} w_{\Gamma}.
    \]
\end{Lem}

This follows from the next theorem.

\begin{Thm}[Analogue of {\cite[Theorem~7.11]{Yos25a}}, {\cite[Lemma~4.12]{SW12}}]
    Let \(C\) be a chord diagram of order \(\le k\) on \(k -g + 1\) lines,
    and let \(\psi\) be the ribbon cycle associated with \(C\).
    Let \(\Gamma\) be a decorated graph of order \(k\) with \(g\) loops and defect \(0\).
    Then,
    \[
        \overline{I}(\Gamma)(\psi)
        = \langle P(\Gamma),\, C \rangle.
    \]
\end{Thm}

\begin{proof}
    By the Symmetry Lemma, the Same System Lemma, and the Ingoing Lemma,
    it suffices to show that if the permutation $\sigma \in S_{2k}$ is induced by an element of $\Gamma$,
    then
    \[
        \int_{C_\sigma} \overline{I}(\Gamma) = 1.
    \]
    In view of Lemma~\ref{lem:pairing_is_1}, it remains to prove that
    \[
        \int_{C_{\sigma} \setminus C^+} \overline{I}(\Gamma) = 0.
    \]
    The region $C_{\sigma} \setminus C^+$ consists of configurations where at least one point lies in a secondary part (a branch).
    Any secondary part is of the form $e_+^m \circ e_- (\dots)$ for some depth $m \ge 0$.
    We prove the vanishing of the integral by induction on the depth $m$.

    First, consider the base case ($m=0$) where a point $x$ lies in the image of $e_-$.
    By the Localizing Lemma, there is exactly one point $y$ in the image of $e_{+}(I \times S)$, and $x$ and $y$ must be connected by a solid edge.
    Thus, by the Same System Lemma, the contribution from this region is zero.

    Next, assume that the integral vanishes for configurations involving branches of the form $e_+^j \circ e_- \circ \cdots$ for all $j < m$.
    Consider a configuration where a point $x$ lies in the image of $e_+^m \circ e_-$.
    From the inductive assumption and the Localizing Lemma, there exists a point $y$ in $e_+^m(I \times S)$ or $e_+(D'_+) $.
    Thus, $x$ and $y$ are connected by a sequence of solid edges.
    Therefore, by the Same System Lemma or the Ingoing Lemma, the contribution is zero.

    Since any point not in $C^+$ must belong to some such region (for some $m$), all contributions from $C_\sigma \setminus C^+$ vanish.
    This completes the proof.
\end{proof}
   
\subsection{Proof of the nontriviality of \texorpdfstring{$c_{p,g}$}{c(p,g)}}\label{subsec:Proof_of_the_non_triviality_of_psi(G(k,g))}

Recall that there is a sequence of quasi-isomorphisms:
\[
    HGC_{n,j} \xtwoheadleftarrow{\simeq} PGC_{n,j}' \xtwoheadrightarrow{\simeq} PGC_{n,j} \xtwoheadleftarrow{\simeq} DGC_{n,j}.
\]
Let $H = \sum_{\Gamma} \frac{w_{\Gamma}}{\operatorname{Aut}(\Gamma)}\Gamma$ be a graph cocycle in $HGC_{n,j}(k,g)$. 
We construct a lift $\widetilde{H} \in DGC_{n,j}$ as follows:
First, since the map $PGC_{n,j}' \to HGC_{n,j}$ is a surjective quasi-isomorphism, we can lift $H$ to a cycle $H' \in PGC_{n,j}'$. 
Next, we project $H'$ to $PGC_{n,j}$ via the natural map. 
Finally, since the map $DGC_{n,j} \to PGC_{n,j}$ is also a surjective quasi-isomorphism, we lift the image in $PGC_{n,j}$ to a cycle $\widetilde{H}$ in $DGC_{n,j}$.

In this context, by saying that $\widetilde{H}$ is a ``lift'' of $H$ we mean that for any graph $\Gamma$ (excluding those with double edges or self-loops), the coefficient of $\Gamma$ in $\widetilde{H}$ coincides with that in $H$.

Let $^{*}gdPGC_{n,j}'$ denote the graph complex obtained as the quotient space of $PGC_{n,j}'$ by the subspace generated by non-good plain graphs.
(In other words, $^{*}gdPGC_{n,j}$ is the chain complex consisting of good plain graphs.)
Then, the following theorem holds.

\begin{Thm}[Theorem~2.2, \cite{Yos23}]
    The natural map $H_{\mathrm{top}}(^{*}PGC_{n,j}') \longrightarrow H_{top}(^\ast gdPGC_{n,j}) $ is injective.
\end{Thm}

By virtue of this theorem, we may assume that for any term in the cycle $\widetilde{H}$, its underlying plain graph (ignoring decorations) is good.

\begin{Thm}\label{Thm:NonTrivialityOfCycleAssociatedToGkg}
    \[
        \langle \overline{I}(\widetilde{H}), \, c_{p,g} \rangle = \pm w_{G_{p,g}}
    \]
    Here, the sign of $w_{G_{p,g}}$ depends only on the orientation of $G_{p,g}$.
\end{Thm}

The theorem follows by applying the counting formula to $\widetilde{H}$ together with Lemma 4.3.3.

\begin{Lem}
    Let $G$ be a set of good plain graphs such that $\langle \Gamma, D(G_{p,g}) \rangle \neq 0$.
    \[
        \pm [G_{p,g}] = \sum_{\Gamma \in G} [\Gamma] \in H_{\mathrm{top}}(^{*} gdPGC_{n,j}').
    \]
    Here, the sign of the right-hand side is taken so that $\langle D(G_{p,g}), \Gamma \rangle = 1$, and the sign of the left-hand side \emph{depends only on} the choice of the labeling of $G_{p,g}$.
\end{Lem}

\begin{proof}
    The set $G$ consists of $2^{2g-2}$ graphs.
    Each graph in this set is characterized by its solid edges, which form a path of length $2g$ on the vertices $\{0,1,\dots,2g\}$ with a unique minimum at $0$.
    (In other words, the solid part of $\Gamma \in G$ traces a trajectory that is first decreasing and then increasing).
    
    The proof is based on the fact that the graph $[G_{k,g}]$ can be transformed into a sum of chord diagrams.
    The following diagram illustrates this transformation for the first few steps.
    It should be noted that the vertex indices $p_0, \dots, p_4$ in the figure correspond to the $y$-coordinates of the chord diagram in reverse order:

    \begin{center}
        \scalebox{0.4}{\usebox{\equationchordsII}}
    \end{center}

    By recursively applying this relation, the graph $[G_{k,g}]$ is eventually expanded into a linear combination of graphs.
    Observe that the solid edges appearing in this sum exactly trace the paths defined in the set $G$.
    Thus, the expansion coincides precisely with $\sum_{\Gamma \in G} [\Gamma]$.
\end{proof}

\section{The cycle is in the image of the Hurewicz map} \label{Sec:The cycle is in the image of Hurewicz map}

Finally, we prove that our cycle lies in the image of the Hurewicz map. 
The essential idea follows \cite{Wat22}: we decompose the cycle constructed in \S~\ref{subsec:Cycles associated with chord diagrams} into elementary components possessing the Brunnian property.

\subsection{Boxed ribbon presentations}

\begin{Def}[boxed ribbon presentations] \label{def of boxed ribbon presentation}
     A \emph{boxed ribbon presentation}
    $P = \mathcal{D} \cup \mathcal{B} \cup \mathcal{L} \cup \mathcal{Q}$ 
    is a neat immersion consisting of \(l\) rectangular components and \(m\) line segments in  
    $I^3$, where
    \begin{itemize}
        \item 
        $\mathcal{D} = \mathcal{D}_1 \cup \cdots \cup \mathcal{D}_l$,  
        with $\mathcal{D}_i = D_{i:0} \cup \cdots \cup D_{i:k_i}$.  
        Here each $D_{i:k}$ ($k \ge 1$) is a disjoint embedded disk, and 
        $D_{i:0}$ is the rectangle 
        \[
          I \times \{p_i\} \times [-1, - \frac{1}{2} ],
        \]
        where the points $(p_i)_i$ are distinct elements of $[-1,1]$.
        
        \item 
        $\mathcal{B} = \mathcal{B}_1 \cup \cdots \cup \mathcal{B}_l$, where $\mathcal{B}_i = B_{i:1} \cup \cdots \cup B_{i:k_i}$ consists of $k_i$ disjoint bands (each $B_{i:l} \cong I \times I$).
        
        \item 
        $\mathcal{L} = \bigcup_{j=1}^m L_j$, where each $L_j$ is the line segment \( L_j = \{q_j\} \times I \times \{1/2\}\).

       \end{itemize}

       such that 

        \begin{itemize}
        \item 
        Each band $B_{i:l}$ connects two distinct elements of $\mathcal{D}_i$  
        (either two disks or possibly the rectangle $D_{i:0}$),  
        and it is allowed to intersect other disks in $\mathcal{D}$ transversally.
        
        \item 
        Each line segment $L_j$ may also intersect the interiors of disks in $\mathcal{D}$ transversally.
    \end{itemize}
    
    Each disk may intersect several bands but at most one line segment.  
\end{Def}

    % We consider the additional structure $\mathcal{Q} = \bigcup_{i=1}^k Q_i$ as in definition \ref{Def of ribbon presentation}, each of which is a neighborhood of a crossing between disks and bands. 

    %  Furthermore, we assume that for each $1 \le i \le m$, the condition $\mathcal{D}_i \cap \mathcal{L} = \emptyset$ (called Type condition) or $\mathcal{D}_i \cap \mathcal{B} = \emptyset$ is satisfied.
    % If $\mathcal{D}_i$ satisfies type $A$, $\mathcal{D}_i$ is called \emph{Type Y}, and otherwise called \emph{Type N}.

An intersection of a line segment and a disk is also called a crossing, and such a disk is also called a leaf.
Near a crossing between a band and a line segment, we use the local chart as usual.

\begin{align*}
    D&= \{ (x_1,x_2,x_3) \in \R^3 | \, x_1^2 + x_3^2 \le 1, x_2 = 0  \} \\
    L&= \{ (x_1, x_2, x_3) \in \R^3 | \,  x_1 = x_3 =0, \, |x_2| < 3 \} 
\end{align*}

We orient $\mathcal{L}$ in the positive $x_2$-direction.

\begin{Def}
    We define a ribbon $(j+1)$-disk $V_{P}$ in $I^3 \times I^{n-j-2} \times I^{j-1}$
     \[
        V_P = \bigl( (\mathcal{D} \setminus \bigcup_{i=1}^l D_{i;0} ) \times [-\frac{1}{2}, \frac{1}{2}]^{j-1} \bigr) \cup \bigl( \bigcup_{i=1}^l D_{i;0} \times I^{j-1} \bigr) \cup \bigl( \mathcal{B} \times [-\frac{1}{4},\frac{1}{4}]^{j-1} \bigr) \subset I^3 \times \boldsymbol{0} \times I^{j-1}.
    \]
    We also define $V_P' \subset \R^3 \times \boldsymbol{0} \times \R^{j-1}$ 
    to be the manifold obtained by rounding only the corners of $V_P$ 
    that occur along the bands and disks involved in $\mathcal{D}$. 
    Note that for the parts related to the $D_{i;0}$'s, 
    we round only the corners created by their attachment to the bands.
\end{Def}

\begin{Def}
    We define an embedding  $\varphi_P\colon \bigsqcup^l I^j \sqcup \bigsqcup^m I^{n-j-1} \to I^n$ 
    whose image is given by $cl(\partial V_P' \setminus \partial (I^n) ) \sqcup (\mathcal{L} \times I^j \times \boldsymbol{0} )$.
    We take the base point of $\varphi_P\colon \bigsqcup^l I^j \sqcup \bigsqcup^m I^{n-j-1} \to I^n$ to be the standard embedding 
    \[ \bigcup_{i=1} ^l (I \times \{p_i \} \times \{ -1/2 \} \times \bold{0} \times I^{j-1}) \cup \bigcup_{i=1}^m (\{q_i\} \times I  \times \{1/2\} \times I^{n-j-2} \times \bold{0}  )  \hookrightarrow I^3 \times I^{n-j-2} \times I^{j-1}   \]
    % We define $\varphi_{P_{\boldsymbol{v}}}$ similarly.
    Then we can define a natural lift $(\varphi_{P,t})_{t \in [0,1]} $ of $\varphi_P$ to $\pbEmb(\bigsqcup^l I^j \sqcup \bigsqcup^m I^{n-j-1}, I^n)$ in a similar way.
    % Note that if we replace the $2$-disks and bands in the ribbon presentation $P$ with thickened disks, 
    % we obtain a thickened embedding modulo immersion $\varphi_P: \bigsqcup^l I^j \sqcup \bigsqcup^m I^{n-j-1} \to I^n$.
    % Here, the parameterization of the embedding is given by using the canonical path to the standard immersion ...
\end{Def}

We equip the base point of $\pbEmb(\bigsqcup^l I^j \sqcup \bigsqcup^m I^{n-j-1}, I^n)$ with the framing $(f_1, f_2, \dots , f_{n-j} )$ defined by
$f_i =  \partial_{x_{i+1}} $ on $I^j$ and 
$f_i = 
\begin{cases}
    \partial_{3} &(i=1) \\
    \partial_1 &(i=2)  \\
    \partial_{i+n-j} & (3 \le i \le j)
\end{cases}$ on $I^{n-j-2}$.

We equip $\varphi_P$ with a framing induced by the canonical path to the base point and the framing of the base point.

Let $\boldsymbol{v}=(v_1,\dots,v_k) \in (S^{n-j-2})^k$. 
Then we obtain a perturbed embedding modulo immersion $\varphi_{\boldsymbol{v}}$ in a similar way, 
and we denote the resulting cycle by $c_{P}\colon (S^{n-j-2})^k \to \pbEmb(\bigsqcup^l I^j \sqcup \bigsqcup^m I^{n-j-1} , I^n)$.
% we set $V_{P_{\boldsymbol{v}}}$ as ...

\subsection{Type \texorpdfstring{$\rI$}{I} boxed ribbon presentation} \label{sebsec:TypeIribbonPresentation}
We construct the boxed ribbon presentation corresponding to a type $\rI$ vertex.

\begin{Def}
    We define the \emph{Type $\rI$ boxed ribbon presentation}, denoted by $P_{\rI}$, as the ribbon presentation shown in Figure~\ref{fig:MyRibbonTypeI}.
    We denote the component containing $D_0$ (resp. $D_0'$) by $\mathcal{D}$ (resp. $\mathcal{D}'$).
    % Note that the crossing of the disk $D_1$ (resp. $D_2$) with $L_1$ is positive (resp. negative).
    \begin{figure}[h]
        \centering
        \scalebox{1.0}{\usebox{\MyRibbonTypeI}}
        \caption{The Type $\rI$ boxed ribbon presentation}
        \label{fig:MyRibbonTypeI}
    \end{figure}
\end{Def}

\begin{Not}
    Let $c_{\rI}'$ denote the element $  c_{P_{\rI}}' \in \pi_{n-j-2} \left( \bEmb_{\partial}(I^j\sqcup I^j \sqcup I^{n-j-1},I^n)\right)$.
    % We define $V_{\rI}$  similarly.
\end{Not}

\begin{Prop} \label{type i belongs to trivial component}
    If $n-j > 2$, the element $ c_{\rI}'$ described above lies in the trivial component.
\end{Prop}

\begin{proof}
    First, since $n-j > 2$, we can perturb the band so that $B_+$ and $D_+$ are disjoint;
     see Figure~\ref{fig:MyRibbonTypeIPerturbed}.
    We then pull back $D'$ to $D_0'$ and perform an S4-move. 
        
    \begin{figure}[h]
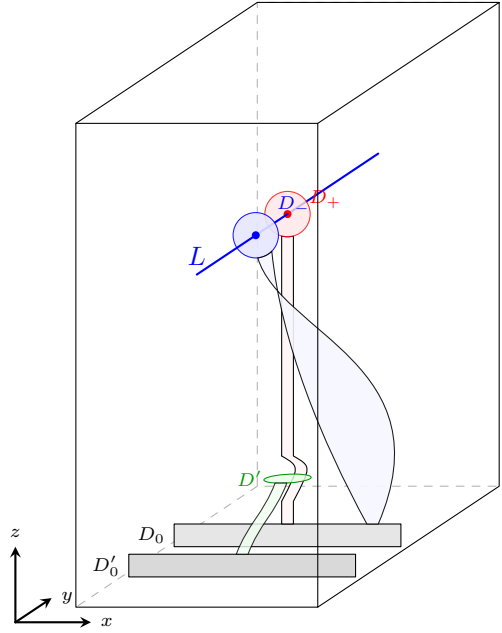

        \centering
        \scalebox{1.0}{\usebox{\MyRibbonTypeIPerturbed}}
        \caption{The perturbed Type $\rI$ boxed ribbon presentation}
        \label{fig:MyRibbonTypeIPerturbed}
    \end{figure}
\end{proof}

Using the path from the ribbon presentation $P_{\rI}$ employed in Proposition~\ref{type i belongs to trivial component}, we change the base point of $c_{\rI}'$ to the trivial one.
We denote the resulting map by $ {c_{\rI}} \colon S^{n-j-2} \longrightarrow \bEmb_{\partial}(I^j\sqcup I^j \sqcup I^{n-j-1},I^n)$.

\begin{Prop} \label{Prop:TypeIisBrunnian}
    $[{c}_{\rI}]$ becomes trivial if any disk component or line segment is removed.
\end{Prop}

\begin{proof}
    If $\mathcal{D}$ is removed, we pull back $D'$ to $D_0$.
    Next, if $\mathcal{D}'$ is removed, we perform an S4-move.
    Finally, if $L$ is removed, we pull back $D_{\pm}$ and then pull back $D'$.
\end{proof}

Thus, to simplify the notation, we isotope $c_{\rI}$ so that the restriction of $c_{\rI}$ to the $I^j$ associated with the $\mathcal{D}'$ component is constant. 
We regard $c_{\rI}$ as an element in $\bEmb_{\partial}(\nu (I^j) \cup \nu(I^{j}) \cup \nu(I^{n-j-1}), I^n  )$.
Here $\nu(I^j)$ and $\nu(I^{n-j-1})$ are closed tubular neighborhoods.
The notation $\overline{\Emb}_{\partial}$ is used by abuse of notation:
the subscript $\partial$ also records that the boundary identifications
of these tubular neighborhoods, equivalently their boundary framings,
are fixed once and for all, independently of the parameter of the family.

We define $b_1, b_2 \in H_{n-j-1}(I^n \setminus (\nu(I^j) \cup \nu(I^j) \cup \nu(I^{n-j-1}) ))$ so that $b_1$ (resp. $b_2$) is a normal sphere of $I^j \subset I^n$ associated with $\mathcal{D}$ (resp. $\mathcal{D}'$ ).  
$b_1' \in H_{j}(I^n \setminus (\nu(I^j) \cup \nu(I^j) \cup \nu(I^{n-j-1}) ) ) $ is taken so that $b_1'$ is a normal sphere of $I^{n-j-1} \subset I^n$ associated with $L$.

\subsection{Type \texorpdfstring{$\rII$}{II} boxed ribbon presentation} \label{subsec:TypeIIribbonPresentation}

We introduce the Type $\mathrm{II}$ ribbon presentation, which corresponds to a Type $\rII$ vertex.

\begin{Def}
    Let 
    \[
        c_{\rII}' \colon S^{j-1} \longrightarrow \bEmb_{\partial} \bigl(I^j \bigsqcup I^{n-j-1} \bigsqcup I^{n-j-1},\, I^n\bigr)
    \]
    be the family of embeddings modulo immersions obtained from the ribbon presentation shown in Figure~\ref{fig:MyRibbonTypeII} 
    by rotating one of the stems associated with $D_+'$ once around the stem associated with $D_+$. 
    % In other words, \(c_{\rII}\) parametrizes the admissible embeddings produced by this one-turn motion of the stem.  
    Again, for further details, see \cite[Section 6.2]{Yos25a}. 
    We call this family the Type II cycle.
    
     \begin{figure}[h]
        \begin{center}
        \scalebox{1.0}{\usebox{\MyRibbonTypeII}}
        \end{center}
        \caption{The Type $\rII$ boxed ribbon presentation}
        \label{fig:MyRibbonTypeII}
    \end{figure}    
\end{Def}

\begin{Prop} \label{Type ii is in trivial component}
    The above $c_{\rII}'$ is a family in the trivial component.
\end{Prop}

\begin{proof}
    Perform an S4-move twice.
\end{proof}

Using the path to the basic ribbon presentation determined by the deformation of the ribbon presentation $c_{\rII}'$ employed in Proposition~\ref{Type ii is in trivial component}, we change the base point of $c_{\rII}'$ to the trivial one.  
The resulting map is denoted by ${c_{\rII}}$.

\begin{Prop} \label{Prop:TypeIIBrunnianLink}
     $ c_{\mathrm{II}}$  is trivial if any line segment or connected component of disks is removed.
\end{Prop}

\begin{proof}
    If $L$ is removed, pull back $D_+$ and $D_-$, then perform an S4-move.
    Similarly, if $L'$ is removed, pull back $D_+'$ and $D_-'$, then perform an S4-move.
    There is nothing to show in the last case.
\end{proof}

We define $b_1 \in H_{n-j-1}(I^n \setminus (\nu(I^j) \cup \nu(I^{n-j-1}) \cup \nu(I^{n-j-1}) ))$ so that $b_1$ is a normal sphere of $I^j \subset I^n$ associated with $\mathcal{D}$ .  
$b_1', b_2' \in H_{j}(I^n \setminus (\nu(I^j) \cup \nu(I^{n-j-1}) \cup \nu(I^{n-j-1}) ) ) $ are taken so that $b_1'$ (resp. $b_2'$) is a normal sphere of $I^{n-j-1} \subset I^n$ associated with $L$ (resp. $L'$).

\subsection{Iterated surgery}

In this section, to prove that $c_{p,g}$ is in the image of the Hurewicz map, we reconstruct the cycle using an idea based on families of clasper surgery (\cite{Wat09a}).

First, for each white vertex of $G_{p,g}$, we label the half-edges of $G_{p,g}$.

\begin{Mysec}[Labeling of the half-edges of $G_{p,g}$]
    Each edge of $G_{p,g}$ is oriented as in Figure~\ref{fig:orientedGpg}.
    For each vertex $v$, we equip the incoming (resp.\ outgoing) half-edges attached to $v$ with a labeling as follows:

\begin{itemize}
    \item For a \emph{white vertex incident to a hair}: the hair is labeled by $1$, and the other incoming half-edge is labeled by $2$. The unique outgoing edge is labeled by $1$.
    \item For a \emph{white vertex on the left-hand side circle}: the incoming half-edge contained in the circle is labeled by $1$, and the other incoming half-edge is labeled by $2$. The unique outgoing edge is labeled by $1$.
    \item For the \emph{remaining white vertices}: the unique incoming half-edge is labeled by $1$. The outgoing half-edge in the circle is labeled by $1$, and the other outgoing edge is labeled by $2$.
\end{itemize}
\end{Mysec}

% --- [Definition] Y-link construction from Arrow Graph ---
\medskip
Recall the notion of Y-graph (\cite{Gou99}).

\begin{Def}[Y-graph associated with $G_{p,g}$]
Let $\Gamma= G_{p,g}$ and assume that each edge of $G_{p,g}$ is oriented and each half-edge incident to each white vertex is labeled. 
Given a framed embedding $f\colon \Gamma \to \R^{n-1} \times \mathbb{R} $ satisfying $f(\Gamma) \cap \R^j =  f(B(\Gamma))$, we associate a Y-link as follows:

\begin{enumerate}
    \item[(1)] For each edge $e$ (except for hairs of $\Gamma$), let $P(e) \subset \Gamma$ be a small neighborhood of the midpoint of $e$.
    Choose $P(e)$ such that it is disjoint from other edges or vertices of $\Gamma$. If $e \neq e'$, then $P(e) \cap P(e') = \emptyset$, and $P(e) \cap f(e)$ is an interval.
    
    \item[(2)] Decompose the oriented closed interval $P(e) \cap f(e)$ into three sub-intervals: $f(e) \cap P(e)= [a,b] \cup [b,c] \cup [c,d]$. Remove the middle interval $[b,c]$ and attach a suitably rescaled standard Hopf link $S^j \cup S^{n-j-1}$ such that $S^{j}$ is attached at $b$ and $S^{n-j-1}$ is attached at $c$.
    Here, the standard Hopf link is chosen so that its linking number is 1. 
    
    \item[(3)] Remove the neighborhood of $v \in B(\Gamma)$ and attach a small sphere normal to $\R^j$ (standard normal sphere), which has linking number $1$ with $\R^j$.
\end{enumerate}

This procedure yields a disjoint union $G= G_1 \cup \dots \cup G_{|W(\Gamma)|}$. We call each component $G_i$ a \emph{Y-graph}.
There are two types of Y-graph components, depending on whether the corresponding vertex is Type~$\rI$ or Type~$\rII$.

\begin{figure}[htbp]
    \centering
    \includegraphics[width=0.5\linewidth]{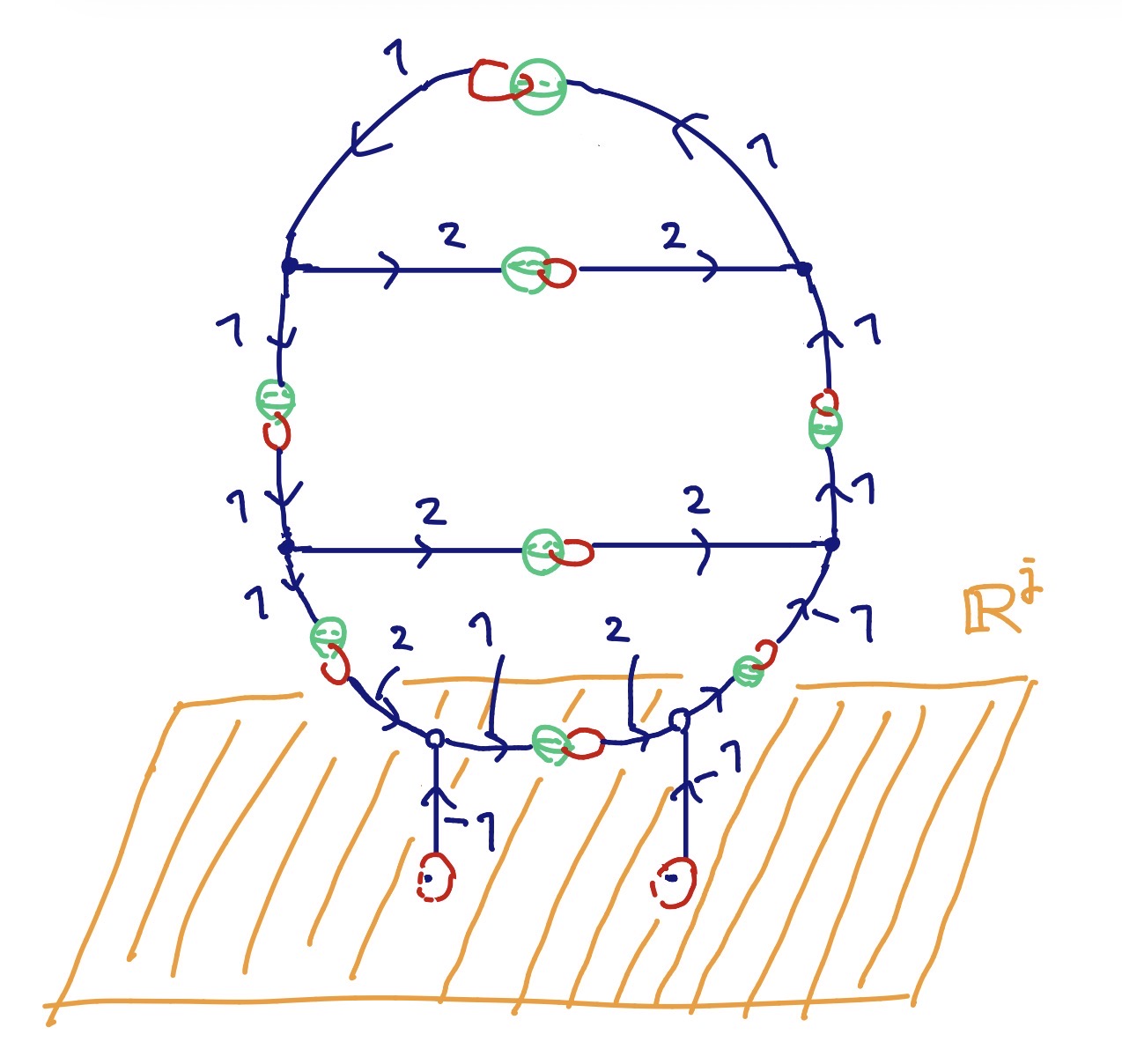}
    \caption{Y-graphs $G$ for $\Gamma= G_{1,3}$}
    \label{fig:Y-graphs_G}
\end{figure}

\begin{figure}[htbp]
    \centering
    \begin{subfigure}{0.48\linewidth}
        \centering
        \includegraphics[width=0.5\linewidth]{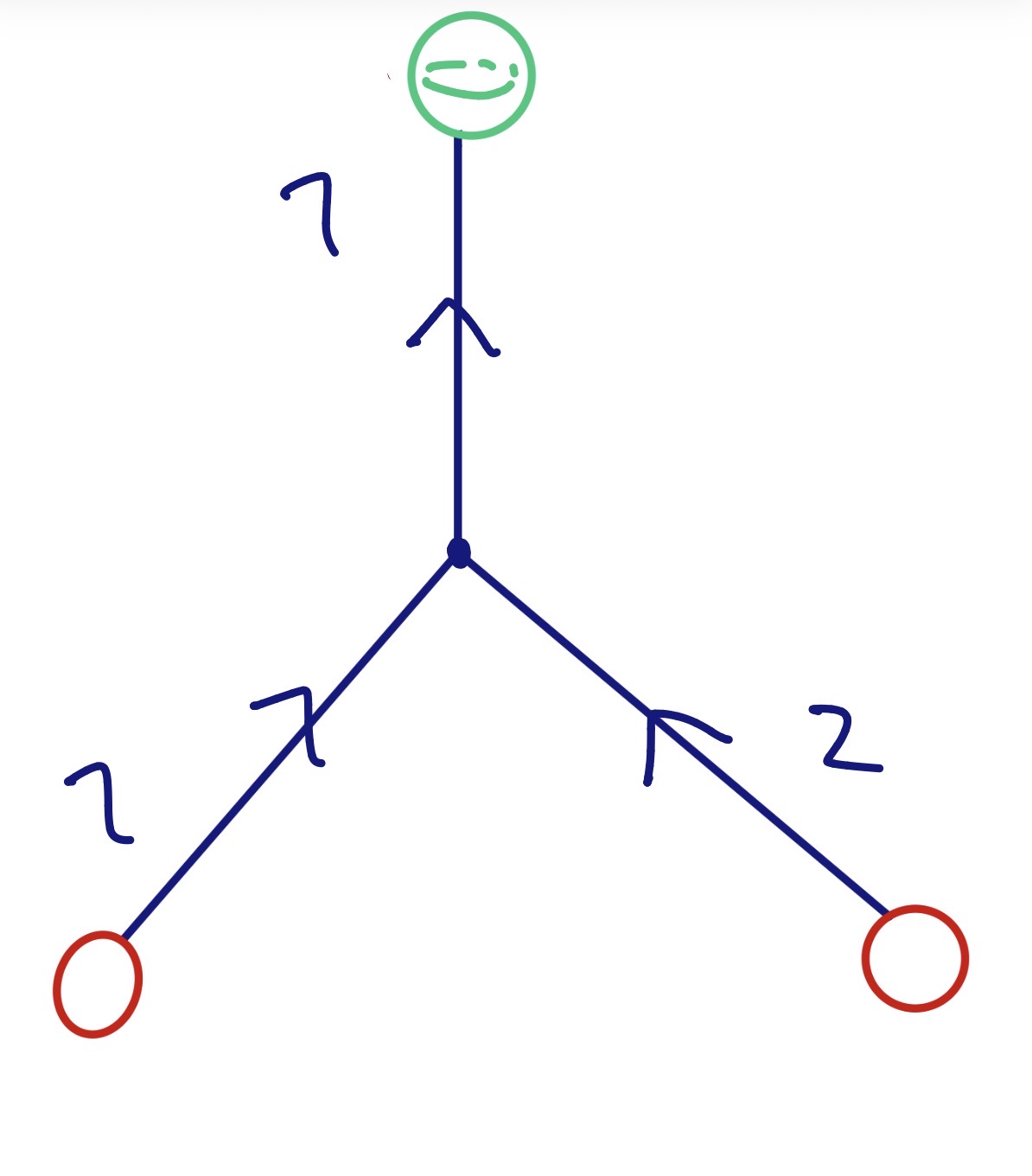}
        \caption{Type $\rI$}
        \label{fig:Type_I_G_i}
    \end{subfigure}
    \hfill
    \begin{subfigure}{0.48\linewidth}
        \centering
        \includegraphics[width=0.5\linewidth]{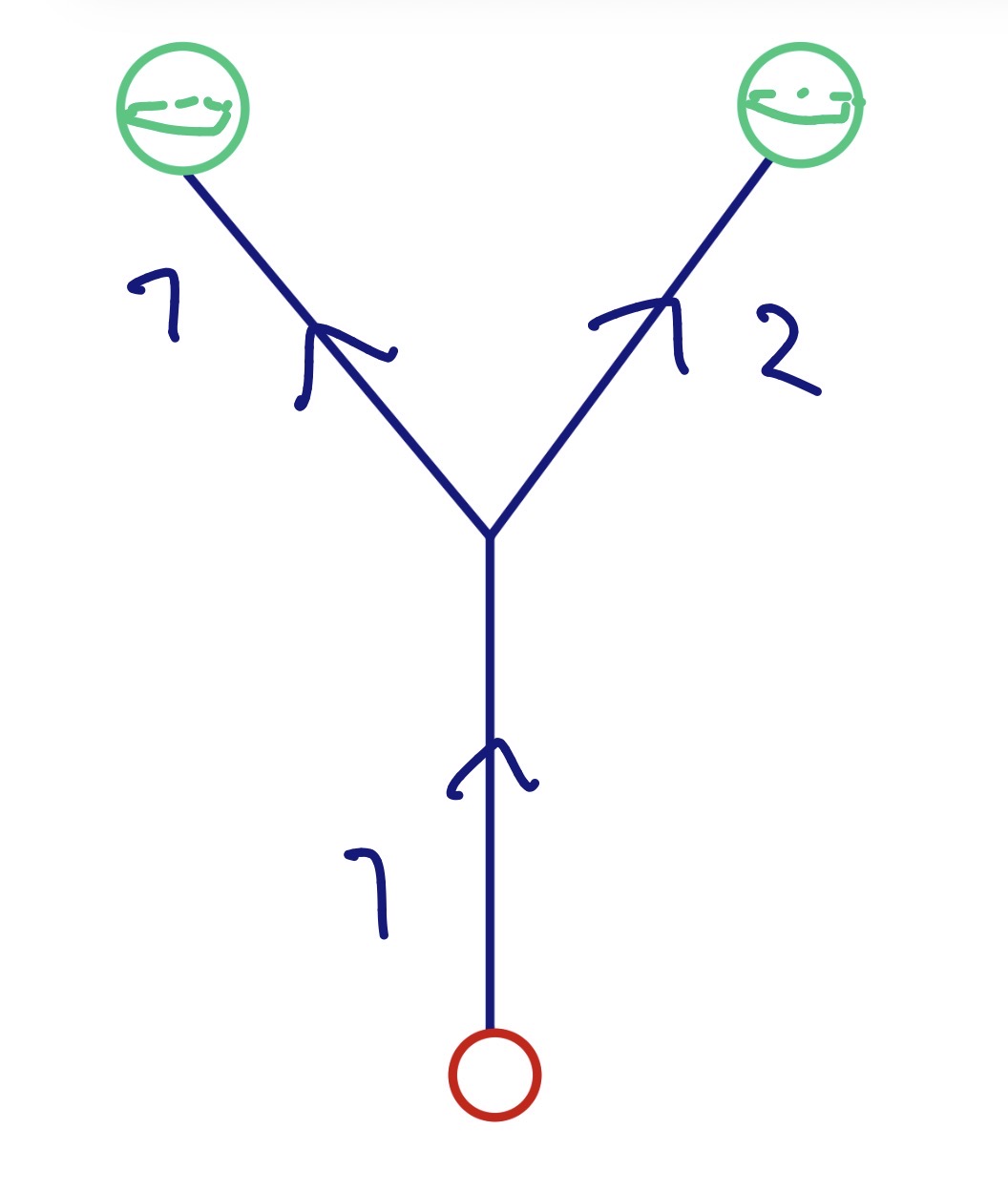}
        \caption{Type $\rII$}
        \label{fig:Type_II_G_i}
    \end{subfigure}
    \caption{Figures of Y-graph $G_i$}
    \label{fig:Y-graph_G_i}
\end{figure}

\end{Def}

Thicken $G_k \subset \R^n$ and denote the resulting manifold by $V_k$.

% --- [Identification] Geometric setup for Type I and II ---
\medskip

\begin{Mysec}[Identification of thickened Y-graphs with standard models]\label{Mysec:IdentificationOfThickenedY-graphswithstandardmodels}    
    If $G_k$ corresponds to a Type~$\rI$ vertex, we identify $V_{k}$ with 
    \[
        I^n \setminus (\nu(I^j) \cup \nu(I^j) \cup \nu(I^{n-j-1}))
    \]
    as described in \S~\ref{sebsec:TypeIribbonPresentation}, in such a way that $b_i$ corresponds to the link associated with the $i$-th incoming half-edge, and $b_1'$ corresponds to the component of the Hopf link associated with the outgoing half-edge.
    
    Similarly, if $G_k$ corresponds to a Type~$\rII$ vertex, we identify $V_{k}$ with 
    \[
        I^n \setminus (\nu(I^j) \cup \nu(I^{n-j-1}) \cup \nu(I^{n-j-1}))
    \]
    as in \S~\ref{subsec:TypeIIribbonPresentation} so that $b_1$ corresponds to the link associated with the incoming half-edge, and $b_i'$ to the component of the Hopf link associated with the $i$-th outgoing half-edge.
\end{Mysec}

Let $V_{k [i]}$ be a closed tubular neighborhood of the union of $G_k$ and the standard spanning disk of the component of the Hopf link associated with the $i$-th incoming half-edge attached to the $k$-th white vertex.
Define the complement $j$-handle $h_{k[i]}^{j} = \overline{ V_{k [i]} \setminus V_{k}}$. 
Similarly, for the $i$-th outgoing edge, define $V^{[i]}_{k}$ as the union of the tubular neighborhood of the standard spanning disk of the component of the Hopf link associated with the $i$-th outgoing edge and $V_{k}$. The $n-j-1$ handle $h_k^{n-j-1[i]}$ is defined analogously.

% --- [Construction] Main inductive construction of the cycle ---
Relabel $W(G_{p,g})$ so that the leftmost white vertex on the top edge is the first. Starting from that vertex, order the remaining vertices following the counter-clockwise direction.
Let $r \le 2(g+p-1)$ and let $B\{ r \} = S_1 \times S_2 \times \dots \times S_{r}$, where
\[
S_i = \begin{cases}
    S^{n-j-2} & (i \text{ is Type }\rI), \\
    S^{j-1} & (i \text{ is Type } \rII).
\end{cases}
\]

By identifying the relevant subspace
\[
V_k \cup h_{k[1]}^{j} \cup h_{k[2]}^{j} \cup h_k^{n-j-1[1]}
\quad \text{or} \quad
V_k \cup h_{k[1]}^{j} \cup h_k^{n-j-1[1]} \cup h_k^{n-j-1[2]},
\]
according as \(k\) is of Type~\(\rI\) or Type~\(\rII\), with the copy of
\(I^3\) in the boxed ribbon presentation appearing in
\ref{Mysec:IdentificationOfThickenedY-graphswithstandardmodels}, we regard
the cycle \(c_T\) associated with the Type~\(T\) ribbon presentation as an element
\[
    c_{T} \in \left[ S_k,  \mathrm{hofib}\left( \Emb_{\partial}(h_{k[1]}^j , V_{k[1] } ) \to \Emb_{\partial}\left(h_{k[1]}^j, V_{k[1]} \cup \textstyle\bigcup_i h_k^{n-j-1[i] }\right) \right) \right],
\]
where the base point is the inclusion
\(h_{k[1]}^j \hookrightarrow V_{k[1]}\).

We now reconstruct $c_{p,g}$ by iterated surgery. 

\begin{Con}[Reconstruction of $c_{p,g}$ by iterated surgery]\label{con:c(p,g)byClasperSurgery}
    Let $i \in \{1, 2, \dots, 2p \}$ and let $H^j_i = h^j_{k(i)[1]}$, where the $i$-th hair is attached to the $k(i)$-th white vertex.

    For $r \in \{ 1,2, \dots, 2g-1 \}$, we inductively define submanifolds $\mathcal{V}_r \subset \widetilde{\mathcal{V}}_r \subset \R^n$ and a map
    \[
        c_{r}\colon B\{ r \} \to \mathrm{hofib}( \Emb_{\partial}(H_1^j, \mathcal{V}_r) \to \Emb_{\partial}(H_1^j, \widetilde {\mathcal{V}}_r))
    \]
    as follows.

    First, for $r=1$, we define $\mathcal{V}_1 = V_{1}$ and $\widetilde{\mathcal{V}}_1 = V_{1}^{[1]}$.
    The map
    \[ 
        c_1\colon B \{ 1 \} \to \mathrm{hofib}( \Emb_{\partial}(H_1^j , \mathcal{V}_1) \to \Emb_{\partial}(H_1^j, \widetilde{ \mathcal{V}_1})) 
    \]
    is defined simply by $c_{\rI}$.

    Next, assume that $\mathcal{V}_r$, $\widetilde{\mathcal{V}}_r$, and $c_r$ have been constructed.
    If the $(r+1)$-st white vertex is of type $T$, we take $c_{T}$ to be a Type~$T$ cycle:
    \[
        c_{T}\colon S_{r+1} \to \mathrm{hofib}\left( \Emb_{\partial}(h^j_{r+1[1]} , V_{{r+1}[1]}) \to \Emb_{\partial}(h^j_{r+1,[1]}, {V}_{r+1[1]}\cup \bigcup_i h_{r+1}^{n-j-1[i]} ) \right).
    \]
    We set $\mathcal{V}_{r+1} = \mathcal{V}_r \cup V_{r+1[1]} \cup V_{r}^{[1]}$ and $\widetilde{ \mathcal{V}}_{r+1} = \widetilde{ \mathcal{V}}_r \cup V_{r+1[1]} \cup ( \bigcup_i V_{r+1}^{[i]} )$, and we define \
    \[
        c_{r+1}\colon B\{ r +1 \} \to \mathrm{hofib}( \Emb_{\partial}(H_1^j, \mathcal{V}_{r+1}) \to \Emb_{\partial}(H_1^j, \widetilde {\mathcal{V}}_{r+1}))
    \]

    as
    
    \[ 
        (c_{r+1})_t = 
        \begin{cases}
            (c_T)_{2t} \circ (c_{r})_0  &(0 \le t \le 1/2) \\
            (c_r)_{2t-1} &(1/2 \le t \le 1 )
        \end{cases} 
    \]

    For the subsequent range $r \in \{ 2g, \dots, 2g+2p-2 \}$, we extend the construction to
    \[
        c_r\colon B \{ r \} \to \mathrm{hofib}( \Emb_{\partial}(H_1^j \cup \dots \cup H^j_{r-2g+2} , \mathcal{V}_r) \to \Emb_{\partial}(H_1^j \cup \dots \cup H^j_{r-2g+2} , \widetilde{ \mathcal{V}}_r)).
    \]
    Specifically, we set $\mathcal{V}_{r} = \mathcal{V}_{r-1} \cup V_{r[1]} \cup V_{r-1}^{[1]} \cup H_{r-2g+2}$ and $\widetilde{ \mathcal{V}}_{r} = \widetilde{ \mathcal{V}}_{r-1} \cup \mathcal{V}_{r} \cup V_{r}^{[1]}$, and define $c_{r+1}$ by
    \[ 
        (c_{r+1})_t(x) = 
        \begin{cases} 
            (c_{r})_t (x) & (x \in H_1 \cup \cdots \cup H_r), \\
            (c_{\rI})_t(x) & (x \in H_{r+1}).
        \end{cases} 
    \]

    Finally, we obtain a long embedding modulo immersion $c_{p,g} \in \overline{\Emb}_c(\R^j, \R^n)$ defined by
    \[ 
        c_{p,g}(x)=
        \begin{cases}
            c_{2g+2p-2}(x) & (x \in H_1 \cup \cdots \cup H_{2p} ), \\
            x & (\text{otherwise}).
        \end{cases}
    \]
\end{Con}

% The constructed cycle coincides with the one constructed in \S~\ref{subsec:Cycles associated with chord diagrams} when the chord diagram is $D(G_{p,g})$, up to the parameterization of the path to the trivial immersion.
Recall that $c_T$ is the cycle obtained from $c_T'$ by a change of base point (see Proposition~\ref{type i belongs to trivial component} and Proposition~\ref{Type ii is in trivial component}).

\begin{Lem}\label{Lem:SameHomologycpgIfReplacecTWithc'T}
    For each white vertex $v$ of $G_{p,g}$, replacing the cycle $c_{T}$ associated with $v$ by $c'_{T}$ does not change the resulting homology class in $\bEmb_c(\R^j, \R^n)$.
\end{Lem}

\begin{proof}
    This follows from the definition of iterated surgery and the fact that $c_{T}$ is obtained from $c_T'$ by a change of base point.
\end{proof}

\begin{Lem} \label{Lem:CoinsidenceOfTwoCycle}
    The cycle constructed in Construction~\ref{con:c(p,g)byClasperSurgery} represents the same homology class as $c_{p,g}$ constructed in \S~\ref{sec:Construction of the cycle}.
\end{Lem}

\medskip

To prove this lemma, we reinterpret Construction~\ref{con:c(p,g)byClasperSurgery} in terms of compositions of embedded boxed ribbon presentations.

% --- [Reinterpretation] Connection to Boxed Ribbon Presentations ---

\begin{Mysec}[Compositions of embedded boxed ribbon presentations]\label{def:BoxedCompositionOfEmbeddedRibbonPresentation}
Let $P$ and $P'$ be boxed ribbon presentations embedded in $\R^3$ that satisfy the following local model.

Consider a Hopf link in $\R^3$:
\begin{align*}
    C\colon & z^2 + x^2=1, y=0 \\
    C'\colon & y^2+(z-1)^2 =1, x=0
\end{align*}
where the orientations of the standard spanning disks are given by $dz \wedge dx$ and $dy \wedge dz$, respectively.
Equip the Hopf link with $0$-framing.

We identify $I^3 \setminus \nu( \{q_1\} \times I \times \{1/2 \} )$ with $\nu(C)$, and $I^3 \setminus \nu(I \times \{p_1 \} \times \{-1/2 \} )$ with $\nu(C')$, in such a way that the normal sphere of $I \times \{ p_1 \} \times \{ -1/2 \}$ (resp.\ the normal sphere of $\{ q_1 \} \times I \times \{ 1/2\}$) corresponds to $C'$ (resp.\ $C$).
The union of the tubular neighborhoods of the spanning disks of $C$ and $C'$ identifies with $I^3$; this is the $I^3$ associated with the composed presentation $P''$.

Take a segment $L_1 \subset P$ and a connected component of disks $\mathcal{D}_1' \subset P'$.
We define the composition of embedded boxed ribbon presentations in $\R^3$, denoted by $P''=P' \circ_{1,1} P$, as follows.

Set $\mathcal{L}'' = (\mathcal{L} \cup \mathcal{L}') \setminus L_1$. We define $\mathcal{B}'' \cup \mathcal{D}''$ by replacing some disks in $\mathcal{B} \cup \mathcal{D}$.

For each disk $D$ in $\mathcal{D}$ intersecting $L_1$, assume that
\[
    D \cap I_{P'}^3
    =
    \{(x,y,z) \in I^3 \mid y=a_D,\ -1 \le z \le -1/2\},
\]
where the constants $a_D$ are distinct for distinct disks $D$.
In other words, $D \cap I_{P'}^3$ is a parallel copy of the base rectangle $D_{1;0}'$.

Choose a thickening
\[
    (\mathcal{B}_1' \cup \mathcal{D}_1') \times [-\delta,\delta] \longrightarrow I^3
\]
of the inclusion $\mathcal{B}_1' \cup \mathcal{D}_1' \subset I^3$ such that
$(\mathcal{B}_1' \cup \mathcal{D}_1') \times \{0\}$ maps to
$\mathcal{B}_1' \cup \mathcal{D}_1'$, and, for each $\alpha \in [-\delta,\delta]$,
\[
    D_{0,1}' \times \{\alpha\}
\]
maps to
\[
    \{(x,y,z) \in I^3 \mid y=\alpha,\ -1 \le z \le -1/2\}.
\]
Replace $D \cap I_{P'}^3$ by the image of
$(\mathcal{B}_1' \cup \mathcal{D}_1') \times \{a_D\}$.
This modification of $\mathcal{B} \cup \mathcal{D}$ yields a union of disks and bands,
which we denote by $(\mathcal{B} \cup \mathcal{D})^{\flat}$.

Then set
\[
    \mathcal{D}'' \cup \mathcal{B}''
    =
    (\mathcal{D} \cup \mathcal{B})^{\flat}
    \cup
    \bigcup_{i \neq 1} (\mathcal{D}'_i \cup \mathcal{B}'_i).
\]

\end{Mysec}

\begin{Exa}
    Here is an example of the local model for the composition of two Type~$\rI$ ribbon presentations.
    At the two crossings labeled "2", the bands are perturbed simultaneously.

    \begin{figure}
        \centering
        \includegraphics[width=0.3\linewidth]{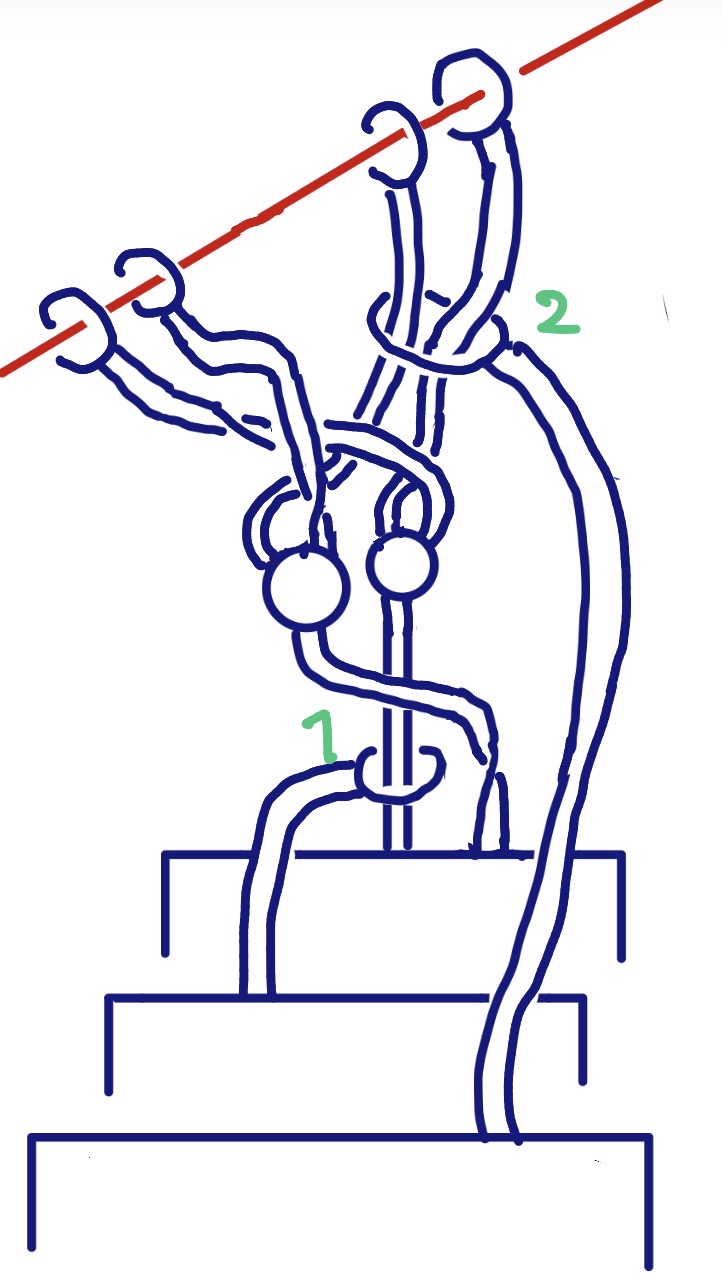}
        \caption{Compositions $c_{\rI}' \circ_{1,1} c_{\rI}'$}
        \label{fig:compositions_of_Type_I}
    \end{figure}
\end{Exa}

% Label名からスペースを除去しました（推奨）
\begin{Mysec}[Reinterpretation of Construction~\ref{con:c(p,g)byClasperSurgery} via embedded boxed ribbon presentations]
\label{mysec:Reinterpretation_of_Construction_Via_Clasper_Surgery}

    Let $P_0$ be a ribbon presentation with $2p$ labeled free disks
    $D_1, \dots, D_{2p}$ and $2p$ bands with base disks.
    Recall that the cycle associated with $P_0$ represents a trivial embedding
    modulo immersion.

    We define Y-graphs
    \[
        \check{G} = \check{G}_1 \cup \dots \cup \check{G}_{2p+g-1} \subset \R^3
    \]
    associated with $G_{p,g}$.
    We arrange them so that
    \[
        f(\Gamma) \cap P_0 = f(B(\Gamma)),
    \]
    and assume that each $B(\Gamma)$ lies in the interior of a distinct free disk
    of $P_0$.
    To obtain $\check{G}$, we replace a neighborhood of each $f(B(\Gamma))$
    by a small sphere normal to the corresponding free disk.
    Moreover, for each edge, we replace the middle interval of the edge by the
    standard Hopf link.

    To reinterpret the surgery construction in terms of ribbon presentations,
    we first assign a local model to each component.
    We identify the thickened component
    $\check{V}_i \subset \R^3$, that is, the thickening of $\check{G}_i$, with
    \[
        I^3 \setminus
        \bigl(\nu(\partial_+ D_{0}) \cup \nu(\partial D'_0) \cup \nu(L)\bigr)
    \]
    if $\check{G}_i$ is of Type~$\rI$, and with
    \[
        I^3 \setminus
        \bigl(\nu(\partial_+ D_{0}) \cup \nu(L) \cup \nu(L')\bigr)
    \]
    if $\check{G}_i$ is of Type~$\rII$.
    Here $\partial_+ D_0$ denotes the closure
    \[
        \overline{\partial D_{0} \setminus \partial I^3}.
    \]
    Thus, for each component $\check{G}_i$ of type $T$, we embed a Type~$T$
    ribbon presentation into $\check{V}_i$.
    For each edge of $G_{p,g}$, we compose the two adjacent ribbon presentations
    using the local model described in
    \S~\ref{def:BoxedCompositionOfEmbeddedRibbonPresentation}.

    We now reinterpret the iterated surgery in
    Construction~\ref{con:c(p,g)byClasperSurgery} as a sequential composition
    of these boxed ribbon presentations.

    First, we perform the composition corresponding to the leftmost hair.
    This corresponds to the initial step of the construction for $r=1$, where
    the map $c_1$ is defined simply by $c_{\rI}$.

    Next, for $r \in \{1, \dots, 2g-1\}$, we perform the compositions
    corresponding to the inductive definition
    \[
        c_{r+1} = (c_T \circ c_r) \ast c_r.
    \]
    Geometrically, this corresponds to composing along the left-hand dashed
    edges, which connect Type~$\rI$ vertices in the counter-clockwise direction
    along the circle of $G_{p,g}$, and then along the edges attached to the
    Type~$\rII$ vertex from bottom to top.

    Finally, for $r \in \{2g, \dots, 2g+2p-2\}$, the construction extends to
    the remaining hairs.
    We perform the compositions along the top dashed edges from right to left.
    This corresponds to the final step of
    Construction~\ref{con:c(p,g)byClasperSurgery}, where the map is extended
    to include the remaining hairs.
\end{Mysec}

\begin{proof}[Proof of Lemma~\ref{Lem:CoinsidenceOfTwoCycle}]
    First, the cycle obtained in Section~\ref{mysec:Reinterpretation_of_Construction_Via_Clasper_Surgery} coincides with the one constructed in Construction~\ref{con:c(p,g)byClasperSurgery}.
    This follows from the discussion in ~\ref{mysec:Reinterpretation_of_Construction_Via_Clasper_Surgery}, combined with Lemma~\ref{Lem:SameHomologycpgIfReplacecTWithc'T} (invariance under base point change).
    
    Consequently, the cycle constructed in \ref{mysec:Reinterpretation_of_Construction_Via_Clasper_Surgery} is exactly the $c_{p,g}$ constructed in \S~\ref{sec:Construction of the cycle}.
\end{proof}

\subsection{Proof that the cycle lies in the image of the Hurewicz map}

\begin{Prop} \label{Thm:HurewiczMap}
    The associated cycle ${c}_{p,g}$ is homotopic to a map that factors through the quotient map 
    \[ q \colon (S^{j-1})^{g-1} \times (S^{n-j-2})^{2p+g-1}  \longrightarrow S^{(n-j-2)(2p+g-1) + (j-1)(g-1)}. \]
\end{Prop}

To prove this proposition, we need the following lemma:

\begin{Lem} \label{Lem:InjectivityOfDegree1Map}
    Let $d, d'$ be nongegative integers and let $\psi\colon S^{d} \times S^{d'} \to S^{d+d'}$ be the quotient map obtained by collapsing the wedge sum $S^{d} \vee S^{d'}$.
    Let $Y$ be a pointed space.
    Then the induced map $\psi^{*}\colon [S^{d+d'}, Y ]_{\ast} \to [S^d \times S^{d'}, Y]_{\ast}$ is injective.
\end{Lem}

\begin{proof}
    The case where $d=0$ or $d'=0$ is immediate.
    Consider the product CW-structure of $S^d \times S^{d'}$.
    Let
    \[
        \phi \colon S^{d+d'-1} \longrightarrow S^d \vee S^{d'}
    \]
    be the attaching map of the $(d+d')$-cell, and let
    \[
        i \colon S^d \vee S^{d'} \longrightarrow S^d \times S^{d'}
    \]
    be the standard inclusion.
    Then we have the following cofibration sequence:
    \[
        S^{d+d'-1}
        \xrightarrow{\phi}
        S^d \vee S^{d'}
        \xrightarrow{i}
        S^d \times S^{d'}
        \xrightarrow{\psi}
        S^d \wedge S^{d'}
        \xrightarrow{\Sigma \phi}
        \Sigma(S^d \vee S^{d'})
        \xrightarrow{\Sigma i}
        \Sigma(S^d \times S^{d'}).
    \]

    Note that $\Sigma i$ admits a left homotopy inverse.
    Indeed, this follows from the suspension splitting
    \[
        \Sigma(S^d \times S^{d'})
        \simeq
        \Sigma S^d \vee \Sigma S^{d'} \vee \Sigma(S^d \wedge S^{d'}),
    \]
    under which $\Sigma i$ identifies with the inclusion of the first two wedge
    summands; see, for example, the proof of
    \cite[Proposition~4I.1]{Hat02}.

    Applying pointed homotopy classes into a pointed space Y, we obtain the
    exact sequence of pointed sets
    \[
        [\Sigma(S^d \times S^{d'}),Y]
        \xrightarrow{(\Sigma i)^*}
        [\Sigma(S^d \vee S^{d'}),Y]
        \xrightarrow{(\Sigma\phi)^*}
        [S^d \wedge S^{d'},Y]
        \xrightarrow{\psi^*}
        [S^d \times S^{d'},Y].
    \]

    Since $\Sigma i$ admits a left homotopy inverse, the induced map
    \[
        (\Sigma i)^*
        \colon
        [\Sigma(S^d \times S^{d'}),Y]
        \longrightarrow
        [\Sigma(S^d \vee S^{d'}),Y]
    \]
    is surjective.
    Hence, by exactness, $(\Sigma\phi)^*$ sends every element to the base point.

    Recall that, by exactness of the Puppe sequence in the sense of pointed
    sets, two elements
    \[
        f,g \in [S^d \wedge S^{d'},Y]
    \]
    have the same image under $\psi^*$ if and only if they lie in the same orbit
    under the natural action of
    \[
        [\Sigma(S^d \vee S^{d'}),Y].
    \]
    This action is induced by $(\Sigma\phi)^*$.
    Since $(\Sigma\phi)^*$ is trivial, this action is trivial.
    Therefore, if $\psi^*(f)=\psi^*(g)$, then $f$ is homotopic to $g$.
    Thus $\psi^*$ is injective.

    This completes the proof.
\end{proof}

\begin{proof}[Proof of Proposition~\ref{Thm:HurewiczMap}]
    We show that $c_r$ factors through a map of degree $1$, and that the factored map 
    \[ \overline{c}_r\colon S_1 \wedge \dots \wedge S_{r} (= S^d) \to \mathrm{hofib}( \Emb_{\partial}(H_1^j, \mathcal{V}_r) \to \Emb_{\partial}(H_1^j, \widetilde {\mathcal{V}}_r)) \]
    lies in the image of the map 
    \[ \mathrm{hofib}( \Emb_{\partial}(H_1^j, \mathcal{V}_r) \to \Emb_{\partial}(H_1^j, \widetilde {\mathcal{V}}_r)) \to \mathrm{hofib}( \Emb_{\partial}(H_1^j, \mathcal{V}_r \cup h_r^{[1]}) \to \Emb_{\partial}(H_1^j, \widetilde {\mathcal{V}}_r)). \]
    
    Let us denote the base points of the spaces by $\ast$.
    % Let us denote Type~I and Type~II cycles by \bar{c}_{I} and \bar{c}_{II}, respectively.
    
    The base case, where the cycle consists of a single primitive component is trivial because the cycle satisfies the Brunnian property by Proposition~\ref{Prop:TypeIisBrunnian} and Proposition~\ref{Prop:TypeIIBrunnianLink}.

    For the inductive step, we explicitly show the case $r \le 2g-2$; the remaining cases are verified similarly.
    Assume that $c_{r-1}$ factors through a degree $1$ map $B\{r-1\} \to S_1 \wedge \dots \wedge S_{r-1} = S^d$, and that the factored map $\overline{c}_{r-1}$ has the Brunnian property with respect to forgetting each $H^j $ and $ h^{n-j-1[1]}_{r} $
    Let $c_{T} \colon S_r \to \mathrm{hofib}( \Emb_{\partial}(h^j_{r[1]} , V_{{r+1}[1]}) \to \Emb_{\partial}(h^j_{r[1]}, {V}_{r[1]}\cup \bigcup_i h_{r}^{[i]} ))$ be the Type~$T$ cycle corresponding to the $r$-th white vertex.

    By the construction of $c_r$, it is defined on the product $S^d \times S_r$. We denote this map by $c_{r}'$.
    If we restrict $c_r'$ to $S^d \times \{\ast \}$, the restriction is trivial because $\overline{c}_{r-1}$ is trivial if we forget $h_{r-1}^{n-j-1[1]}$.
    Similarly, the restriction to $\{\ast\} \times S_r$ is trivial.
    Thus, the map $c_r'$ factors through the quotient $S^{d} \times S_r \to S^d \wedge S_r$. Let $\overline{c}_r$ denote this factored map.
    
    If we forget $h_{r}^{n-j-1[1]}$, then $c_T$ becomes trivial.
    Since $\overline{c}_{r-1}$ is trivial if we forget $h_{r-1}^{n-j-1[1]}$, $c_{r}'$ is trivial if we forget $h_{r}^{n-j-1[1]}$.
    By Lemma~\ref{Lem:InjectivityOfDegree1Map}, the factored map $\overline{c}_r$ also possesses the same property.
\end{proof}

\begin{Thm} \label{Thm:NonTriviality}
    Let $c_{p,g}\colon (S^{j-1})^{g-1} \times (S^{n-j-2})^{2p+g-1} \to \bEmb_c(\R^j, \R^n)$ be the cycle associated with the chord diagram $D(G_{p,g})$.
    Let $\omega\colon H_{\mathrm{top}} (^{*}HGC(k,g)) \to \C$ be an $\mathfrak{sl}_2$-weight system, and let $H = \sum_{\Gamma } \frac{\omega(\Gamma)}{\mathrm{Aut} (\Gamma)}$ be a cocycle on $HGC^{\mathrm{top}}_{n,j}$.
    Then, $\langle \overline{I}(H), [\overline{c}_{p,g}] \rangle = \omega(G_{p,g}) \neq 0$.
    Moreover, $[{c}_{p,g}]$ lies in the image of the Hurewicz map.
\end{Thm}

\begin{proof}
    The first part follows from Theorem~\ref{Thm:NonTrivialityOfCycleAssociatedToGkg}.
    The second part is a consequence of Proposition~\ref{Thm:HurewiczMap}.
\end{proof}

\small
\printbibliography

\end{document}